\documentclass[preprint]{elsarticle}

\usepackage{etex}               

\usepackage[bookmarks=true,colorlinks=true,linkcolor=blue]{hyperref}

\usepackage[margin=1.in]{geometry}

\usepackage{xcolor}
\definecolor{jr@red}{RGB}{228,26,28}
\definecolor{jr@blue}{RGB}{55,126,184}
\definecolor{jr@green}{RGB}{77,175,74}
\definecolor{jr@purple}{RGB}{152,78,163}
\definecolor{jr@orange}{RGB}{255,127,0}
\definecolor{jr@yellow}{RGB}{255,255,51}
\definecolor{jr@brown}{RGB}{166,86,40}
\definecolor{jr@pink}{RGB}{247,129,191}
\definecolor{jr@gray}{RGB}{153,153,153}

\usepackage{amsmath}          
\usepackage{amssymb}          
\usepackage{mathrsfs}         
\usepackage{mathtools}        

\usepackage{soul}             
\usepackage[normalem]{ulem}

\usepackage{caption}          
\usepackage{makecell,multirow,multicol,booktabs} 

\usepackage[group-separator={,}]{siunitx}

\usepackage{tikz}
  \usetikzlibrary{calc,positioning}
  \usetikzlibrary{arrows.meta}
\usepackage{tikz-3dplot}

\usepackage{subcaption}

\usepackage{pgfplots}
  \pgfplotsset{compat=newest}
\usepackage{graphicx}
\usepackage[all]{xy}
\usepackage{lineno}

\usepackage{hyperref}

\usepackage{amsthm}
\theoremstyle{plain}
\newtheorem{lemma}{Lemma}[section]

\newtheorem{proposition}[lemma]{Proposition}
\theoremstyle{definition}

\theoremstyle{remark}

\usepackage[section]{algorithm} 
\usepackage{algpseudocode}  

\makeatletter%
\newcounter{algorithmicH} 
\let\oldalgorithmic\algorithmic
\renewcommand{\algorithmic}{%
  \stepcounter{algorithmicH} 
  \oldalgorithmic} 
\renewcommand{\theHALG@line}{ALG@line.\thealgorithmicH.\arabic{ALG@line}}
\makeatother%

\newcommand{\boldvec}[1]{\ensuremath{\mathbf{#1}}}

\ifdefined\vec
  \renewcommand{\vec}[1]{\boldvec{#1}}
\else
  \newcommand{\vec}[1]{\boldvec{#1}}
\fi

\newcommand{\pp}[2]{\frac{\partial #1}{\partial #2}}

\newcommand{\bs}[1]{\boldsymbol{#1}}

\makeatletter
\newcommand{\editHighlighting}[3]{%
  \if@display%
  \textcolor{red!40!white}{\text{\sout{#1}} \textcolor{#2}{#3}}%
  \else%
  \textcolor{red!40!white}{\sout{#1} \textcolor{#2}{#3}}%
  \fi%
}
\makeatother

\newif\ifshowstatus
\showstatustrue 

\begin{document}

\begin{frontmatter}

\title{Structure-Preserving Transfer of Grad--Shafranov Equilibria to Magnetohydrodynamic Solvers}

  \author[gt]{Rushan Zhang}
  \ead{rzhangbq@gatech.edu}

  \author[lanl]{Golo Wimmer\fnref{lanlThanks}}
  \ead{gwimmer@lanl.gov}

  \author[gt]{Qi Tang\corref{cor1}\fnref{ascrThanks}}
  \ead{qtang@gatech.edu}

  \address[gt]{School of Computational Science and Engineering, Georgia Institute of Technology, Atlanta, GA 30332.}
  \address[lanl]{Theoretical Division, Los Alamos National Laboratory, Los Alamos, NM 87545.}

  \cortext[cor1]{Corresponding authors}
  \fntext[lanlThanks]{Research was partially supported by the Laboratory Directed Research and Development program of Los
    Alamos National Laboratory (LANL), under project number 20240261ER, as well as the U.S. Department of
    Energy Office of Fusion Energy Sciences Base Theory Program, at LANL under contract No.~89233218CNA000001.}
  \fntext[ascrThanks]{Research was partially supported by the U.S.~Department of Energy Advanced Scientific Computing Research program of Mathematical Multifaceted Integrated Capability Center (MMICC).
  }

  \begin{abstract}
    Magnetohydrodynamic (MHD) solvers used to study dynamic plasmas for magnetic confinement fusion typically rely on initial conditions that describe force balance, which are provided by an equilibrium solver based on the Grad--Shafranov (GS) equation. Transferring such equilibria from the GS discretization to the MHD discretization often introduces errors that lead to unwanted perturbations to the equilibria on the level of the MHD discretization. In this work, we identify and analyze sources of such errors in the context of finite element methods, with a focus on the force balance and divergence-free properties of the loaded equilibria. In particular, we {study} three main sources of errors: (1) the improper choice of finite element spaces in the MHD scheme relative to the poloidal flux and toroidal field function spaces in the GS scheme, (2) the misalignment of the meshes from two solvers, and (3) possibly under-resolved strong gradients near the separatrix. With this in mind, we study the impact of different choices of finite element spaces, including those based on compatible finite elements. In addition, we also investigate the impact of mesh misalignment and propose to conduct mesh refinement to resolve the strong gradients near the separatrix. Numerical experiments are conducted to demonstrate equilibria errors arising in the transferred initial conditions. Results show that force balance is best preserved when structure-preserving finite element spaces are used and when the MHD and GS meshes are both aligned and refined. Given that the poloidal flux is often computed in continuous Galerkin spaces, we further demonstrate that projecting the magnetic field into divergence-conforming spaces is optimal for preserving force balance, while projection into curl-conforming spaces, although less optimal for force balance, weakly preserves the divergence-free property. 
  \end{abstract}

  \begin{keyword}
    Compatible Finite Element\sep Grad--Shafranov Equation\sep Magnetohydrodynamics \sep Magnetic Confinement Fusion
  \end{keyword}
\end{frontmatter}



\section{Introduction}
\label{sec:intro}

The magnetohydrodynamic (MHD) equations are widely used for simulating long-time instabilities such as disruptions in magnetic confinement fusion~(MCF) devices \cite{jardin2010computational, sovinec2004nonlinear, hoelzl2021jorek}. Such instabilities may arise from small perturbations to an MHD equilibrium \cite{igochine2015active}, and thereby, clean and accurate equilibrium solutions as initial conditions are essential for stability analysis. In axisymmetric devices, equilibrium solutions are typically obtained by solving the Grad--Shafranov (GS) equation, which expresses the magnetic field in axisymmetric $(r,z)$-coordinates in terms of its poloidal and toroidal components, using the poloidal magnetic flux function $\Psi$ and the toroidal magnetic field profile $f(\Psi)$ \cite{jardin2010computational}. Although axisymmetry allows the field to be represented using 2D scalar functions, time-dependent MHD simulations of inherently 3D instabilities still require a full 3D vector field as the initial condition. In particular, in order to express the Lorentz force term, a discrete representation of the initial magnetic field $\mathbf{B}$ and the corresponding current density vector $\mathbf{J}$ is required. Consequently, converting GS solver outputs for $\Psi$ and $f$ into the full magnetic vector field is necessary.

There are various strategies for solving the GS equation and loading GS-solver-based solutions into transient MHD solvers. The standard approach is to store $\Psi$-data on a Cartesian grid along with a list of 1D arrays for different profiles that are needed to reconstruct full fields. This often follows formats used by the widely used EFIT code \cite{lao1985reconstruction}, in which $\Psi$ is stored in a 2D array and the toroidal field function $f$ and pressure $p$ are documented as arrays depending on normalized $\Psi$. The magnetic field may then be reconstructed on the Cartesian grid using, e.g., finite differences, and can then be interpolated onto the MHD solver grid \cite{bonilla2023fully,jorti2023mimetic,hamilton2025aligning}. This workflow motivated the recent advance of several finite-difference-based GS solvers that use Cartesian grids \cite{freegsgithub, amorisco2024freegsnke}, cut-cells \cite{liu2021parallel} and structured adaptive grids \cite{farmakalides2025cratos}.

While GS solvers are commonly based on finite differences, large-scale MHD solvers for axisymmetric MCF devices are often based on finite elements \cite{sovinec2004nonlinear, hoelzl2021jorek}. {When providing equilibria for these MHD solvers, the GS solver has been developed based on finite elements \cite{heumann2015quasi, peng2020adaptive, serino2024adaptive}.} However, a systematic study on how to transfer solutions between the two finite element solvers is currently missing, and {fixed-boundary equilibrium solvers} are often applied to address this gap. For instance, \cite{howell2014solving} adopts a spectral element basis and uses a spectral element expansion to compute the field values required by the MHD solver, while the JOREK code \cite{hoelzl2021jorek} re-solves the GS equation on the same mesh as the MHD solver and incorporates $\Psi$ directly into the transient MHD evolution after loading from a free-boundary GS solver. {\cite{king2017effect} also pointed out that recomputing the equilibrium significantly improves the equilibrium used by the NIMROD code.} In fact, re-solving has been found to be necessary to achieve a good force balance in the MHD simulation \cite{jepson2016importance}, {indicating that current practices of interpolation may not be sufficient to provide a good initial condition for dynamic MHD solvers. For instance, \cite{jorti2023mimetic} reports a force imbalance in the initial condition and requires additional damping of the force residuals before the MHD evolution.}

In practice, {the quality of the initial state for stability analysis depends on the choice of discretization in GS solvers and the strategy of transferring to a given MHD grid. One immediate consideration is the choice of finite-element spaces. The compatible finite element framework implements a discrete de-Rham complex, in which the gradient, curl, and divergence operators map naturally between successive finite element spaces \cite{cotter2023compatible}. Consequently, compatible finite element spaces provide a structure-preserving representation of the derived equilibrium quantities and avoid additional projections that are inconsistent with the discrete differential operators. Another consideration is the interpolation error between unaligned meshes. For example, \cite{ferraro2016multi} has implied the need for using the same discretization for both the GS and MHD solvers. JOREK's built in GS solver solves the GS equation on the same mesh as the MHD solver and thus avoids interpolation errors \cite{hoelzl2021jorek}. Besides, under-resolved gradients near the separatrix in the original GS solution may also lead to significant errors in the MHD equilibrium, as is widely recognized for numerical applications. These observations motivate the investigation of three potential sources of transfer error:} (1) incompatibilities between the GS and MHD discretizations, (2) interpolation between unaligned computational meshes, and (3) sharp gradients near the separatrix.

In this work, we analyze the above error sources, with particular focus on the first source, which arises from the choice of GS vs MHD discretizations. Our starting point is an MFEM-based free-boundary GS solver \cite{serino2024adaptive}. {To better show the error, our computation is normalized with respect to the field values at the magnetic axis.} Our investigation is carried out within a compatible finite element framework. The compatible finite element method is a mixed finite element method that introduces a sequence of finite element spaces designed to form differential complexes, in which one finite element space is mapped to another by differential operators \cite{boffi2013mixed,arnold2006finite}. Compatible finite elements have been considered extensively for Faraday's law and more generally MHD as a constrained transport method \cite{bochev2001matching,hu2017stable,gawlik2022finite,wimmer2024structure}, ensuring an exact zero-divergence property for the magnetic field (in a weak or a strong sense). Within this framework, given the differential operation occurring in the definition of $\mathbf{B}$ as a function of $\Psi$, we identify sets of compatible finite element spaces such that the GS-to-MHD transfer error is minimized.
Based on this identification, together with meshing related considerations, the main contributions of this work are as follows:
\begin{itemize}
    \item We propose that the numerical GS-to-MHD transfer error is minimized when (1) the finite element spaces used for the magnetic field $\mathbf{B}$ are the spaces that are contained in the same compatible finite element chain as the initial finite element spaces where $\Psi$ and $f$ reside, (2) the mesh in the toroidal cross-section of the transient MHD solver is aligned with the 2D mesh of the GS solver and (3) the mesh is sufficiently refined along the separatrix.
    \item We describe a projection path and a mesh refinement strategy that reflects the above considerations for minimizing the transfer error. In particular, one of our novel projection paths also naturally preserves the magnetic field's divergence-free property. This eliminates the need for an initial divergence cleaning step within the MHD solver. A divergence cleaning step may otherwise lead to additional perturbations to the MHD equilibrium.
    \item We implement the proposed projection path in MFEM \cite{mfem}, demonstrating the importance of both meshing \textit{and} function space considerations necessary to preserve the approximate MHD equilibrium in the MHD solver's initial conditions.
\end{itemize}
{We find that the projection path with a curl-conforming magnetic field preserves the divergence-free property, but results in a large Lorentz force error. Our numerical results suggest that the implicit projection involved in constructing the curl-conforming magnetic field from a CG representation of $\Psi$ may introduce a lower-order error. This indicates that transferring from GS solvers that produce $\Psi$ in a DG space, such as those in \cite{peng2020adaptive,sanchez2019hybridizable}, may offer advantages for simultaneously reducing Lorentz-force errors and preserving the divergence-free property. However, these methods were developed for the fixed-boundary GS problem, and additional work is needed to extend them to the free-boundary setting.} {Beyond the tokamak case studied here, this work suggests that a similar loading algorithm may be needed for stellarator applications when developing a dynamic MHD solver, in order to fully leverage improved equilibria computed using compatible finite elements such as those from MRX~\cite{blickhan2025mrx}.}

The remainder of this paper is structured as follows. In Section \ref{sec:governing-eqns}, we derive our quantities of interest related to transferring data from the GS solver to the MHD solver. In Section \ref{sec:discretization}, we review the compatible finite element method and discuss it in the context of MHD equilibria. In Section \ref{sec:numerical_method}, we introduce our new loading procedure based on compatible finite elements, including a discussion on the magnetic field's divergence-free property and the resulting discrete MHD equilibrium Lorentz force terms. In Sections \ref{sec:implementation} and \ref{sec:results}, we describe our code implementation and present numerical results, respectively. Finally, in Section \ref{sec:conclusion}, we end with a conclusion of our work.


\section{Governing equations}
\label{sec:governing-eqns}

\subsection{Grad--Shafranov equation}
\label{sec:gs-eqns}
In MHD, given a tokamak domain $\Omega$, the non-dimensionalized plasma velocity momentum equation can be written as
\begin{equation}
    \rho \left(\pp{\mathbf{u}}{t} + \mathbf{u} \cdot \nabla \mathbf{u}\right) +  \beta \nabla p + \mathbf{B} \times \mathbf{J} = \mathbf{0},\label{momentum}
\end{equation}
for density $\rho$, flow velocity $\mathbf{u}$, pressure $p$, { magnetic field $\mathbf{B}$, and current density $\mathbf{J} = \nabla \times \mathbf{B}$. The internal-to-magnetic energy ratio $\beta =  \mu p_0 / B_0^2 \ll 1$ is a non-dimensional parameter depending on the vacuum magnetic permeability $\mu_0$ ,  the pressure $p_0$, and magnetic field magnitude $B_0$. In this work, we consider $p_0$ and $B_0$ to be their corresponding magnitudes at the magnetic axis and a system length scale of $L_0 = 1 \mathrm{m}$.}

Typically, in tokamaks, the plasma is considered to be in an axisymmetric equilibrium state where the pressure gradient and Lorentz force are in balance, given by
\begin{equation}
    \beta \nabla p + \mathbf{B} \times \mathbf{J} = \mathbf{0}. \label{equilibrium}
\end{equation}
{Besides the force balance, the magnetic field $\mathbf{B}$ is also required to be divergence-free,}
\begin{equation}
    {\nabla \cdot \mathbf{B} = 0. \label{div_B}}
\end{equation}

To construct a pressure field $p$ and magnetic field $\mathbf{B}$ that satisfy \eqref{equilibrium}, we switch to cylindrical coordinates $(r, \phi, z)$ and drop any dependence on the $\phi$-coordinate in $p$ and $\mathbf{B}$. Next, we express the magnetic field through the magnetic vector potential $\mathbf{A} = (A_r, A_\phi, A_z)$, such that $\mathbf{B} = \nabla_c \times \mathbf{A}$, for cylindrical coordinate curl
\begin{align}
    \begin{split}
        \nabla_c \times \mathbf{A} & = \left(\frac{1}{r} \pp{A_z}{\phi} - \pp{A_\phi}{z}\right) \mathbf{e}_r + \left(\pp{A_r}{z}  - \pp{A_z}{r} \right) \mathbf{e}_\phi + \frac{1}{r} \left(\pp{(r A_\phi)}{r} - \pp{A_r}{\phi}\right) \mathbf{e}_z,
    \end{split}
\end{align}
where $\mathbf{e}_r$, $\mathbf{e}_\phi$, $\mathbf{e}_z$ denote the unit vectors in cylindrical coordinates. Using axisymmetry and the definition of the curl in axisymmetric coordinates, we then have
\begin{equation}
    \mathbf{B} = \nabla_c \times (A_\phi \mathbf{e}_\phi) + \left(\pp{A_r}{z}-\pp{A_z}{r}\right)\mathbf{e}_\phi. \label{B_from_A}
\end{equation}
Next, we transfer our axisymmetric equations to the 2D $(r, z)$-plane for which we use the natural embedding of axisymmetric 3D fields in $(r,\phi,z)$-space to 2D fields in $(r, z)$-space
\begin{equation}
    \sim_\phi \colon A(r, \phi, z) \rightarrow A(r, z),
\end{equation}
and define
\begin{equation}
    A_\phi \; \sim_\phi \; \frac{1}{r}\Psi(r, z), \qquad \pp{A_r}{z}-\pp{A_z}{r} \; \sim_\phi \frac{1}{r}f(r, z),\label{Psi_f_from_A}
\end{equation}
for fields $\Psi$ and $f$ defined in the 2D $(r, z)$-plane. In view of our finite-element-based discretizations to be introduced below, we further define the 2D planar gradient and perpendicular operators according to
\begin{equation}
    \nabla = \left(\pp{}{r}, \pp{}{z}\right)^T, \qquad (A_r, A_z)^\perp = (-A_z, A_r), \qquad \nabla^\perp = \left(-\pp{}{z}, \pp{}{r}\right)^T, \label{2D_gradient}
\end{equation}
which we can also use to rewrite the terms in \eqref{B_from_A} as
\begin{equation}
    \nabla_c \times (A_\phi \mathbf{e}_\phi) \; \sim_\phi \; \frac{1}{r} \nabla^\perp \Psi(r, z), \qquad \left(\pp{A_r}{z}-\pp{A_z}{r}\right)\mathbf{e}_\phi \; \sim_\phi \frac{1}{r} f(r, z). \label{B_from_psi_f}
\end{equation}
$\Psi$ is called the \textit{poloidal flux function}, and $\frac{1}{r} \nabla^\perp \Psi(r, z)$ corresponds to the 2D vector component of $\mathbf{B}$ embedded in the \textit{poloidal} $(r, z)$-plane. In contrast, $f$ is called the \textit{toroidal field function}, and the scalar field $f/r$ corresponds to the out-of-$(r, z)$-plane component of $\mathbf{B}$ aligned with the \textit{toroidal} $\phi$-coordinate. Assuming the embedded 2D planar pressure and toroidal field function to be functions of $\Psi$, i.e., $p = p\big(\Psi(r, z)\big)$ and $f = f\big(\Psi(r, z)\big)$, the equilibrium equation \eqref{equilibrium} can then be written purely as an equation in $\Psi(r, z)$, which is called the \textit{Grad--Shafranov (GS)} equation. {The starting point of this work is solutions to the following free-boundary GS equation, formulated as a PDE-constrained optimization problem:
\paragraph{Free-boundary Grad--Shafranov problem}
Determine the poloidal magnetic flux, $\Psi$, and solenoid and coil currents, $I_i$,
  that minimizes a plasma shape at control points, subject to the PDE constraints
  \begin{align}
  \begin{split}
    - \frac{1}{\mu r} \Delta^* \Psi
    &= \begin{cases}
         r p'(\Psi) + \frac{1}{\mu r} f(\Psi) f'(\Psi), & \text{\rm in } \Omega_{\rm p}(\Psi), \\
         I_i / | \Omega_{c_i} |, & \text{\rm in } \Omega_{c_i}, \\
         0, & \text{\rm elsewhere in } \Omega_{\infty}
       \end{cases} \\
    \Psi(0, z) &= 0, \\
    \lim_{\| (r, z)\| \rightarrow +\infty} \Psi(r, z) &= 0 \label{GS_problem}
    \end{split}
  \end{align}
 where $\Delta^* \Psi:= r \partial_r \left(\frac{1}{r} \partial_r \Psi\right) + \partial_z^2 \Psi$.
For further details, see the \cite{jardin2010computational} or our previous work \cite{serino2024adaptive}. 
In the current work, we do not aim to solve the GS equation itself. 
Rather, given a GS solution, \emph{our goal is to determine the best way to transfer it into initial conditions for a dynamic MHD simulation.}
  }

The {GS} equation is a PDE in $(r, z)$-space, and its domain is an extension of the vertical cross-section (i.e., poloidal plane) of a tokamak. It can be solved using classical PDE discretization techniques such as the finite element method, returning $\Psi(r, z)$ on a 2D computational mesh, as well as $f$ and $p$ as functions of $\Psi$ on discrete points $\{\Psi_i\}_{i=0}^m$ along a 1D interval $[\Psi_{\min}, \Psi_{\max}]$. In practice, $f$ and $p$ can then be expressed directly on the 2D computational mesh as $f\big(P_{1D}(\Psi(r, z))\big)$ and $p\big(P_{1D}(\Psi(r, z))\big)$, where $P_{1D}$ is an interpolator acting on $\{\Psi_i\}_{i=0}^m$.

In addition, we note that in view of the discussion below, the split of $\mathbf{B}$ into a $\Psi$- and an $f$-dependent part as in \eqref{B_from_psi_f} will be useful in the process of loading data from the 2D GS mesh into the 3D MHD mesh. The expression $\nabla_c \times \left(A_\phi \mathbf{e}_\phi\right)$ is a vector that lies purely within the $(r, z)$-plane, while $\left(\pp{A_r}{z}-\pp{A_z}{r}\right)\mathbf{e}_\phi$ points along the $\phi$-direction.

\subsection{2D component-wise conversion of the magnetic potential to MHD quantities} \label{sec:2D_components}
For axisymmetric equilibria considered in this work, as described above, vector fields can be decomposed into a 2D vector component within the $(r, z)$ poloidal plane, denoted here by the subscript ``$p$'', and a 2D out-of-plane toroidal scalar component, denoted by the subscript ``$t$'', which is also a function in $(r, z)$-coordinates.  Using these relations, we can express the embedded poloidal and the toroidal components of the magnetic field $\mathbf{B}$ within the 2D $(r, z)$-plane as
\begin{subequations}
    \label{B_from_GS}
    \begin{align}
        \mathbf{B}_p(r, z) & \coloneqq \frac{1}{r} \nabla^\perp \Psi, \label{Bp_from_psi} \\
        B_t(r, z)          & \coloneqq \frac{f}{r}. \label{Bt_from_f}
    \end{align}
\end{subequations}
Note that \eqref{B_from_GS} implies that the magnetic field is divergence-free, since in the axisymmetric case the $\phi$-component $B_\phi \; \sim_\phi \; f/r$ of $\mathbf{B}$ is independent of the $\phi$-coordinate, and therefore
\begin{align}
    \begin{split}
        \nabla_c \cdot \mathbf{B} & = \frac{1}{r} \pp{\big(r (B_p)_r\big)}{r} + \pp{(B_p)_z}{z}
        \; \sim_\phi \; \frac{1}{r}\nabla \cdot \left(r \mathbf{B}_p\right)
        = \nabla \cdot \nabla^\perp \Psi
        = 0,
        \label{div_B_p}
    \end{split}
\end{align}
where $\nabla_c \cdot (\cdot)$ denotes the divergence in cylindrical coordinates, and where we substituted for the definition of \eqref{Bp_from_psi}. We will revisit the divergence of the magnetic field in our discretizations below, and for this purpose define $D_b$ in the 2D $(r, z)$-plane as
\begin{equation}
    D_b = \frac{1}{r}\nabla \cdot \left(r \mathbf{B}_p\right).
    \label{div_B_p_def}
\end{equation}

Next to considering the magnetic field's divergence computation, we will also consider the computation of the current density $\mathbf{J} = \nabla \times \mathbf{B}$, which is a good indicator of the accuracy of our procedure for loading from GS data to MHD data. In particular, if the discrete current density as an MHD function is noisy, then we can expect the same to hold true for the Lorentz force term $\mathbf{B} \times \mathbf{J}$, which appears in the MHD equilibrium. Being an axisymmetric vector field, we again decompose $\mathbf{J}$ into a poloidal vector and toroidal scalar component, which can be easily computed in an analogous fashion to \eqref{B_from_A}. This leads to axisymmetric fields that can be expressed in 2D $(r, z)$-coordinates as
\begin{subequations}
    \label{J_from_B}
    \begin{align}
        \mathbf{J}_p & = \frac{1}{r}\nabla^\perp (rB_t),   \\
        J_t          & = -\nabla^\perp \cdot \mathbf{B}_p.
    \end{align}
\end{subequations}
Note that the current density's components can also be expressed directly as functions of $\Psi$ and $f$ given \eqref{B_from_GS}, which are
\begin{subequations}
    \label{J_direct_nondiscr}
    \begin{align}
        \mathbf{J}_p & = \frac{1}{r}\nabla^\perp f,                          \\
        J_t          & = -\nabla \cdot \left(\frac{1}{r} \nabla \Psi\right).
    \end{align}
\end{subequations}

{When comparing results of mapping from GS to MHD fields, we will consider the force balance equation~\eqref{equilibrium} as a validation metric. Although force balancing \eqref{equilibrium}  is fully achieve in the GS solver,}
our previous work finds that such a force balance equilibrium is difficult to achieve in transient MHD codes \cite{jorti2023mimetic}, which motivates the current work.

{For the Lorentz force term in~\eqref{equilibrium},} given the component-wise magnetic field and its corresponding current density, the poloidal and toroidal components of the Lorentz force are given in the 2D $(r, z)$-plane by
\begin{subequations}    \label{eq:lorentz-force-from-BJ}
    \begin{align}
        [\mathbf{B}\times\mathbf{J}]_p & = - B_t \mathbf{J}_p^\perp + J_t \mathbf{B}_p^\perp, \\
        [\mathbf{B}\times\mathbf{J}]_t & = \mathbf{B}_p \cdot \mathbf{J}_p^\perp.
    \end{align}
\end{subequations}

Note that the Lorentz force's components can also be expressed directly as functions of $\mathbf{B}_p$ and $B_t$ given~\eqref{J_from_B}, which leads to
\begin{subequations}
    \label{eq:lorentz-force-from-B}
    \begin{align}
        [\mathbf{B}\times\mathbf{J}]_p & = \frac{1}{r} B_t \nabla (rB_t) - (\nabla^\perp\cdot \mathbf{B}_p) \mathbf{B}_p^\perp, \\
        [\mathbf{B}\times\mathbf{J}]_t & = -\frac{1}{r}\left(\mathbf{B}_p \cdot \nabla(rB_t)\right).
    \end{align}
\end{subequations}
A more detailed derivation for the component-wise magnetic field as a function of the magnetic vector potential, as well as the Lorentz force reconstruction is given in~\ref{apx:diff_ops_2D}.

{In summary, the starting point of our GS-to-MHD transferral task are the GS fields, given by $\Psi$, $f$, and $p$, solved from \eqref{GS_problem}. We aim to compute on a separate grid, which would be used for MHD simulations, the poloidal and toroidal components $\mathbf{B}_p$, $B_t$  via \eqref{B_from_GS}, and the pressure field. We further investigate diagnostics of $\mathbf{B}$ including the current density's components $\mathbf{J}_p$, $J_t$ via \eqref{J_direct_nondiscr}, the divergence of the poloidal component of the magnetic field $D_b$ from \eqref{div_B_p_def}, and the components of the Lorentz force term $[\mathbf{B}\times\mathbf{J}]$ from \eqref{eq:lorentz-force-from-BJ} and \eqref{eq:lorentz-force-from-B}.}

\section{Finite element discretization}
\label{sec:discretization}
When mapping GS fields to MHD fields, we consider a suitable discrete representation of the magnetic field components \eqref{B_from_GS}. Ideally, this discrete representation should satisfy the discrete divergence-free condition \eqref{div_B_p} and ensure that the discrete toroidal and poloidal current density components, and consequently the Lorentz force, exhibit low levels of noise. From the GS solver, we assume 2D fields given in discrete finite element spaces
\begin{equation}
    \Psi \in \mathbb{V}_\Psi (\mathcal{T}_{GS})_k, \hspace{1cm} f \in \mathbb{V}_f (\mathcal{T}_{GS})_l,
\end{equation}
where $\mathcal{T}_{GS}$ is the mesh used in the GS solver, and $k$, $l$ denote the spaces' polynomial degrees. Analogously, we assume
\begin{equation}
    \mathbf{B}_p \in \mathbb{V}_{\mathbf{B}_p} (\mathcal{T}_{2D})_m, \hspace{1cm} B_t \in \mathbb{V}_{B_t} (\mathcal{T}_{2D})_n,
\end{equation}
and
\begin{equation}
    \mathbf{J}_p \in \mathbb{V}_{\mathbf{J}_p} (\mathcal{T}_{2D})_m, \hspace{1cm} J_t \in \mathbb{V}_{J_t} (\mathcal{T}_{2D})_n,
\end{equation}
where $\mathcal{T}_{2D}$ denotes the 2D poloidal restriction of the mesh used in the MHD solver -- noting that we assume axisymmetric meshes for MHD -- and $m$, $n$ denote the spaces' polynomial degrees.

\subsection{Compatible finite element spaces}
\begin{figure}[h]
    \centering
    \begin{tikzpicture}
        \tikzset{
        dot/.style={circle, fill=black, inner sep=1.5pt},
        vec/.style={-{Latex[length=2mm]}, thick},
        box/.style={draw=black, thick}
        }

        \begin{scope}
            \fill[gray!30] (0,0) rectangle (2,2);
            \draw[box] (0,0) rectangle (2,2);
            \foreach \x in {0,2} {
                    \foreach \y in {0,2} {
                            \ifnum\x=1\relax
                                \ifnum\y=1\relax
                                    \node[dot] at (\x,\y) {};
                                \fi
                            \else
                                \ifnum\y=1\relax \node[dot] at (\x,\y) {}; \fi
                            \fi
                        }
                }
            \foreach \x in {0,2} {
                    \foreach \y in {0,2} {
                            \node[dot] at (\x,\y) {};
                        }
                }
            \node[below=0.3cm] at (1,0) {$CG(\mathcal{T}_{2D})_1$};
        \end{scope}

        \begin{scope}[xshift=3.5cm]
            \fill[gray!30] (0,0) rectangle (2,2);
            \draw[box] (0,0) rectangle (2,2);

            \foreach \x/\y in {1/0, 2/1, 1/2, 0/1} {
                    \node[dot] at (\x,\y) {};
                }

            \foreach \x/\y/\dx/\dy in {
                    1/0/0.3/0, 1/0/-0.3/0,   
                    2/1/0/0.3, 2/1/0/-0.3,   
                    1/2/0.3/0, 1/2/-0.3/0,   
                    0/1/0/0.3, 0/1/0/-0.3    
                } {
                    \draw[vec] (\x,\y) -- ++(\dx,\dy);
                }
            \node[below=0.3cm] at (1,0) {$H(\text{curl}, \mathcal{T}_{2D})_1$};
        \end{scope}

        \begin{scope}[xshift=7cm]
            \fill[gray!30] (0,0) rectangle (2,2);
            \draw[box] (0,0) rectangle (2,2);

            \foreach \x/\y in {1/0, 2/1, 1/2, 0/1} {
                    \node[dot] at (\x,\y) {};
                }

            \foreach \x/\y/\dx/\dy in {
                    1/0/0/0.3, 1/0/0/-0.3,   
                    2/1/-0.3/0, 2/1/0.3/0,   
                    1/2/0/-0.3, 1/2/0/0.3,   
                    0/1/0.3/0, 0/1/-0.3/0    
                } {
                    \draw[vec] (\x,\y) -- ++(\dx,\dy);
                }
            \node[below=0.3cm] at (1,0) {$H(\text{div}, \mathcal{T}_{2D})_1$};
        \end{scope}

        \begin{scope}[xshift=10.5cm]
            \fill[gray!30] (0,0) rectangle (2,2);
            \draw[box] (0,0) rectangle (2,2);
            \node[dot] at (1,1) {};
            \node[below=0.3cm] at (1,0) {$DG(\mathcal{T}_{2D})_0$};
        \end{scope}

        \begin{scope}[yshift=-3.5cm]
            \fill[gray!30] (0,0) rectangle (2,2);
            \draw[box] (0,0) rectangle (2,2);
            \foreach \x in {0,1,2} {
                    \foreach \y in {0,1,2} {
                            \ifnum\x=1\relax
                                \ifnum\y=1\relax
                                    \node[dot] at (\x,\y) {};
                                \fi
                            \else
                                \ifnum\y=1\relax \node[dot] at (\x,\y) {}; \fi
                            \fi
                        }
                }
            \foreach \x in {0,1,2} {
                    \foreach \y in {0,1,2} {
                            \node[dot] at (\x,\y) {};
                        }
                }
            \node[below=0.3cm] at (1,0) {$CG(\mathcal{T}_{2D})_2$};
        \end{scope}

        \begin{scope}[xshift=3.5cm, yshift=-3.5cm]
            \fill[gray!30] (0,0) rectangle (2,2);
            \draw[box] (0,0) rectangle (2,2);

            \foreach \x/\y in {
                    0.67/0, 1.33/0,    
                    2/0.67, 2/1.33,    
                    1.33/2, 0.67/2,    
                    0/1.33, 0/0.67     
                } {
                    \node[dot] at (\x,\y) {};
                }

            \foreach \x/\y/\dx/\dy in {
                    0.67/0/0.3/0, 0.67/0/-0.3/0,   
                    1.33/0/0.3/0, 1.33/0/-0.3/0,   
                    2/0.67/0/0.3, 2/0.67/0/-0.3,   
                    2/1.33/0/0.3, 2/1.33/0/-0.3,   
                    1.33/2/0.3/0, 1.33/2/-0.3/0,   
                    0.67/2/0.3/0, 0.67/2/-0.3/0,   
                    0/1.33/0/0.3, 0/1.33/0/-0.3,   
                    0/0.67/0/0.3, 0/0.67/0/-0.3    
                } {
                    \draw[vec] (\x,\y) -- ++(\dx,\dy);
                }

            \node[dot] at (0.67,1) {};
            \node[dot] at (1.33,1) {};
            \node[dot] at (1,0.67) {};
            \node[dot] at (1,1.33) {};
            \draw[vec] (0.67,1) -- ++(0.3,0);   
            \draw[vec] (0.67,1) -- ++(-0.3,0);  
            \draw[vec] (1.33,1) -- ++(0.3,0);   
            \draw[vec] (1.33,1) -- ++(-0.3,0);  
            \draw[vec] (1,0.67) -- ++(0,0.3);   
            \draw[vec] (1,0.67) -- ++(0,-0.3);  
            \draw[vec] (1,1.33) -- ++(0,0.3);   
            \draw[vec] (1,1.33) -- ++(0,-0.3);  
            \node[below=0.3cm] at (1,0) {$H(\text{curl}, \mathcal{T}_{2D})_2$};
        \end{scope}

        \begin{scope}[xshift=7cm, yshift=-3.5cm]
            \fill[gray!30] (0,0) rectangle (2,2);
            \draw[box] (0,0) rectangle (2,2);

            \foreach \x/\y in {
                    0.67/0, 1.33/0,    
                    2/0.67, 2/1.33,    
                    1.33/2, 0.67/2,    
                    0/1.33, 0/0.67     
                } {
                    \node[dot] at (\x,\y) {};
                }

            \foreach \x/\y/\dx/\dy in {
                    0.67/0/0/0.3, 0.67/0/0/-0.3,   
                    1.33/0/0/0.3, 1.33/0/0/-0.3,   
                    2/0.67/-0.3/0, 2/0.67/0.3/0,   
                    2/1.33/-0.3/0, 2/1.33/0.3/0,   
                    1.33/2/0/-0.3, 1.33/2/0/0.3,   
                    0.67/2/0/-0.3, 0.67/2/0/0.3,   
                    0/1.33/0.3/0, 0/1.33/-0.3/0,   
                    0/0.67/0.3/0, 0/0.67/-0.3/0    
                } {
                    \draw[vec] (\x,\y) -- ++(\dx,\dy);
                }
            \node[dot] at (0.67,1) {};
            \node[dot] at (1.33,1) {};
            \node[dot] at (1,0.67) {};
            \node[dot] at (1,1.33) {};
            \draw[vec] (0.67,1) -- ++(0,0.3);   
            \draw[vec] (0.67,1) -- ++(0,-0.3);  
            \draw[vec] (1.33,1) -- ++(0,0.3);   
            \draw[vec] (1.33,1) -- ++(0,-0.3);  
            \draw[vec] (1,0.67) -- ++(0.3,0);   
            \draw[vec] (1,0.67) -- ++(-0.3,0);  
            \draw[vec] (1,1.33) -- ++(0.3,0);   
            \draw[vec] (1,1.33) -- ++(-0.3,0);  
            \node[below=0.3cm] at (1,0) {$H(\text{div}, \mathcal{T}_{2D})_2$};
        \end{scope}

        \begin{scope}[xshift=10.5cm, yshift=-3.5cm]
            \fill[gray!30] (0,0) rectangle (2,2);
            \draw[box] (0,0) rectangle (2,2);

            \foreach \x in {0.7,1.3} {
                    \foreach \y in {0.7,1.3} {
                            \node[dot] at (\x,\y) {};
                        }
                }
            \node[below=0.3cm] at (1,0) {$DG(\mathcal{T}_{2D})_1$};
        \end{scope}

    \end{tikzpicture}
    \caption{First-order and second-order continuous Galerkin spaces on a 2D quadrilateral element, along with their corresponding curl- and divergence-conforming spaces, and discontinuous Galerkin spaces.}
    \label{fig:2d_fem_spaces}
\end{figure}

Here, we introduce compatible finite element spaces to be used for the magnetic field and current density fields. {In the next section, we will show how they fit naturally into the context of transferring data from  $\Psi$ and $f$ to $\mathbf{B}_p$ and $B_t$.} Using 2D quadrilateral elements as an example, Figure~\ref{fig:2d_fem_spaces} shows the lowest two orders of 2D compatible finite element spaces. The top-left shows the first-order CG space $CG(\mathcal{T}_{2D})_1$, where the basis functions are defined on vertices. The gradient of a scalar defined in this CG space results in vectors naturally lying on cell edges, which corresponds to the curl-conforming space $H(\text{curl}, \mathcal{T}_{2D})_1$, while applying the gradient with a $90^\circ$ rotation -- that is $\nabla^\perp$ defined in \eqref{2D_gradient} -- results in vectors lying perpendicular to the edges, which corresponds to the divergence-conforming space $H(\text{div}, \mathcal{T}_{2D})_1$. Finally, applying the divergence to vector fields defined in the $H(\text{div}, \mathcal{T}_{2D})_1$ space, or alternatively the divergence with a $90^\circ$ rotation to vector fields defined in the $H(\text{curl}, \mathcal{T}_{2D})_1$ space, results in scalars lying in the zeroth-order DG space $DG(\mathcal{T}_{2D})_0$. The second row shows the second-lowest order of 2D compatible finite element spaces, with similar relationships with respect to the differential operators. The curl- and divergence-conforming spaces considered in this work are given by the Raviart-Thomas spaces $RT^e_m$, $RT^f_m$ for triangles and $RTc^e_m$, $RTc^f_m$ for quadrilaterals (as depicted in Figure \ref{fig:2d_fem_spaces}).

Altogether, the 2D compatible finite element spaces can be seen in the context of a discrete de-Rham complex. That is, one space is mapped to another through differential operators, analogously to their non-discrete counterparts. The curl- and divergence-conforming 2D de-Rham complexes are given by
\begin{displaymath}
    \xymatrix{
        H^1(\Omega) \ar[r]^{\nabla} \ar[d]^{\pi^0} & H(\text{{\normalfont curl}};\Omega) \ar[r]^{\nabla^\perp \cdot} \ar[d]^{\pi^1} & L^2(\Omega) \ar[d]^{\pi^2} \\
        \mathbb{V}_0(\Omega) \ar[r]^{\nabla}  &\mathbb{V}_1(\Omega) \ar[r]^{\nabla^\perp \cdot} & \mathbb{V}_2(\Omega)} \hspace{1cm}
    \xymatrix{
        H^1(\Omega) \ar[r]^{\nabla^\perp} \ar[d]^{\pi^0} & H(\text{{\normalfont div}};\Omega)\ar[d]^{\pi^1} \ar[r]^{\nabla \cdot} & L^2(\Omega) \ar[d]^{\pi^2} \\
        \mathbb{V}_0(\Omega) \ar[r]^{\nabla^\perp}  & \mathbb{V}_1(\Omega) \ar[r]^{\nabla \cdot}  & \mathbb{V}_2(\Omega)},
\end{displaymath}
respectively, for a 2D domain $\Omega$, and where $\pi^i$ denote projections from the Sobolev spaces to their discrete counterparts. $H^1(\Omega)$ corresponds to the $CG(\mathcal{T}_{2D})_m$ space, $H(\text{curl};\Omega)$ to the $H(\text{curl}, \mathcal{T}_{2D})_m$ space, $H(\text{div};\Omega)$ to the $H(\text{div}, \mathcal{T}_{2D})_m$ space and $L^2(\Omega)$ to the $DG(\mathcal{T}_{2D})_{m-1}$ space. {Importantly, these complexes are based on the Cartesian differential operators $\nabla$ and $\nabla^\perp$ in the $(r,z)$-plane that are interpreted as a Cartesian space with unit vectors $\mathbf{e}_r=(1,0)$ and $\mathbf{e}_z=(0,1)$ as defined in \eqref{2D_gradient}. Although it is possible to construct differential complexes from the cylindrical-coordinate operators, in practice it is simpler to work with the Cartesian operators. This is because finite element software packages such as MFEM assume the underlying mesh to be embedded in Euclidean space, with geometric coordinate transformations handled by the reference-element mappings. Consequently, in Section \ref{sec:2D_components} we formulate the equations for $\mathbf{B}$ and its derived quantities, such as $\mathbf{J}$, using Cartesian differential operators with cylindrical coordinate transformation factors $r$ written out explicitly.}

\newcommand{\drawbottom}{%
    \filldraw[fill=gray!40, opacity=0.4, draw=none] (0,0.25,0) -- (1,0.25,0) -- (1,0.25,1) -- (0,0.25,1) -- cycle; 
    \draw[thick] (0,0.25,0) -- (1,0.25,0); 
    \draw[thick] (0,0.25,0) -- (0,0.25,1); 
    \draw[thick] (1,0.25,0) -- (1,0.25,1); 
    \draw[thick] (0,0.25,1) -- (1,0.25,1); 
}

\newcommand{\drawcube}{%
    \filldraw[fill=gray!40, opacity=0.4, draw=none] (0,0,0) -- (1,0,0) -- (1,1,0) -- (0,1,0) -- cycle; 
    \filldraw[fill=gray!40, opacity=0.4, draw=none] (0,0,0) -- (0,1,0) -- (0,1,1) -- (0,0,1) -- cycle; 
    \filldraw[fill=gray!40, opacity=0.4, draw=none] (0,0,0) -- (1,0,0) -- (1,0,1) -- (0,0,1) -- cycle; 
    \filldraw[fill=gray!60, opacity=0.6, draw=none] (1,0,0) -- (1,1,0) -- (1,1,1) -- (1,0,1) -- cycle; 
    \filldraw[fill=gray!10, opacity=0.1, draw=none] (0,1,0) -- (1,1,0) -- (1,1,1) -- (0,1,1) -- cycle; 
    \filldraw[fill=gray!40, opacity=0.4, draw=none] (0,0,1) -- (1,0,1) -- (1,1,1) -- (0,1,1) -- cycle; 
    \draw[thick, gray, dashed] (0,0,0) -- (1,0,0); 
    \draw[thick] (1,0,0) -- (1,1,0); 
    \draw[thick] (1,1,0) -- (0,1,0); 
    \draw[thick, gray, dashed] (0,1,0) -- (0,0,0); 
    \draw[thick] (0,0,1) -- (1,0,1); 
    \draw[thick] (1,0,1) -- (1,1,1); 
    \draw[thick] (1,1,1) -- (0,1,1); 
    \draw[thick] (0,1,1) -- (0,0,1); 
    \draw[thick, gray, dashed] (0,0,0) -- (0,0,1); 
    \draw[thick] (1,0,0) -- (1,0,1); 
    \draw[thick] (1,1,0) -- (1,1,1); 
    \draw[thick] (0,1,0) -- (0,1,1); 
}

\newcommand{\drawcoords}{%
\draw[-{Latex[length=2mm]}, thick] (0,0,0) -- (0,0,0.55)
node[anchor=east, inner sep=1pt] {$\hat{\mathbf{i}}$};
\draw[-{Latex[length=2mm]}, thick] (0,0,0) -- (0.55,0,0)
    node[anchor=west, inner sep=1pt] {$\hat{\mathbf{j}}$};
        \draw[-{Latex[length=2mm]}, thick] (0,0,0) -- (0,0.55,0)
        node[anchor=south west, inner sep=1pt] {$\hat{\mathbf{k}}$};
}

\begin{figure}[h]
    \centering
    \begin{tikzpicture}[scale=1.2, every node/.style={font=\large}]
        \tikzset{
        vec/.style={-{Latex[length=2mm]}, thick}
        }
        \begin{scope}[yshift=1.2cm]
            \node at (0.25,-.75) {\normalsize $H(\text{div}, \mathcal{T}_{2D})_1$};
            \drawbottom
            \foreach \a/\b/\c/\dx/\dy/\dz in {
            0.5/0.25/0/0/0/0.3, 0.5/0.25/0/0/0/-0.3,
            0.5/0.25/1/0/0/0.3, 0.5/0.25/1/0/0/-0.3,
            0/0.25/0.5/0.16/0/0, 0/0.25/0.5/-0.16/0/0,
            1/0.25/0.5/0.16/0/0, 1/0.25/0.5/-0.16/0/0
            } {
            \filldraw[fill=black] (\a,\b,\c) circle (1pt);
            \draw[-{Latex[length=1.3mm]}, thick] (\a,\b,\c) -- ++(\dx,\dy,\dz);
            }

            \node at (1.75,0) {\Huge$\otimes$};

            \node at (2.75,-.75) {\normalsize $DG(\mathcal{T}_{1D})_0$};
            \draw[thick] (2.75,-0.50) -- (2.75,0.50);
            \fill (2.75,0) circle (1.5pt);

            \node at (3.75,0) {\Huge$=$};
        \end{scope}
        \begin{scope}[xshift=5cm, yshift=0.8cm]
            \drawcube
            \foreach \a/\b/\c/\dx/\dy/\dz in {
            0.5/0.5/0/0/0/0.3, 0.5/0.5/0/0/0/-0.3,
            0.5/0.5/1/0/0/0.3, 0.5/0.5/1/0/0/-0.3,
            0/0.5/0.5/0.16/0/0, 0/0.5/0.5/-0.16/0/0,
            1/0.5/0.5/0.16/0/0, 1/0.5/0.5/-0.16/0/0
            } {
            \pgfmathsetmacro{\makegray}{0}
            \pgfmathparse{\a==0 ? 1 : (\b==0 ? 1 : (\c==0 ? 1 : 0))}
            \let\makegray\pgfmathresult

            \ifnum\makegray=1
                \filldraw[fill=gray, gray] (\a,\b,\c) circle (1pt);
                \draw[-{Latex[length=1.3mm]}, thick, gray] (\a,\b,\c) -- ++(\dx,\dy,\dz);
            \else
                \filldraw[fill=black] (\a,\b,\c) circle (1pt);
                \draw[-{Latex[length=1.3mm]}, thick] (\a,\b,\c) -- ++(\dx,\dy,\dz);
            \fi
            }
        \end{scope}
        \begin{scope}
            \node at (5.25,0) {\Huge$\oplus$};
        \end{scope}
        \begin{scope}[xshift=5cm, yshift=-1.4cm]
            \drawcube
            \foreach \a/\b/\c/\dx/\dy/\dz in {
            0.5/0/0.5/0/0.16/0, 0.5/0/0.5/0/-0.16/0,
            0.5/1/0.5/0/0.16/0, 0.5/1/0.5/0/-0.16/0
            } {
            \pgfmathsetmacro{\makegray}{0}
            \pgfmathparse{\a==0 ? 1 : (\b==0 ? 1 : (\c==0 ? 1 : 0))}
            \let\makegray\pgfmathresult

            \ifnum\makegray=1
                \filldraw[fill=gray, gray] (\a,\b,\c) circle (1pt);
                \draw[-{Latex[length=1.3mm]}, thick, gray] (\a,\b,\c) -- ++(\dx,\dy,\dz);
            \else
                \filldraw[fill=black] (\a,\b,\c) circle (1pt);
                \draw[-{Latex[length=1.3mm]}, thick] (\a,\b,\c) -- ++(\dx,\dy,\dz);
            \fi
            }
        \end{scope}
        \begin{scope}[yshift=-1cm]
            \node at (0.25,-.75) {\normalsize $DG(\mathcal{T}_{2D})_0$};
            \drawbottom
            \filldraw[fill=black] (0.5,0.25,0.5) circle (1pt);

            \node at (1.75,0) {\Huge$\otimes$};

            \node at (2.75,-.75) {\normalsize $CG(\mathcal{T}_{1D})_1$};
            \draw[thick] (2.75,-0.50) -- (2.75,0.50);
            \fill (2.75,-0.50) circle (1.5pt);
            \fill (2.75,0.50) circle (1.5pt);

            \node at (3.75,0) {\Huge$=$};
        \end{scope}

        \begin{scope}[xshift=6.5cm]
            \node at (0,0) {\Huge$=$};
        \end{scope}

        \begin{scope}[xshift=7.5cm, yshift=-0.1cm]
            \drawcube
            \foreach \a/\b/\c/\dx/\dy/\dz in {
            0.5/0.5/0/0/0/0.3, 0.5/0.5/0/0/0/-0.3,
            0.5/0.5/1/0/0/0.3, 0.5/0.5/1/0/0/-0.3,
            0.5/0/0.5/0/0.16/0, 0.5/0/0.5/0/-0.16/0,
            0.5/1/0.5/0/0.16/0, 0.5/1/0.5/0/-0.16/0,
            0/0.5/0.5/0.16/0/0, 0/0.5/0.5/-0.16/0/0,
            1/0.5/0.5/0.16/0/0, 1/0.5/0.5/-0.16/0/0
            } {
            \pgfmathsetmacro{\makegray}{0}
            \pgfmathparse{\a==0 ? 1 : (\b==0 ? 1 : (\c==0 ? 1 : 0))}
            \let\makegray\pgfmathresult

            \ifnum\makegray=1
                \filldraw[fill=gray, gray] (\a,\b,\c) circle (1pt);
                \draw[-{Latex[length=1.3mm]}, thick, gray] (\a,\b,\c) -- ++(\dx,\dy,\dz);
            \else
                \filldraw[fill=black] (\a,\b,\c) circle (1pt);
                \draw[-{Latex[length=1.3mm]}, thick] (\a,\b,\c) -- ++(\dx,\dy,\dz);
            \fi
            }
            \node at (0.25,-.75) {\normalsize $H(\text{div}, \mathcal{T}_{3D})_1$};

        \end{scope}

        \begin{scope}[xshift=-1.5cm, yshift=-1.5cm, scale=0.85]
            \drawcoords
        \end{scope}

    \end{tikzpicture}
    \caption{Illustration of forming 3D finite element spaces from tensor products of 1D finite element spaces and 2D finite element spaces, using the example of a 3D divergence-conforming space based on a 2D divergence-conforming space and a 2D DG space.} 
    \label{fig:2d_tensor_product_to_3d}
\end{figure}
{Moreover, the compatible finite element setup can further be motivated when porting data on 2D poloidal meshes to 3D tokamak meshes obtained by toroidally extruding the 2D meshes. In particular,} 3D compatible finite element spaces can be obtained from tensor products between 1D compatible finite element spaces and 2D compatible finite element spaces~\cite{mcrae2016automated}. For example, Figure~\ref{fig:2d_tensor_product_to_3d} shows an illustration of $H(\text{div}, \mathcal{T}_{3D})_1$ formed as the sum of $H(\text{div}, \mathcal{T}_{2D})_1 \otimes DG(\mathcal{T}_{1D})_0$ and $DG(\mathcal{T}_{2D})_0 \otimes CG(\mathcal{T}_{1D})_1$, where~$\otimes$ denotes the tensor product of function spaces. We can similarly form other compatible finite element spaces in 3D, which are illustrated in Figure~\ref{fig:3d_fem_spaces}, where the lowest order compatible finite element spaces on a 3D hexahedral element are shown as an example.

{Note that this construction is independent of the type of extrusion and works equally for angular extrusions used to create curved 3D tokamak meshes from 2D poloidal plane meshes. Upon porting from 2D to 3D, the construction implies that vector fields in the 2D $(r, z)$-plane continue to be associated with the unit vectors $\mathbf{e}_r$ and $\mathbf{e}_z$ only, and not the additional angular unit vector $\mathbf{e}_\phi$. Similarly, scalar fields in 2D are associated only with $\mathbf{e}_\phi$ after porting to 3D (c.f. Figure \ref{fig:2d_tensor_product_to_3d} with unit vectors $(\hat{\mathbf{i}}, \hat{\mathbf{j}}, \hat{\mathbf{k}})$ taken to be $(\hat{\mathbf{e}}_r, \hat{\mathbf{e}}_z, \hat{\mathbf{e}}_\phi)$). In this case, provided the right space pairings are used (more details below), the 2D poloidal vector and toroidal scalar fields $\mathbf{B}_p$, $B_t$, can naturally be transferred into 3D curl- or div-conforming spaces. These spaces in turn are of interest in 3D MHD codes as they allow for divergence-free discretizations for the magnetic field. Finally, we note that while we will use the 2D compatible finite element chains to motivate possible choices of spaces for $\mathbf{B}_p$ and $B_t$ with respect to $\Psi$ and $f$, the observation on 2D to 3D porting motivates the choice of spaces for $\mathbf{B}_p$ and $B_t$ relative to each other. 
}

\begin{figure}[h]
    \centering
    \begin{tikzpicture}[scale=2]
        \begin{scope}
            \drawcube
            \foreach \x in {0,1} \foreach \y in {0,1} \foreach \z in {0,1} {
                    \filldraw[fill=black] (\x,\y,\z) circle (1pt);
                }
            \filldraw[fill=gray, gray] (0,0,0) circle (1pt);
            \node[below=0.3cm] at (0.25,-0.375,0) {$CG(\mathcal{T}_{3D})_1$};
        \end{scope}
        \begin{scope}[xshift=1.75cm]
            \drawcube
            \foreach \a/\b/\c/\dx/\dy/\dz in {
            0.5/0/0/0.15/0/0, 0.5/0/0/-0.15/0/0,
            0.5/1/0/0.15/0/0, 0.5/1/0/-0.15/0/0,
            0.5/0/1/0.15/0/0, 0.5/0/1/-0.15/0/0,
            0.5/1/1/0.15/0/0, 0.5/1/1/-0.15/0/0,
            0/0.5/0/0/0.15/0, 0/0.5/0/0/-0.15/0,
            1/0.5/0/0/0.15/0, 1/0.5/0/0/-0.15/0,
            0/0.5/1/0/0.15/0, 0/0.5/1/0/-0.15/0,
            1/0.5/1/0/0.15/0, 1/0.5/1/0/-0.15/0,
            0/0/0.5/0/0/0.275, 0/0/0.5/0/0/-0.275,
            1/0/0.5/0/0/0.275, 1/0/0.5/0/0/-0.275,
            0/1/0.5/0/0/0.275, 0/1/0.5/0/0/-0.275,
            1/1/0.5/0/0/0.275, 1/1/0.5/0/0/-0.275
            } {
            \pgfmathsetmacro{\makegray}{0}
            \pgfmathparse{(\a==0 && \b==0) ? 1 : (\a==0 && \c==0) ? 1 : (\b==0 && \c==0) ? 1 : 0}
            \let\makegray\pgfmathresult

            \ifnum\makegray=1
                \filldraw[fill=gray, gray] (\a,\b,\c) circle (1pt);
                \draw[-{Latex[length=2mm]}, thick, gray] (\a,\b,\c) -- ++(\dx,\dy,\dz);
            \else
                \filldraw[fill=black] (\a,\b,\c) circle (1pt);
                \draw[-{Latex[length=2mm]}, thick] (\a,\b,\c) -- ++(\dx,\dy,\dz);
            \fi
            }
            \foreach \a/\b/\c/\dx/\dy/\dz in {
            0.5/0/0/0.15/0/0, 0.5/0/0/-0.15/0/0,
            0.5/1/0/0.15/0/0, 0.5/1/0/-0.15/0/0,
            0.5/0/1/0.15/0/0, 0.5/0/1/-0.15/0/0,
            0.5/1/1/0.15/0/0, 0.5/1/1/-0.15/0/0,
            0/0.5/0/0/0.15/0, 0/0.5/0/0/-0.15/0,
            1/0.5/0/0/0.15/0, 1/0.5/0/0/-0.15/0,
            0/0.5/1/0/0.15/0, 0/0.5/1/0/-0.15/0,
            1/0.5/1/0/0.15/0, 1/0.5/1/0/-0.15/0,
            0/0/0.5/0/0/0.275, 0/0/0.5/0/0/-0.275,
            1/0/0.5/0/0/0.275, 1/0/0.5/0/0/-0.275,
            0/1/0.5/0/0/0.275, 0/1/0.5/0/0/-0.275,
            1/1/0.5/0/0/0.275, 1/1/0.5/0/0/-0.275
            } {
            \pgfmathsetmacro{\makegray}{0}
            \pgfmathparse{(\a==0 && \b==0) ? 1 : (\a==0 && \c==0) ? 1 : (\b==0 && \c==0) ? 1 : 0}
            \let\makegray\pgfmathresult

            \ifnum\makegray=0
                \filldraw[fill=black] (\a,\b,\c) circle (1pt);
                \draw[-{Latex[length=2mm]}, thick] (\a,\b,\c) -- ++(\dx,\dy,\dz);
            \fi
            }
            \node[below=0.3cm] at (0.25,-0.375,0) {$H(\text{curl}, \mathcal{T}_{3D})_1$};
        \end{scope}
        \begin{scope}[xshift=3.5cm]
            \drawcube
            \foreach \a/\b/\c/\dx/\dy/\dz in {
            0.5/0.5/0/0/0/0.275, 0.5/0.5/0/0/0/-0.275,
            0.5/0.5/1/0/0/0.275, 0.5/0.5/1/0/0/-0.275,
            0.5/0/0.5/0/0.15/0, 0.5/0/0.5/0/-0.15/0,
            0.5/1/0.5/0/0.15/0, 0.5/1/0.5/0/-0.15/0,
            0/0.5/0.5/0.15/0/0, 0/0.5/0.5/-0.15/0/0,
            1/0.5/0.5/0.15/0/0, 1/0.5/0.5/-0.15/0/0
            } {
            \pgfmathsetmacro{\makegray}{0}
            \pgfmathparse{\a==0 ? 1 : (\b==0 ? 1 : (\c==0 ? 1 : 0))}
            \let\makegray\pgfmathresult

            \ifnum\makegray=1
                \filldraw[fill=gray, gray] (\a,\b,\c) circle (1pt);
                \draw[-{Latex[length=2mm]}, thick, gray] (\a,\b,\c) -- ++(\dx,\dy,\dz);
            \else
                \filldraw[fill=black] (\a,\b,\c) circle (1pt);
                \draw[-{Latex[length=2mm]}, thick] (\a,\b,\c) -- ++(\dx,\dy,\dz);
            \fi
            }
            \node[below=0.3cm] at (0.25,-0.375,0) {$H(\text{div}, \mathcal{T}_{3D})_1$};
        \end{scope}
        \begin{scope}[xshift=5.25cm]
            \drawcube
            \filldraw[fill=gray, gray] (0.5,0.5,0.5) circle (1pt);
            \node[below=0.3cm] at (0.25,-0.375,0) {$DG(\mathcal{T}_{3D})_0$};
        \end{scope}
    \end{tikzpicture}
    \caption{First-order continuous Galerkin space on a 3D hexahedral element, along with its corresponding curl- and divergence-conforming spaces, and discontinuous Galerkin space.}
    \label{fig:3d_fem_spaces}
\end{figure}



In light of the formation of 3D spaces from tensor products of 1D and 2D spaces, we can decompose a 3D vector finite element space into a 2D vector finite element space and a 2D scalar finite element space within the $(r,z)$-plane, and consider $\mathbf{B}_p$ in the former and $B_t$ in the latter space. For divergence-conforming magnetic fields, these two spaces are given by
\begin{equation}
    \mathbf{B}(r,\phi,z) \in H(\text{div}, \mathcal{T}_{3D})_m \;\;\; \rightarrow \;\;\; \mathbb{V}_{\mathbf{B}_p} = H(\text{div}, \mathcal{T}_{2D})_m, \;\;\; \mathbb{V}_{B_t} = DG(\mathcal{T}_{2D})_{m-1},
    \label{eq:B_spaces_div}
\end{equation}
where $H(\text{div}, \mathcal{T})_m$ denotes a divergence-conforming space of polynomial order $m$ defined on a mesh $\mathcal{T}$, and $DG(\mathcal{T}_{2D})_{m-1}$ denotes the DG space appearing in the same de-Rham complex as the divergence-conforming space (see Figure \ref{fig:2d_tensor_product_to_3d}). On the other hand, for curl-conforming magnetic fields, the spaces are given by
\begin{equation}
    \mathbf{B}(r,\phi,z) \in H(\text{curl}, \mathcal{T}_{3D})_m \;\;\; \rightarrow \;\;\; \mathbb{V}_{\mathbf{B}_p} = H(\text{curl}, \mathcal{T}_{2D})_m, \;\;\; \mathbb{V}_{B_t} = CG(\mathcal{T}_{2D})_{m},
    \label{eq:B_spaces_curl}
\end{equation}
where $H(\text{curl}, \mathcal{T})_m$ is defined analogously to the divergence-conforming version, and $CG(\mathcal{T}_{2D})_{m}$ denotes a CG space of the same de-Rham complex. Similarly for the divergence-conforming current density fields, we have
\begin{equation}
    \mathbf{J}(r,\phi,z) \in H(\text{div}, \mathcal{T}_{3D})_m \;\;\; \rightarrow \;\;\; \mathbb{V}_{\mathbf{J}_p} = H(\text{div}, \mathcal{T}_{2D})_m, \;\;\; \mathbb{V}_{J_t} = DG(\mathcal{T}_{2D})_{m-1},
    \label{eq:J_spaces_div}
\end{equation}
and for the curl-conforming current density fields, we have
\begin{equation}
    \mathbf{J}(r,\phi,z) \in H(\text{curl}, \mathcal{T}_{3D})_m \;\;\; \rightarrow \;\;\; \mathbb{V}_{\mathbf{J}_p} = H(\text{curl}, \mathcal{T}_{2D})_m, \;\;\; \mathbb{V}_{J_t} = CG(\mathcal{T}_{2D})_{m}.
    \label{eq:J_spaces_curl}
\end{equation}


\subsection{Vector-CG spaces}
Besides the compatible finite element spaces, another choice of finite element spaces for vector fields in MHD is given by vector-CG spaces (e.g.,~\cite{bonilla2023fully}). While such spaces are simpler to construct than divergence- and curl-conforming finite element spaces, the resulting MHD discretizations generally require divergence-cleaning methods since the resulting discrete magnetic field evolution in time will generally not be naturally divergence-free. These spaces are defined as $CG(\mathcal{T}_{2D})_m^d$, where the degree $m$ is the polynomial degree of the finite element space and $d$ is the dimension of the vector field. Under vector-CG spaces, for the special case of MHD discretizations in $(r, \phi, z)$-coordinates, as well as for poloidal planes that lie on the $x$- or $y$-axis for MHD discretizations in $(x, y, z)$-coordinates, the magnetic field $\mathbf{B}$ can again be decomposed as
\begin{equation}
    \mathbf{B}(r,\phi,z) \in CG(\mathcal{T}_{2D})_m^3 \;\;\; \rightarrow \;\;\; \mathbb{V}_{\mathbf{B}_p} = CG(\mathcal{T}_{2D})_m^2, \;\;\; \mathbb{V}_{B_t} = CG(\mathcal{T}_{2D})_m,
    \label{eq:B_spaces_vectorCG}
\end{equation}
while the current density $\mathbf{J}$ is decomposed as
\begin{equation}
    \mathbf{J}(r,\phi,z) \in CG(\mathcal{T}_{2D})_m^3 \;\;\; \rightarrow \;\;\; \mathbb{V}_{\mathbf{J}_p} = CG(\mathcal{T}_{2D})_m^2, \;\;\; \mathbb{V}_{J_t} = CG(\mathcal{T}_{2D})_m.
    \label{eq:J_spaces_vectorCG}
\end{equation}
\section{Numerical discretization}
\label{sec:numerical_method}
\subsection{Discretization of the magnetic field}
In light of the definitions of $\mathbf{B}_p$ and $B_t$ given in \eqref{B_from_GS}, we next describe the magnetic fields' discrete derivations given the choice of finite element spaces for the scalar fields $\Psi$ and $f$.

If $\Psi$ is a CG field, then according to the 2D de-Rham complex in the $(r,z)$-plane, $\mathbf{B}_p$ should be divergence-conforming, and computed from \eqref{Bp_from_psi} as
\begin{align}
    \langle \mathbf{w}, r \mathbf{B}_p - \nabla^\perp \Psi \rangle = 0, \qquad \forall \mathbf{w} \in H(\text{div}, \mathcal{T}_{2D})_m, \label{discr_B_p_Hdiv}
\end{align}
since $\nabla^\perp \Psi \in H(\text{div}, \mathcal{T}_{2D})_m$. {Here, in light of our discussion above on the discrete 2D de-Rham complexes, $\langle, \rangle$ denotes the $L^2$ inner product in Cartesian $(r,z)$-space, and the Jacobian factor of $r$ associated with the underlying axisymmetric cylindrical coordinate system is explicitly written out in \eqref{discr_B_p_Hdiv}.} Note that the derivation of $\mathbf{B}_p$ only holds weakly due to the additional factor of $r$.

In contrast, if $\Psi$ is a DG field, then in an analogous argument, $\mathbf{B}_p$ should be curl-conforming, and computed weakly as
\begin{align}
    \langle \mathbf{\Sigma}, r \mathbf{B}_p \rangle +  \langle \nabla^\perp \cdot \mathbf{\Sigma}, \Psi \rangle - \int_{\partial \mathcal{T}_{2D}}  \Psi (\mathbf{\Sigma} \cdot \mathbf{n}^\perp) \, dS = 0, \qquad \forall \mathbf{\Sigma} \in H(\text{curl}, \mathcal{T}_{2D})_m,  \label{discr_B_p_Hcurl}
\end{align}
where we used integration by parts (for further details, see \ref{apx:IBP}). Here, we note that finite element-based GS solvers typically do not return $\Psi$ as a DG field. However, in this case we may still compute $\mathbf{B}_p$ in the $H(\text{curl}, \mathcal{T}_{2D})_m$ space using \eqref{discr_B_p_Hcurl}. The pairing of $\Psi$ with the DG field $\nabla^\perp \cdot \mathbf{\Sigma}$ (noting the 2D planar de-Rham complex) then implicitly corresponds to projecting $\Psi$ into the DG space.

For the toroidal component $B_t$, if $f$ is a CG field, then $B_t$ also naturally corresponds to a CG field, computed as
\begin{align}
    \langle \eta, r B_t - f \rangle = 0, \qquad \forall \eta \in CG(\mathcal{T}_{2D})_{m}, \label{discr_B_t_CG}
\end{align}
while if $f$ is a DG field, then similarly $B_t$ should be a DG field, computed as
\begin{align}
    \langle \phi, r B_t - f \rangle = 0, \qquad \forall \phi \in DG(\mathcal{T}_{2D})_{m-1}. \label{discr_B_t_DG}
\end{align}
Again, we note that $f$ is generally not returned as a DG field in GS solvers, and we may still consider \eqref{discr_B_t_DG} to compute $B_t$ for CG fields $f$.

Besides compatible finite element spaces, we also consider vector-CG spaces for comparison. If $\Psi$ is a CG field, then $\mathbf{B}_p$ should be in a 2D vector-CG space, and computed as
\begin{align}
    \langle \mathbf{v}, r \mathbf{B}_p - \nabla^\perp \Psi \rangle = 0, \qquad \forall \mathbf{v} \in CG(\mathcal{T}_{2D})_m^2. \label{discr_B_p_CG}
\end{align}
The toroidal component $B_t$ should lie in a 2D scalar CG space, which is already given in \eqref{discr_B_t_CG}.

Note that at this point, we have not specified the degree $m$ yet. In practice, the degree may, for example, be set according to the position of the $\Psi$-finite element space within the de-Rham complex. For instance, if $\Psi \in CG(\mathcal{T}_{2D})_k$, then we would consider a divergence-conforming poloidal magnetic field $\mathbf{B}_p$ set in $H(\text{div}, \mathcal{T}_{2D})_k$ (assuming Raviart-Thomas spaces). If $\Psi \in DG(\mathcal{T}_{2D})_k$, then we would instead consider a curl-conforming field $\mathbf{B}_p$ set in $H(\text{curl}, \mathcal{T}_{2D})_{k+1}$. Note, however, that the best choice of degree may change given the additional factor of $r$ in the above equations.

\subsubsection{Projection error analysis}
    {The discrete constructions of $\mathbf{B}_p$ above can be analytically examined. Throughout, we write $\|\mathbf{u}\|_2 \coloneqq \langle \mathbf{u}, \mathbf{u}\rangle^{1/2}$ for the $L^2(\mathcal{T}_{2D})$ norm and $\|\mathbf{u}\|_r \coloneqq \langle r\mathbf{u}, \mathbf{u}\rangle^{1/2}$ for its $r$-weighted counterpart, and set $r_{\min} \coloneqq \min_{\mathcal{T}_{2D}} r$ and $r_{\max} \coloneqq \max_{\mathcal{T}_{2D}} r$, so that}
    \begin{equation}
        {\sqrt{r_{\min}}\,\|\mathbf{u}\|_2 \le \|\mathbf{u}\|_r \le \sqrt{r_{\max}}\,\|\mathbf{u}\|_2 \qquad \forall \, \mathbf{u} \in L^2(\mathcal{T}_{2D})^2.}
        \label{eq:norm_equivalence}
    \end{equation}
    \paragraph{Divergence-conforming case}
    {Consider first the divergence-conforming case \eqref{discr_B_p_Hdiv}, and let $\mathbf{B}_p \coloneqq \tfrac{1}{r}\nabla^\perp \Psi$ with a CG field $\Psi$ denote the exact poloidal field \eqref{Bp_from_psi} and $\mathbf{B}_p^h \in H(\text{div}, \mathcal{T}_{2D})_m$ its discrete counterpart, where $h$ is the mesh size. By \eqref{discr_B_p_Hdiv} we obtain,}
    \begin{equation}
        {\langle \mathbf{w}, r(\mathbf{B}_p^h - \mathbf{B}_p)\rangle = 0 \qquad \forall \mathbf{w} \in H(\text{div}, \mathcal{T}_{2D})_m ,}
    \end{equation}
    {that is, $\mathbf{B}_p^h$ is the $r$-weighted $L^2$-projection of $\mathbf{B}_p$ onto $H(\text{div}, \mathcal{T}_{2D})_m$, and hence the best approximation property in the weighted norm yields}
    \begin{equation}
        {\|\mathbf{B}_p - \mathbf{B}_p^h\|_r = \inf_{\mathbf{v} \in H(\text{div}, \mathcal{T}_{2D})_m} \|\mathbf{B}_p - \mathbf{v}\|_r .}
    \end{equation}
    {Applying \eqref{eq:norm_equivalence} and dividing through by $\sqrt{r_{\min}}$ then yields, for every $\mathbf{v} \in H(\text{div}, \mathcal{T}_{2D})_m$,}
    \begin{equation}
        {\|\mathbf{B}_p - \mathbf{B}_p^h\|_2 \le \sqrt{\frac{r_{\max}}{r_{\min}}}\, \|\mathbf{B}_p - \mathbf{v}\|_2 .}
        \label{eq:proj_err_div_bound}
    \end{equation}
    {Since $\nabla^\perp \Psi \in H(\text{div}, \mathcal{T}_{2D})_m$, we may take $\mathbf{v} = \alpha\,\nabla^\perp \Psi$ for any $\alpha \in \mathbb{R}$, and recall that $\mathbf{B}_p = \tfrac{1}{r}\nabla^\perp \Psi$, so that}
    \begin{equation}
        {\|\mathbf{B}_p - \mathbf{B}_p^h\| \le \sqrt{\frac{r_{\max}}{r_{\min}}}\, \left\|\left(r^{-1} - \alpha\right)\nabla^\perp \Psi\right\|_2 \le \sqrt{\frac{r_{\max}}{r_{\min}}}\, \left\|r^{-1} - \alpha\right\|_{L^\infty(\mathcal{T}_{2D})} \|\nabla^\perp \Psi\|_2 .}
    \end{equation}
    {Finally, choosing $\alpha = \frac12\big(r_{min}^{-1} + r_{max}^{-1}\big)$ can be shown to minimize the $L^\infty$ distance to $1/r$, which leads to}
    \begin{equation}
        {\left\|r^{-1} - \alpha\right\|_{L^\infty(\mathcal{T}_{2D})} = \left\|r^{-1} - \tfrac12\big(r_{min}^{-1} + r_{max}^{-1}\big)\right\|_{L^\infty(\mathcal{T}_{2D})} = \frac12 \left(r_{min}^{-1} + r_{max}^{-1}\right),}
        \label{eq:osc_constant}
    \end{equation}
    {and therefore}
    \begin{equation}
        {\|\mathbf{B}_p - \mathbf{B}_p^h\|_2 \le c \|\nabla^\perp \Psi\|_2,}
        \label{eq:proj_err_div}
    \end{equation}
    {where $c = \frac12 \sqrt{\frac{r_{\max}}{r_{\min}}}(r_{\min}^{-1} + r_{\max}^{-1})$ is a constant only depending on the domain. In particular, note that $\lim_{r_{min} \rightarrow \infty} c = 0$, that is as we move our domain away from the cylindrical coordinate pole, the estimate improves towards zero error. In our numerical results section below, we find approximately $c=0.13$.}
    \paragraph{Curl-conforming case 1}
    {For the curl-conforming case \eqref{discr_B_p_Hcurl}, assume first $\Psi \in DG(\mathcal{T}_{2D})_{m-1}$ and define $\mathbf{q}(\Psi) \in H(\text{curl}, \mathcal{T}_{2D})_m$ to be the discrete, weak curl-conforming representation of $\nabla^\perp \Psi$ given by}
    \begin{equation}
        {\langle \mathbf{\Sigma}, \mathbf{q}\rangle = -\langle \nabla^\perp \cdot \mathbf{\Sigma}, \Psi\rangle + \int_{\partial \mathcal{T}_{2D}} \Psi\,(\mathbf{\Sigma} \cdot \mathbf{n}^\perp)\, dS \qquad \forall \mathbf{\Sigma} \in H(\text{curl}, \mathcal{T}_{2D})_m .}
    \end{equation}
    {Comparing with \eqref{discr_B_p_Hcurl}, the discrete field $\mathbf{B}_p^h \in H(\text{curl}, \mathcal{T}_{2D})_m$ is the $r$-weighted $L^2$-projection of $\mathbf{B}_p \coloneqq \tfrac1r \mathbf{q}$,}
    \begin{equation}
        {\langle \mathbf{\Sigma}, r(\mathbf{B}_p^h - \mathbf{B}_p)\rangle = 0 \qquad \forall \mathbf{\Sigma} \in H(\text{curl}, \mathcal{T}_{2D})_m .}
    \end{equation}
    {By the same best-approximation argument as in the previous case, now over $H(\text{curl}, \mathcal{T}_{2D})_m$, we obtain for any $\mathbf{v} \in H(\text{curl}, \mathcal{T}_{2D})_m$}
    \begin{equation}
        {\|\mathbf{B}_p - \mathbf{B}_p^h\|_2 \le \sqrt{\frac{r_{\max}}{r_{\min}}}\, \|\mathbf{B}_p - \mathbf{v}\|_2 .} \label{eq:proj_err_curl_bound}
    \end{equation}
    {Since $\mathbf{q} \in H(\text{curl}, \mathcal{T}_{2D})_m$, we may take $\mathbf{v} = \alpha\,\mathbf{q}$ for any $\alpha \in \mathbb{R}$, and recall that $\mathbf{B}_p = \tfrac1r \mathbf{q}$, so that}
    \begin{equation}
        {\|\mathbf{B}_p - \mathbf{B}_p^h\|_2 \le \sqrt{\frac{r_{\max}}{r_{\min}}}\, \left\|\left(r^{-1} - \alpha\right)\mathbf{q}\right\|_2 \le \sqrt{\frac{r_{\max}}{r_{\min}}}\, \left\|r^{-1} - \alpha\right\|_{L^\infty(\mathcal{T}_{2D})} \|\mathbf{q}\|_2 \le c \|\mathbf{q}\|_2,}
    \end{equation}
    {where we proceeded as in the div-conforming case, with $c>0$ defined as before.}
    \paragraph{Curl-conforming case 2}
    {Our GS solver returns $\Psi \in CG(\mathcal{T}_{2D})_{m}$, rather than in the DG space, which leads to an additional error source. Let $P_{DG}\Psi$ denote the $L^2$-projection of $\Psi$ onto $DG(\mathcal{T}_{2D})_{m-1}$. Since the test functions satisfy $\nabla^\perp \cdot \mathbf{\Sigma} \in DG(\mathcal{T}_{2D})_{m-1}$, the pairing in \eqref{discr_B_p_Hcurl} only ``sees'' this projected component, so that $\mathbf{q}(\Psi) = \mathbf{q}(P_{DG} \Psi)$. Writing the exact poloidal field as before as in the div-conforming case as $\mathbf{B}_p = \tfrac1r\nabla^\perp\Psi$, and adding and subtracting $\tfrac1r\mathbf{q} = \tfrac1r\mathbf{q}(\Psi) = \tfrac1r\mathbf{q}(P_{DG}\Psi) $, the error decomposes as}
    \begin{equation}
        {\mathbf{B}_p - \mathbf{B}_p^h = \left(\frac1r\mathbf{q} - \mathbf{B}_p^h\right) + \left(\frac1r\nabla^\perp\Psi - \frac1r\mathbf{q}\right) ,} \label{error_q}
    \end{equation}
    {so that, by the triangle inequality, the error becomes}
    \begin{equation}
        {\|\mathbf{B}_p - \mathbf{B}_p^h\|_2 \le \left\|r^{-1}\mathbf{q} - \mathbf{B}_p^h\right\|_2 + \left\|r^{-1}\nabla^\perp\Psi - r^{-1}\mathbf{q}\right\|_2 ,}
    \end{equation}
    {where the first term is the weighted-projection error bounded as before. For the second term, we have that}
    \begin{equation}
        {\left\|r^{-1}\nabla^\perp\Psi - r^{-1}\mathbf{q}\right\|_2 \le \left\|r^{-1}\right\|_{L^\infty(\mathcal{T}_{2D})} \|\nabla^\perp\Psi - \mathbf{q}\|_2 .}
    \end{equation}
    {Since $\mathbf{q}$ already captures $\nabla^\perp(P_{DG}\Psi)$, the remaining discrepancy is governed by the unresolved part $\eta \coloneqq \Psi - P_{DG}\Psi$, which is invisible to the discrete pairing but contributes to the true field through a derivative-like operator. An inverse estimate then gives}
    \begin{equation}
        {\|\nabla^\perp\Psi - \mathbf{q}\|_2 \lesssim \|\nabla^\perp\eta\|_2 \lesssim h^{-1}\|\eta\|_2 = h^{-1}\|\Psi - P_{DG}\Psi\|_2 .}
    \end{equation}
    {Combining these bounds, the overall estimate reads}
    \begin{equation}
        {\|\mathbf{B}_p - \mathbf{B}_p^h\|_2 \lesssim c \|\mathbf{q}\|_2 + \left\|r^{-1}\right\|_{L^\infty(\mathcal{T}_{2D})} h^{-1} \|\Psi - P_{DG}\Psi\|_2 .}
        \label{eq:proj_err_curl_general}
    \end{equation}
    {For $\|\Psi - P_{DG}\Psi\|_2 = O(h^m)$, the additional term behaves like $h^{-1}\|\Psi - P_{DG}\Psi\|_2 = O(h^{m-1})$, that is the DG-projection error loses one order of accuracy because it is measured through a derivative-like operator. Our numerical results below confirm that this route introduces significantly larger errors than the divergence-conforming approach.
    }
   \paragraph{Vector CG case} {For the case in which we discretize using CG and vector CG fields for $B_t$ and $\mathbf{B}_p$, respectively, an additional error source arises. Crucially, due to the compatible finite element framework, we were able to set $\mathbf{v} = \alpha\,\nabla^\perp \Psi \in H(\text{div}, \mathcal{T}_{2D})_m$ in the best approximation property \eqref{eq:proj_err_div_bound} and the corresponding curl-conforming version \eqref{eq:proj_err_curl_bound}. For the non-compatible case, this substitution no longer holds true, and the corresponding error can be estimated using a triangle inequality-based error decomposition similar to \eqref{error_q}. This leads to an additional error term of the form 
  \begin{equation}
  \sqrt{\frac{r_{\max}}{r_{\min}}}\|\nabla^\perp \Psi - P_{[CG_m]^2}(\nabla^\perp \Psi)\|_2,
  \end{equation}
  for projection $P_{[CG]^2}$ into the vector CG space, which introduces a mesh-size dependent error. Our numerical results confirm that this route introduces significantly larger errors than the divergence-conforming approach.}

\subsection{Discretization of the divergence of the magnetic field}

Given the above definitions of computing $\mathbf{B}_p$ in cylindrical coordinates, the curl-conforming version can be shown to be weakly divergence-free.

\begin{proposition} \label{prop_zero_div}
    Let $\mathbf{B}_p \in H(\textup{curl}, \mathcal{T}_{2D})_m$ be defined by \eqref{discr_B_p_Hcurl}. Further, assume the boundary of our 2D planar domain to be piecewise linear. Then $\mathbf{B}_p$ admits as discrete divergence $D_b \in CG(\mathcal{T}_{2D})_{m}$ such that $D_b \equiv 0$.
\end{proposition}
\begin{proof}
    We define a discrete weak divergence $D_b$ as
    \begin{align}
        \langle \eta, rD_b \rangle = -\langle r\nabla \eta, \mathbf{B}_p \rangle + \int_{\partial \mathcal{T}_{2D}} r \eta g_1 \, dS, \qquad \forall \eta \in CG(\mathcal{T}_{2D})_{m}, \label{div_B_weak}
    \end{align}
    for
    \begin{equation}
        g_1 \coloneqq -\tfrac{1}{r} \mathbf{n} \cdot \nabla^\perp \Psi.
    \end{equation}
    Since $\eta$ is a CG test function, we have that $\nabla \eta$ is a curl-conforming function, and we can substitute for \eqref{discr_B_p_Hcurl} with $\mathbf{\Sigma}$ set to $\nabla \eta$. This then leads to
    \begin{align}
        \langle \eta, rD_b \rangle =  - \langle \nabla^\perp \cdot \nabla \eta, \Psi \rangle + \int_{\partial \mathcal{T}_{2D}}  \Psi (\nabla \eta \cdot \mathbf{n}^\perp) \, dS + \int_{\partial \mathcal{T}_{2D}} r \eta g_1 \, dS, \qquad \forall \eta \in CG(\mathcal{T}_{2D})_{m}.
    \end{align}
    We then use the discrete identity $\nabla^\perp \cdot \nabla \eta = 0$, such that the first term on the right-hand side vanishes. Further, since we assume a piecewise linear boundary, it can be shown that $\nabla \cdot \mathbf{n}^\perp = 0$ along $\partial \mathcal{T}_{2D}$, and upon applying integration by parts for the first term on the right-hand side (along $\partial \mathcal{T}_{2D}$), we find that the last two terms cancel.
\end{proof}
Note that using the definition of $\mathbf{B}_p$, the term $g_1$ corresponds to $-\mathbf{n} \cdot \mathbf{B}_p$. This is as expected, as can be seen from taking the inner product of the divergence's strong form \eqref{div_B_p_def} with a test function $\eta$ and integrating the divergence of $\mathbf{B}_p$ by parts. In particular, this also implies that our weak definition \eqref{div_B_weak} is consistent with the divergence's strong form.

While the curl-conforming discrete version of $\mathbf{B}_p$ is (weakly) divergence-free, this property no longer holds true for the divergence-conforming version. This is because in strong form \eqref{Bp_from_psi}, we can directly apply the axisymmetric cylindrical coordinate divergence operator $\nabla_c \cdot (\cdot) = \tfrac{1}{r} \nabla \cdot \big(r(\cdot)\big)$, while in weak form \eqref{discr_B_p_Hdiv} -- which is required due to the factor of $r$ -- such a direct application is no longer possible.
In practice, a divergence cleaning operation may therefore be necessary for divergence-conforming $\mathbf{B}_p$.

To verify how the different discretizations preserve the divergence-free property, we will compute the divergence of the magnetic field in both divergence- and curl-conforming finite element spaces and in vector-CG spaces. If $\mathbf{B}_p$ is divergence-conforming, the divergence $D_b$ in DG is discretely computed as
\begin{align}
    \langle r\phi, D_b - \frac{1}{r} \nabla\cdot(r \mathbf{B}_p) \rangle = 0, \qquad \forall \phi \in DG(\mathcal{T}_{2D})_{m-1},
\end{align}
while if $\mathbf{B}_p$ is curl-conforming, the divergence $D_b$ in CG is discretely computed as
\begin{align}
    \langle \eta, rD_b \rangle = -\langle r\nabla \eta, \mathbf{B}_p \rangle + \int_{\partial \mathcal{T}_{2D}} r \eta g_1 \, dS, \qquad \forall \eta \in CG(\mathcal{T}_{2D})_{m}. \label{weak_div}
\end{align}
In the vector-CG case, we can compute the divergence $D_b$ in CG as
\begin{align}
    \langle r\eta, D_b - \frac{1}{r} \nabla\cdot(r \mathbf{B}_p) \rangle = 0, \qquad \forall \eta \in CG(\mathcal{T}_{2D})_{m}.
\end{align}

\subsection{Discretization of the current density field}
Given the discrete magnetic field, the choice of divergence- or curl-conforming finite element space or vector-CG space also determines the choice of space for the current density. If $\mathbf{B}$ is divergence-conforming, then $\mathbf{J}$ is naturally curl-conforming and vice versa. If $\mathbf{B}$ is in a vector-CG space, then for simplicity we also compute $\mathbf{J}$ in a vector-CG space. In the following, we rely on our computed formulas for the current density, such as \eqref{J_from_B}.

For the poloidal component of the current density $\mathbf{J}_p$, if $B_t$ is in a CG space, then $\mathbf{J}_p$ should be divergence-conforming and computed as
\begin{align}
    \langle \mathbf{w}, r \mathbf{J}_p - \nabla^\perp (rB_t) \rangle = 0, \qquad \forall \mathbf{w} \in H(\text{div}, \mathcal{T}_{2D})_m. \label{J_p_from_B_t_CG}
\end{align}
Alternatively, if $B_t$ is in a DG space, then $\mathbf{J}_p$ should be curl-conforming and computed weakly as
\begin{align}
    \langle \mathbf{\Sigma}, r \mathbf{J}_p \rangle + \langle \nabla^\perp \cdot \mathbf{\Sigma}, r B_t \rangle - \int_{\partial \mathcal{T}_{2D}} r B_t (\mathbf{\Sigma} \cdot \mathbf{n}^\perp) \, dS = 0, \qquad \forall \mathbf{\Sigma} \in H(\text{curl}, \mathcal{T}_{2D})_m. \label{discr_J_p_Hcurl}
\end{align}

For the toroidal component $J_t$, the choice of space similarly depends on the poloidal magnetic field $\mathbf{B}_p$. If $\mathbf{B}_p$ is curl-conforming, then $J_t$ should be in a DG space and computed as
\begin{align}
    \langle r \phi, J_t + \nabla \cdot \mathbf{B}_p^\perp \rangle = 0, \qquad \forall \phi \in DG(\mathcal{T}_{2D})_{m-1}. \label{J_t_from_B_p_Hcurl}
\end{align}
On the other hand, if $\mathbf{B}_p$ is divergence-conforming, then $J_t$ should be in a CG space and computed weakly as
\begin{align}
    \langle r \eta, J_t \rangle + \langle \nabla^\perp (r\eta), \mathbf{B}_p \rangle - \int_{\partial \mathcal{T}_{2D}} r \eta (\mathbf{B}_p \cdot \mathbf{n}^\perp) \, dS  = 0, \qquad \forall \eta \in CG(\mathcal{T}_{2D})_{m}. \label{J_t_from_B_p_Hdiv}
\end{align}

In the vector-CG case, given $B_t$ in a CG space, we can compute $\mathbf{J}_p$ in 2D vector-CG space as
\begin{align}
    \langle \mathbf{w}, r \mathbf{J}_p - \nabla^\perp (rB_t) \rangle = 0, \qquad \forall \mathbf{w} \in CG(\mathcal{T}_{2D})_m^2, \label{discr_J_p_CG}
\end{align}
and given $\mathbf{B}_p$ in 2D vector-CG space, we can compute $J_t$ in a CG space as
\begin{align}
    \langle r \eta, J_t + \nabla \cdot \mathbf{B}_p^\perp \rangle = 0, \qquad \forall \eta \in CG(\mathcal{T}_{2D})_{m}. \label{discr_J_t_CG}
\end{align}

Next to this, we note that in principle, we can also discretize the direct current density equations \eqref{J_direct_nondiscr}. For instance, for CG fields $f$ and $\Psi$, we would have divergence-conforming $\mathbf{J}_p$ and $J_t$ in a CG space defined by
\begingroup
\begin{subequations} \label{J_direct}
    \begin{alignat}{3}
         & \langle \mathbf{w}, r \mathbf{J}_p - \nabla^\perp f \rangle = 0,                                                                                                     \qquad                  &  & \forall \mathbf{w} \in H(\text{div}, \mathcal{T}_{2D})_m, \label{J_p_direct} \\
         & \langle \eta, r J_t \rangle - \left\langle \nabla (r \eta), \tfrac{1}{r} \nabla \Psi \right\rangle + \int_{\partial \mathcal{T}_{2D}} \eta ( \mathbf{n} \cdot \nabla \Psi) \, dS = 0, \qquad &  & \forall \eta \in CG(\mathcal{T}_{2D})_m. \label{J_t_direct}
    \end{alignat}
\end{subequations}
\endgroup

\subsection{{Discretization of the force balance equation}}
{To address the force balance equation~\eqref{equilibrium}, since the pressure $p_{GS}$ from the GS solver is in a CG space, according to the 2D de-Rham complex in the $(r,z)$-plane, the pressure gradient term $\beta\nabla p$ on the MHD mesh should be curl-conforming, given as
\begin{equation}
\langle r\mathbf{\Sigma}, \beta \nabla p - \beta \nabla p_{GS} \rangle = 0, \qquad \forall \mathbf{\Sigma} \in H(\text{curl}, \mathcal{T}_{2D})_m.
\end{equation}
We will consider the poloidal component of the Lorentz force to be in the same space as $\beta \nabla p$.
\begin{align}
    {\langle r\mathbf{\Sigma}, [\mathbf{B}\times\mathbf{J}]_p + B_t \mathbf{J}_p^\perp - J_t \mathbf{B}_p^\perp \rangle = 0, \qquad \forall \mathbf{\Sigma} \in H(\text{curl}, \mathcal{T}_{2D})_m.}
    \label{discr_BxJ_p_Hdiv}
\end{align}
As the poloidal component of the Lorentz force is curl-conforming, in order to form a 3D compatible finite element space, the toroidal component of the Lorentz force should be in a CG space, given by
\begin{align}
    \left\langle r\eta, [\mathbf{B}\times\mathbf{J}]_t  - \mathbf{B}_p \cdot \mathbf{J}_p^\perp \right\rangle= 0, \qquad \forall \eta \in CG(\mathcal{T}_{2D})_m.
    \label{discr_BxJ_t_CG}
\end{align}}

{We also include the vector-CG space for comparison. The pressure gradient term $\beta\nabla p$ should be in a 2D vector-CG space computed as}
\begin{equation}
    {\langle r\mathbf{v}, \beta \nabla p - \beta \nabla p_{GS} \rangle = 0, \qquad \forall \mathbf{v} \in CG(\mathcal{T}_{2D})_m^2.}
\end{equation}
Under the vector-CG spaces, we generally will not use $\mathbf{J}$ as an auxiliary variable. Instead, we will compute the Lorentz force directly from the magnetic field $\mathbf{B}$ by \eqref{eq:lorentz-force-from-B}. The discretization of the Lorentz force, in this case, is given by
\begin{subequations}\label{eq:discretization_lorentz_force_vectorCG}
    \begin{alignat}{3}
         & \left\langle r\mathbf{v}, [\mathbf{B}\times\mathbf{J}]_p - \frac{1}{r} B_t \nabla (rB_t) + (\nabla^\perp\cdot \mathbf{B}_p) \mathbf{B}_p^\perp \right\rangle= 0,\qquad &                                          & \forall \mathbf{v} \in CG(\mathcal{T}_{2D})_m^2, \\
         & \left\langle r\eta, [\mathbf{B}\times\mathbf{J}]_t +\frac{1}{r}\left(\mathbf{B}_p \cdot \nabla(rB_t)\right) \right\rangle = 0, \qquad
         &                                                                                                                                                                        & \forall \eta \in CG(\mathcal{T}_{2D})_m.
    \end{alignat}
\end{subequations}

\subsection{Cylindrical coordinate projections}
In the above discussion, we found that the existing factors of $r$ and $1/r$ interfere with the discrete de-Rham complex in $(r,z)$-coordinates, for example, rendering the divergence-conforming magnetic field no longer strictly divergence-free. They may also affect the accuracy of the MHD equilibrium; this can be seen for instance in the computation of $J_t$ either directly from $\Psi$ as in \eqref{J_t_direct}, or indirectly via $\mathbf{B}_p$ by first computing \eqref{discr_B_p_Hdiv}, followed by \eqref{J_t_from_B_p_Hdiv}. These two routes are equivalent only in the absence of the factors of $r$ and $1/r$. Further, all of the above equations may be divided by a factor of $1/r$ before formulating the weak variational form, which may be done for example, to cancel the Jacobian determinant factor of $r$ arising in cylindrical coordinates. In practice, for example, the two computations for divergence-conforming $\mathbf{B}_p$ given by
\begingroup
\begin{subequations} \label{B_p_by_r}
    \begin{alignat}{2}
         & \langle \mathbf{w}, r \mathbf{B}_p - \nabla^\perp \Psi \rangle = 0,            \qquad & \forall \mathbf{w} \in H(\text{div}, \mathcal{T}_{2D})_m, \\
         & \langle \mathbf{w}, \mathbf{B}_p - \tfrac{1}{r} \nabla^\perp \Psi \rangle = 0, \qquad & \forall \mathbf{w} \in H(\text{div}, \mathcal{T}_{2D})_m,
    \end{alignat}
\end{subequations}
\endgroup
may lead to slightly different results depending on the choice of degree $m$. Similarly, the equations for curl-conforming $\mathbf{B}_p$ \eqref{discr_B_p_Hcurl} may also be given by
\begin{align}
    \langle \mathbf{\Sigma}, \mathbf{B}_p \rangle +  \left\langle \nabla^\perp \cdot \left(\frac{1}{r}\mathbf{\Sigma}\right), \Psi \right\rangle - \int_{\partial \mathcal{T}_{2D}} \frac{\Psi}{r} (\mathbf{\Sigma} \cdot \mathbf{n}^\perp) \,dS = 0, \qquad \forall \mathbf{\Sigma} \in H(\text{curl}, \mathcal{T}_{2D})_m.
    \label{discr_B_p_Hcurl_by_r}
\end{align}
This likewise holds true for the current density equations, like \eqref{discr_J_p_Hcurl}, which can alternatively be written as
\begin{align}
    \langle \mathbf{\Sigma}, \mathbf{J}_p \rangle + \left\langle \nabla^\perp \cdot \left(\frac{1}{r}\mathbf{\Sigma}\right), rB_t \right\rangle - \int_{\partial \mathcal{T}_{2D}} B_t (\mathbf{\Sigma} \cdot \mathbf{n}^\perp) \,dS = 0, \qquad \forall \mathbf{\Sigma} \in H(\text{curl}, \mathcal{T}_{2D})_m,
\end{align}
and \eqref{J_direct}, which can alternatively be written as
\begingroup
\begin{subequations} \label{J_direct_by_r}
    \begin{alignat}{3}
         & \langle \mathbf{w}, \mathbf{J}_p - \tfrac{1}{r} \nabla^\perp f \rangle = 0,                                                                                                     \qquad               &  & \forall \mathbf{w} \in H(\text{div}, \mathcal{T}_{2D})_m, \label{J_p_direct_by_r} \\
         & \langle \eta, J_t \rangle -  \left\langle \nabla \eta, \tfrac{1}{r} \nabla \Psi \right\rangle + \int_{\partial \mathcal{T}_{2D}} \tfrac{1}{r} \eta ( \mathbf{n} \cdot \nabla \Psi) \, dS = 0, \qquad &  & \forall \eta \in CG(\mathcal{T}_{2D})_m. \label{J_t_direct_by_r}
    \end{alignat}
\end{subequations}
\endgroup

However, in practice, the impact of the above choices is very minor. We have conducted numerical experiments confirming that the choice of using $r$ or $1/r$ factors does not produce a noticeable difference in the presented results.


\section{Implementation}
\label{sec:implementation}
Our implementation for projecting the GS fields onto the MHD fields is built on the C++ library MFEM~\cite{mfem-web}. The GS fields are obtained from a solver presented in~\cite{serino2024adaptive}, which is implemented in MFEM, with their values represented by MFEM grid functions. The GS solver provides $\Psi(r,z)$ as a first-order polynomial continuous Galerkin function, and defines {$f$ and $p$ as a function of $\Psi$}, which is then interpolated at the degree-of-freedom locations of $\Psi$. In other words, for this work, we assume
\begin{equation}
\Psi \in CG(\mathcal{T}_{GS})_{1}, \qquad f \in CG(\mathcal{T}_{GS})_{1}, \qquad {p \in CG(\mathcal{T}_{GS})_{1}}.
\end{equation}
Often, practical GS solutions, such as those from~\cite{liu2021parallel, amorisco2024freegsnke}, are stored and distributed as EFIT data with $\Psi$-values represented on a regular grid. For our purposes, such data can also be represented as a lowest-order CG field on a finite element mesh corresponding to the regular grid, and our discussion below holds equally true. Finally, we note that extensions to the GS solver from~\cite{serino2024adaptive} to support higher polynomial order spaces and discontinuous Galerkin fields are currently under way, thereby allowing for greater flexibility with respect to our choices of function spaces discussed in Section \ref{sec:discretization}.

In addition to considerations related to the GS solver used in this work, we also briefly discuss our meshing choices and implementation details related to MFEM.

\subsection{{GS solutions}}
{Our GS solutions are obtained from an MFEM-based GS solver~\cite{serino2024adaptive}. Here we consider two sets of GS solutions. The first one is a Taylor state equilibrium, with the pressure $p'(\Psi) = 0$ so that the Lorentz force is the only term in the equilibrium force balance, and the toroidal field function $f(\Psi)$ is a linear function of $\Psi$. The second one is an ITER equilibrium, with $p(\Psi)$ and  $f(\Psi)$ both given by tabulated data as a function of normalized $\Psi$. Both equilibria are computed over the same geometry describing the ITER reactor cross-section. See \cite{serino2024adaptive} for more details on the computational domain and ITER geometry.}

{Note that the force-balance residual in the original GS solution is on the order of $10^{-8}$ for the Taylor equilibrium and $10^{-7}$ for the ITER equilibrium, which is sufficiently small to be considered as a good MHD equilibrium; further details are provided in \cite{serino2024adaptive}.}

\subsection{Mesh}
{For the second and third sources of error mentioned in the introduction (i.e., mesh-related factors) we examine the relation between the GS and MHD meshes $\mathcal{T}_{GS}$ and $\mathcal{T}_{2D}$, respectively. Specifically, we consider four meshes for $\mathcal{T}_{2D}$, which are based on $\mathcal{T}_{GS}$\footnote{Strictly speaking, the GS solver computes on a larger semi-circular cross-section that contains the tokamak reactor region. This domain includes regions that are not part of the  MHD solver. In the numerical results considered here, we instead look into a rectangular mesh sub-section near the plasma region considered in the MHD solver. The mesh is identical to the original GS solver mesh within the plasma region, while the coils outside the plasma region are omitted.}: (i) the GS mesh for Taylor equilibrium as described in \cite{serino2024adaptive}; (ii) the GS mesh perturbed at each mesh node $(r_i, z_i)$ by}
\begin{align}
    \begin{split}
        r_i' &= r_i + \alpha \, \sin(r_i), \\
        z_i' &= z_i + \alpha \, \sin(z_i),
        \label{mesh_perturbation}
    \end{split}
\end{align}
{with $\alpha = 0.05$, (iii) the GS mesh for ITER equilibrium as described in \cite{serino2024adaptive}, and (iv) the GS mesh for Taylor equilibrium with further refined along the separatrix. The second mesh is chosen to show even a small perturbation in the mesh can lead to large errors in the MHD fields, while in practice the misalignment between GS and MHD meshes is more significant.}

The {first, third and fourth} meshes are shown in Figure~\ref{fig:2d_meshes}, while the perturbed mesh is omitted as the perturbation is too small to be easily visually distinguishable from the original GS mesh. {Note that the rectangular domain shown in the visualizations is not our full computational domain. It is zoomed in from a much larger semi-circular computational domain near the plasma region considered in the MHD solver. In all loading tests considered in this work, a natural boundary condition is applied on the boundary, which is placed far from the region of interest. Therefore, the impact of the chosen boundary condition is minimal.} {When examining the choice of finite-element spaces, we use the first and the third meshes for both $\mathcal{T}_{GS}$ and $\mathcal{T}_{2D}$. For the perturbed mesh case, we use the first mesh as $\mathcal{T}_{GS}$ and the perturbed second mesh as $\mathcal{T}_{2D}$ to analyze the effect of misalignment. Finally, when assessing the impact of mesh alignment, we use the refined fourth mesh for both $\mathcal{T}_{GS}$ and $\mathcal{T}_{2D}$ to examine the refinement's effectiveness in dealing with steep gradients near the separatrix.}

\begin{figure}[!htbp]
    \centering
    \begin{subfigure}[t]{0.22\textwidth}
        \centering
        \includegraphics[height=0.20\textheight]{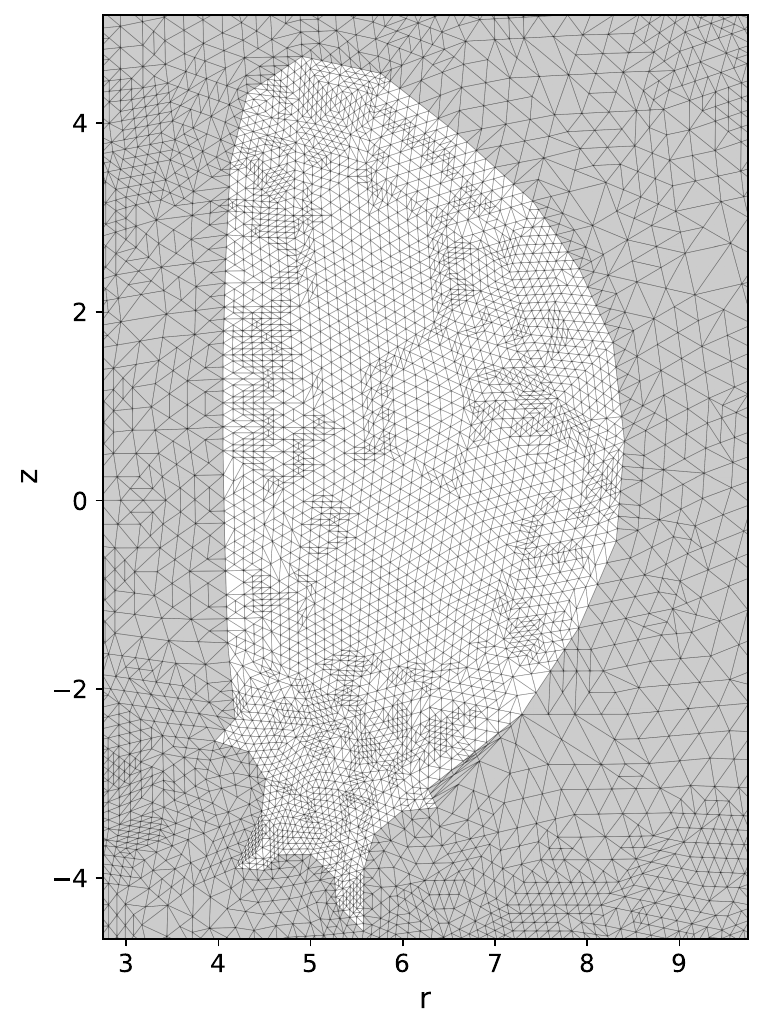}
        \caption{{Mesh for Taylor eql.}}
    \end{subfigure}%
    \begin{subfigure}[t]{0.22\textwidth}
        \centering
        \includegraphics[height=0.20\textheight]{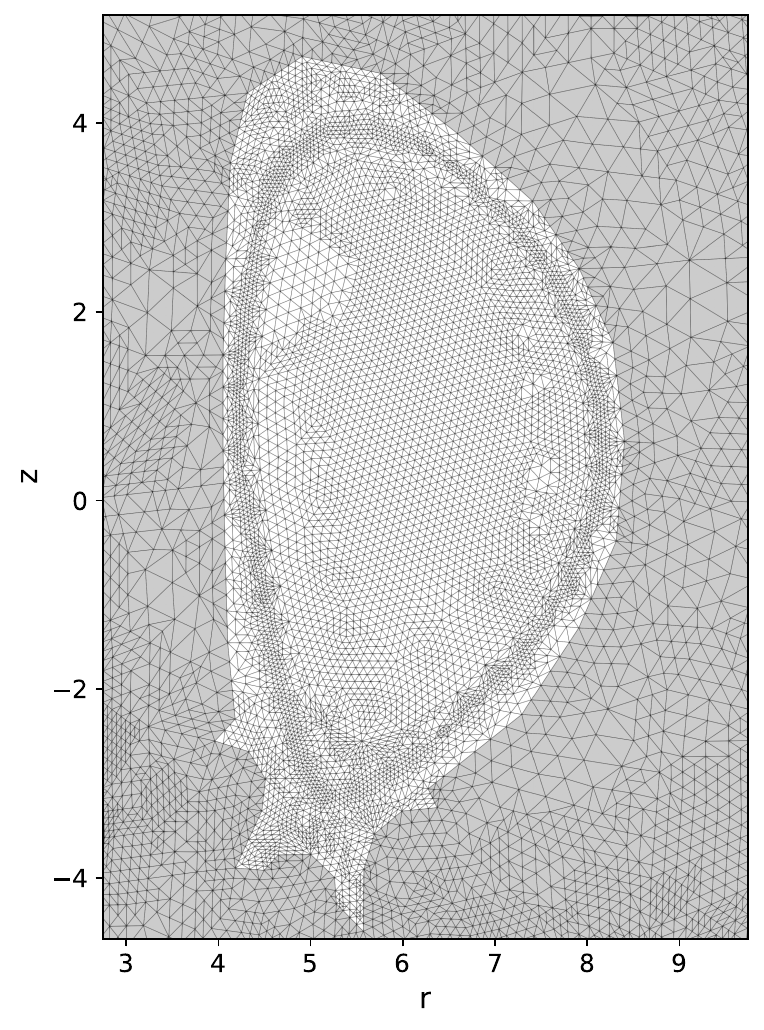}
        \caption{{Refined mesh for Taylor eql.}}
    \end{subfigure}%
    \begin{subfigure}[t]{0.22\textwidth}
        \centering
        \includegraphics[height=0.20\textheight]{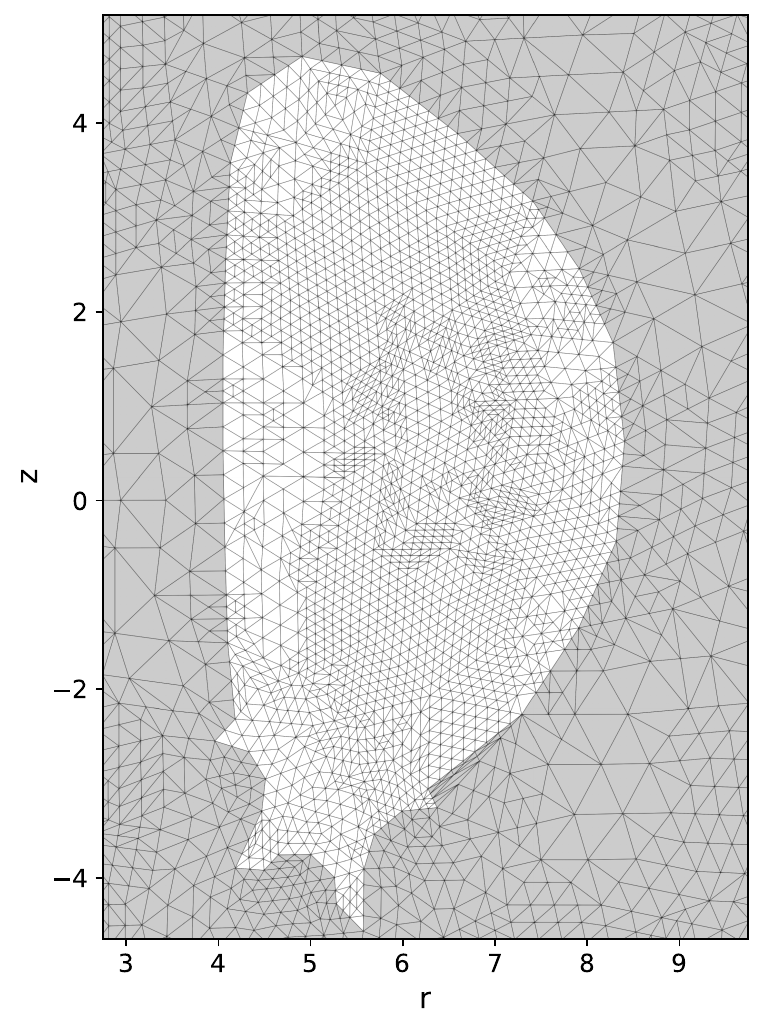}
        \caption{{Mesh for ITER eql.}}
    \end{subfigure}%
    \begin{subfigure}[t]{0.22\textwidth}
        \centering
        \includegraphics[height=0.20\textheight]{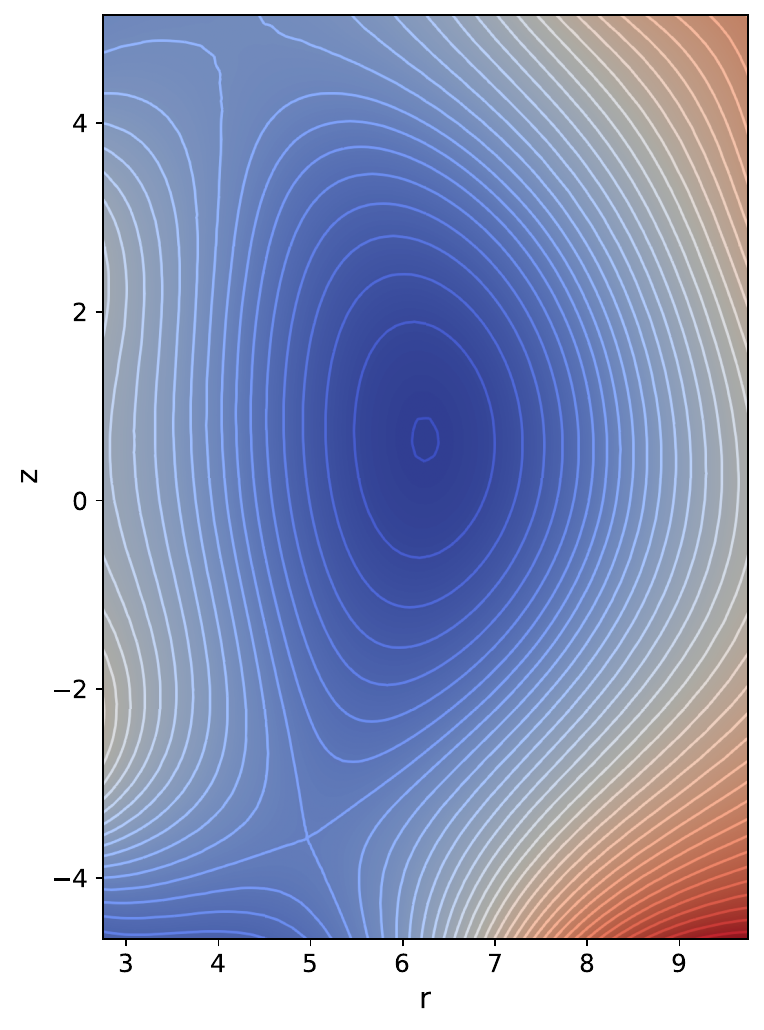}
        \caption{{$\Psi$ from the GS solver}}
        \label{fig:psi_on_gs_mesh}
    \end{subfigure}%
    \begin{subfigure}[t]{0.13\textwidth}
        \includegraphics[height=0.20\textheight]{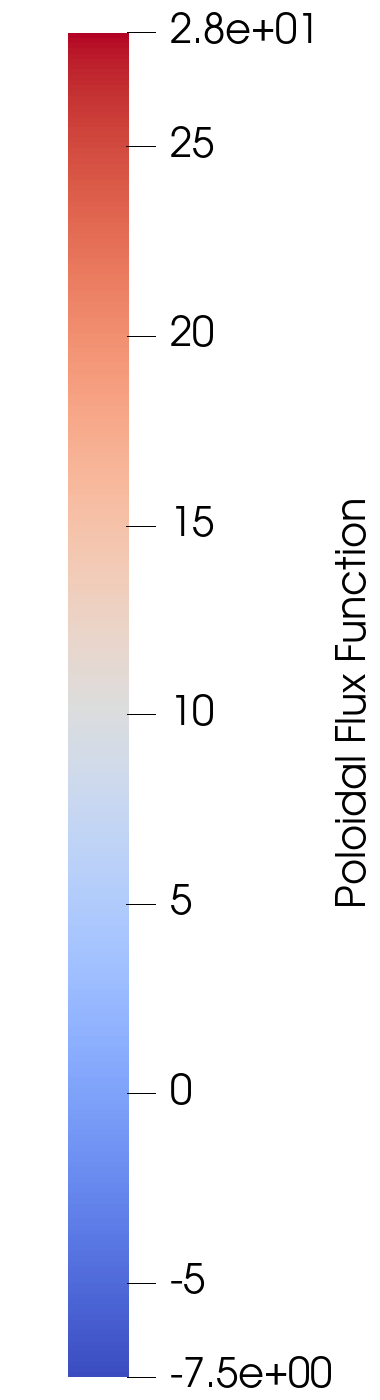}
    \end{subfigure}
    \caption{{2D meshes and input poloidal flux function used in the numerical experiments: (a) GS solver mesh for Taylor equilibrium, (b) GS solver mesh for Taylor equilibrium with refinement along the separatrix, (c) GS solver mesh for ITER equilibrium, and (d) $\Psi$ from the GS solver. The visualizations of the meshes are colored differently by the regions, of which the white region, which is inside the plasma-facing wall, is our MHD equilibrium's plasma region of interest.}}
    \label{fig:2d_meshes}
\end{figure}

\subsection{Product rules}
There are times when cylindrical coordinate projection coefficients cannot be separated from the test function, and therefore these terms cannot be implemented using the basic MFEM integrators. In these cases, we apply product rules to break a single term into two terms that can be computed with the basic integrators. For example, in \eqref{J_t_from_B_p_Hdiv}, the $r$-coefficient cannot be separated from $\nabla^\perp (r\eta)$, and therefore, we apply product rule to obtain,
\begin{align}
        \langle \nabla^\perp (r\eta), \mathbf{B}_p \rangle & = \langle \eta\nabla^\perp r, \mathbf{B}_p \rangle + \langle r\nabla^\perp \eta, \mathbf{B}_p \rangle                                                      \nonumber \\
                                                           & = \langle r\nabla^\perp \eta, \mathbf{B}_p \rangle + \langle \eta\mathbf{e}_z, \mathbf{B}_p \rangle, \qquad \forall \eta \in CG(\mathcal{T}_{2D})_{m}.
\end{align}
Likewise, in \eqref{discr_B_p_Hcurl_by_r}, we apply product rule,
\begin{align}
        \left\langle \nabla^\perp \cdot \left(\frac{1}{r}\mathbf{\Sigma}\right), \Psi \right\rangle & = \left\langle \nabla^\perp \left( \frac{1}{r} \right) \cdot \mathbf{\Sigma}, \Psi \right\rangle + \left\langle \frac{1}{r} \nabla^\perp \cdot \mathbf{\Sigma}, \Psi \right\rangle                                        \nonumber\\
                                                                                                    & = \left\langle -\frac{1}{r^2} \Sigma_r, \Psi \right\rangle + \left\langle \frac{1}{r} \nabla^\perp \cdot \mathbf{\Sigma}, \Psi \right\rangle, \qquad \forall \mathbf{\Sigma} \in H(\text{curl}, \mathcal{T}_{2D})_m.
\end{align}

\subsection{Projection between misaligned meshes}
When projecting between misaligned meshes, it is necessary to evaluate finite element grid functions at arbitrary spatial locations. To achieve this, we use the newly developed MFEM interface for GSLIB, which first locates the mesh element corresponding to the target point and then evaluates the polynomial basis functions at that location. 

Further, to examine the impact of mesh misalignment, we adopt two projection methods depending on whether the field is computed weakly or strongly. For fields that are strongly computed, such as~\eqref{discr_B_p_Hdiv}, we introduce auxiliary variables $\Psi \in CG(\mathcal{T}_{2D})_m$ and $f \in CG(\mathcal{T}_{2D})_m$, which are computed via $L^2$ projections from $\Psi \in CG(\mathcal{T}_{GS})_{m}$ and $f \in CG(\mathcal{T}_{GS})_{m}$, respectively. The inner products in the $L^2$ projections are $r$-weighted due to the Jacobian determinant factor from cylindrical coordinates.
In contrast, the introduction of auxiliary variables is not needed for weak computations, such as~\eqref{discr_B_p_Hcurl}, as $\Psi$ and $f$ appear directly in the right-hand side. In this case, we directly evaluate their values at the quadrature points on grid functions defined in $CG(\mathcal{T}_{GS})_{m}$.

\paragraph{Reproducibility}
All implementation and numerical tests, including the projection approaches mentioned above, are available in the \texttt{tds-load} branch of MFEM at \url{https://github.com/mfem/mfem/tree/tds-load}. During our implementation process, the MFEM version was updated to version~4.6.


\section{Numerical results}
\label{sec:results}
In this section, we present numerical results demonstrating the influence of the three error sources discussed above. First, we examine the effects of different finite element space choices. Next, we analyze the impact of mesh misalignment. Finally, we investigate the influence of mesh refinement along the separatrix as a way to reduce errors introduced by possibly under-resolved strong gradients near the separatrix.
\subsection{Choice of finite element spaces}
\label{sec:results-finite-element-space}
First, to focus on the error induced by the choice of finite element spaces, we begin by using the same mesh for both the MHD and GS solvers. As discussed in Section \ref{sec:implementation}, finite element-based GS solvers typically return $\Psi$ as a function in $CG(\mathcal{T}_{GS})_1$. The field $f = f\big(\Psi(r, z)\big)$ (as well as the pressure $p$) is then given as a function of $\Psi$, which, after interpolation at the degree-of-freedom locations, also lies in $CG(\mathcal{T}_{GS})_1$.

Given these spaces for $\Psi$ and $f$, and following our discussion on compatible finite element spaces in Section \ref{sec:numerical_method}, a natural choice for the magnetic field is to place the poloidal component $\mathbf{B}_p$ in $H(\text{div}, \mathcal{T}_{2D})_m$ and the toroidal component $B_t$ in $CG(\mathcal{T}_{2D})_{m}$. The corresponding current density spaces then follow as $\mathbf{J}_p$ in $H(\text{curl}, \mathcal{T}_{2D})_m$ and $J_t$ in $CG(\mathcal{T}_{2D})_{m}$. Altogether, this leads to the following projections:
\begin{align}
    \begin{split}
         & \Psi \in CG(\mathcal{T}_{2D})_m \;\;\; \rightarrow \;\;\; \mathbf{B}_p \in H(\text{div}, \mathcal{T}_{2D})_m\;\;\; \rightarrow \;\;\; J_t \in CG(\mathcal{T}_{2D})_{m}, \\
         & f \in CG(\mathcal{T}_{2D})_m \;\;\; \rightarrow \;\;\; B_t \in CG(\mathcal{T}_{2D})_m\;\;\; \rightarrow\;\;\; \mathbf{J}_p \in H(\text{div}, \mathcal{T}_{2D})_m,
        \label{eq:ideal_proj_path}
    \end{split}
\end{align}
where we set $m = 1$ in view of the first-order polynomial functions $\Psi$ and $f$ returned by the GS solver.

However, as discussed in Section \ref{sec:discretization}, the above choice of spaces does not naturally follow from a point of view of loading to 3D MHD fields. From \eqref{eq:B_spaces_div}--\eqref{eq:J_spaces_curl}, we have shown that a poloidal component in $H(\text{div}, \mathcal{T}_{2D})_m$ is naturally compatible with a toroidal component in $DG(\mathcal{T}_{2D})_{m-1}$ when constructing 3D finite element spaces from 2D ones via tensor products, while a poloidal component in $H(\text{curl}, \mathcal{T}_{2D})_m$ is naturally compatible with a toroidal component in $CG(\mathcal{T}_{2D})_{m}$.
With this in mind, we can consider two alternative projection paths. The first projection path is to put $B_t$ in $DG(\mathcal{T}_{2D})_{m-1}$, and correspondingly $\mathbf{J}_p$ in $H(\text{curl}, \mathcal{T}_{2D})_m$ in view of the 2D compatible finite element chain. This leads to the following projections:
\begin{align}
    \begin{split}
         & \Psi \in CG(\mathcal{T}_{2D})_m \;\;\; \overset{\eqref{discr_B_p_Hdiv}}{\rightarrow} \;\;\; \mathbf{B}_p \in H(\text{div}, \mathcal{T}_{2D})_m\;\;\; \overset{\eqref{J_t_from_B_p_Hdiv}}{\rightarrow} \;\;\; J_t \in CG(\mathcal{T}_{2D})_{m}, \\
         & f \in CG(\mathcal{T}_{2D})_m \;\;\; \overset{\eqref{discr_B_t_DG}}{\rightarrow} \;\;\; B_t \in DG(\mathcal{T}_{2D})_{m-1}\;\;\; \overset{\eqref{discr_J_p_Hcurl}}{\rightarrow} \;\;\; \mathbf{J}_p \in H(\text{curl}, \mathcal{T}_{2D})_m,
        \label{eq:proj_path_A}
    \end{split}
\end{align}
where we included equation numbers for the corresponding finite element computations from Section \ref{sec:numerical_method}. The other alternative projection path in view of our 3D tensor products is to instead project $\mathbf{B}_p$ into $H(\text{curl}, \mathcal{T}_{2D})_m$ and keep $B_t$ in $CG(\mathcal{T}_{2D})_{m}$. This in turn leads to the following projections:
\begin{align}
    \begin{split}
         & \Psi \in CG(\mathcal{T}_{2D})_m \;\;\; \overset{\eqref{discr_B_p_Hcurl}}{\rightarrow} \;\;\; \mathbf{B}_p \in H(\text{curl}, \mathcal{T}_{2D})_m\;\;\; \overset{\eqref{J_t_from_B_p_Hcurl}}{\rightarrow} \;\;\; J_t \in DG(\mathcal{T}_{2D})_{m-1}, \\
         & f \in CG(\mathcal{T}_{2D})_m \;\;\; \overset{\eqref{discr_B_t_CG}}{\rightarrow} \;\;\; B_t \in CG(\mathcal{T}_{2D})_{m}\;\;\; \overset{\eqref{J_p_from_B_t_CG}}{\rightarrow} \;\;\;\; \mathbf{J}_p \in H(\text{div}, \mathcal{T}_{2D})_m,
        \label{eq:proj_path_B}
    \end{split}
\end{align}
where again we included the corresponding finite element computations' equation numbers. {Notice that unlike \eqref{eq:ideal_proj_path}, where the differentiations are computed in the same de-Rham complex, the above two projection paths involve implicit projections between finite element spaces. \eqref{eq:proj_path_A} computes $B_t$ in $DG(\mathcal{T}_{2D})_{m-1}$ from $f$ in $CG(\mathcal{T}_{2D})_m$ via \eqref{discr_B_t_DG} while \eqref{eq:proj_path_B} computes $\mathbf{B}_p$ in $H(\text{curl}, \mathcal{T}_{2D})_m$ from $\Psi$ in $CG(\mathcal{T}_{2D})_m$ via \eqref{discr_B_p_Hcurl}, both of which involve an implicit projection from $CG(\mathcal{T}_{2D})_m$ to $DG(\mathcal{T}_{2D})_{m-1}$. However, given that the path from $\Psi$ to $J_t$ involves two differentiations while the path from $f$ to $\mathbf{J}_p$ involves only one differentiation, the projection path \eqref{eq:proj_path_B} would be impacted more by the implicit projection error from $CG(\mathcal{T}_{2D})_m$ to $DG(\mathcal{T}_{2D})_{m-1}$.}

\begin{figure}[htb]
    \centering
    \begin{subfigure}[t]{0.22\textwidth}
        \centering
        \includegraphics[height=0.20\textheight]{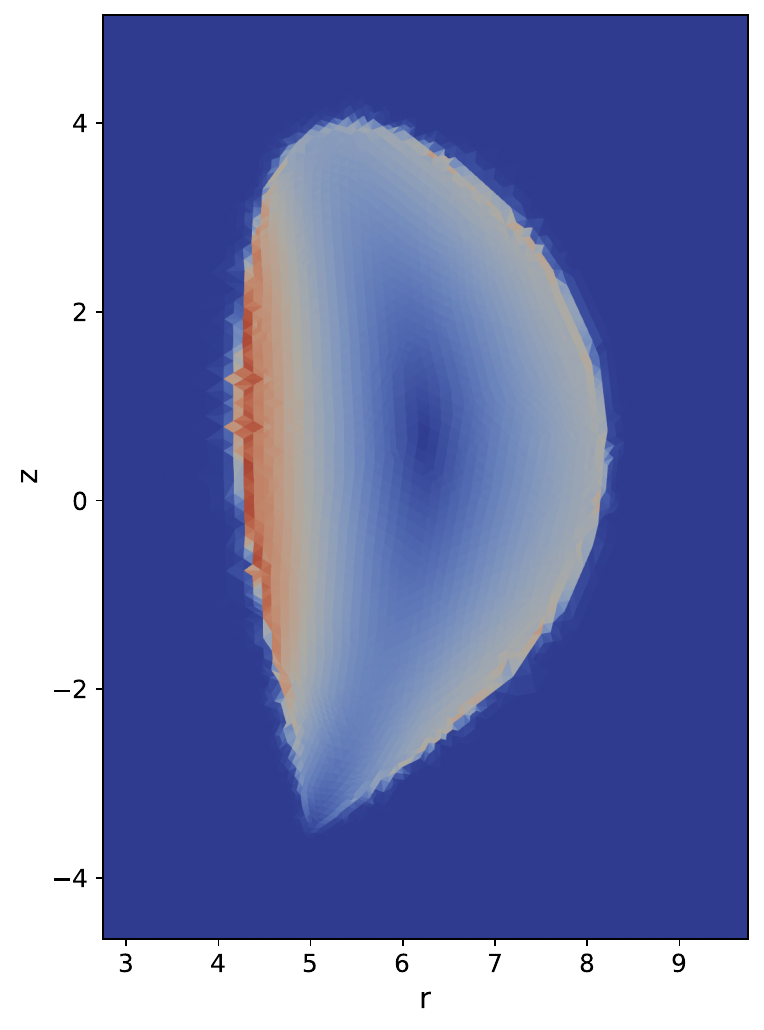}
        \caption{$H(\text{div}, \mathcal{T}_{2D})_m$~\eqref{J_p_direct}}
    \end{subfigure}%
    \begin{subfigure}[t]{0.22\textwidth}
        \centering
        \includegraphics[height=0.20\textheight]{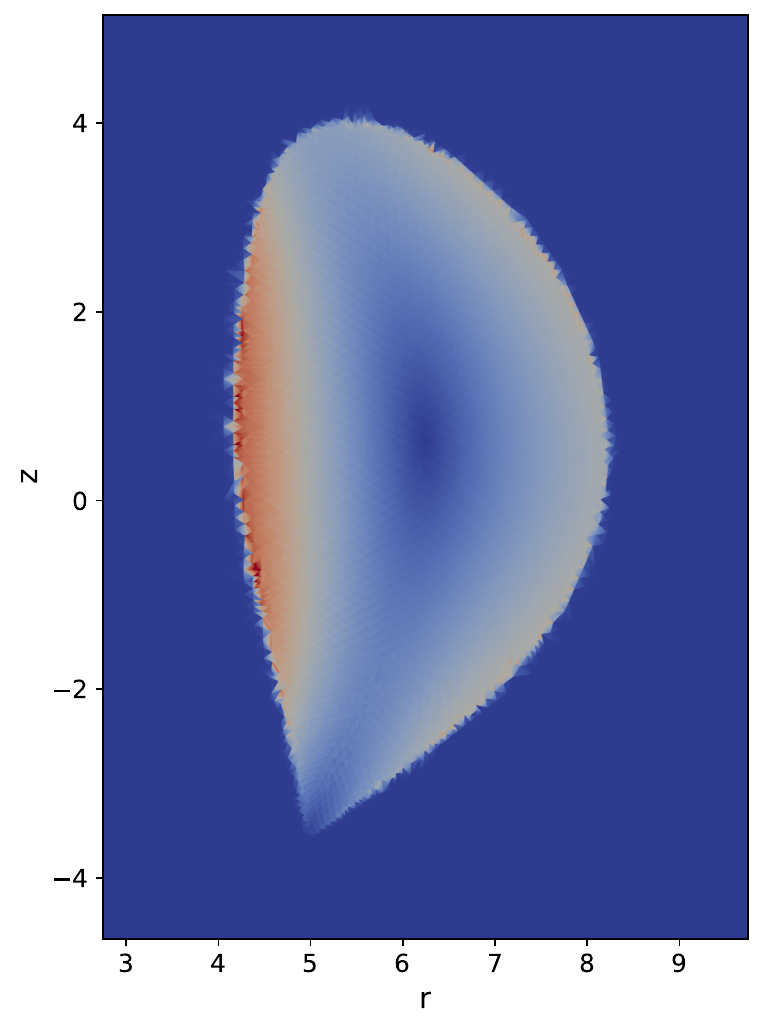}
        \caption{$H(\text{curl}, \mathcal{T}_{2D})_m$~\eqref{eq:proj_path_A}}
    \end{subfigure}%
    \begin{subfigure}[t]{0.22\textwidth}
        \centering
        \includegraphics[height=0.20\textheight]{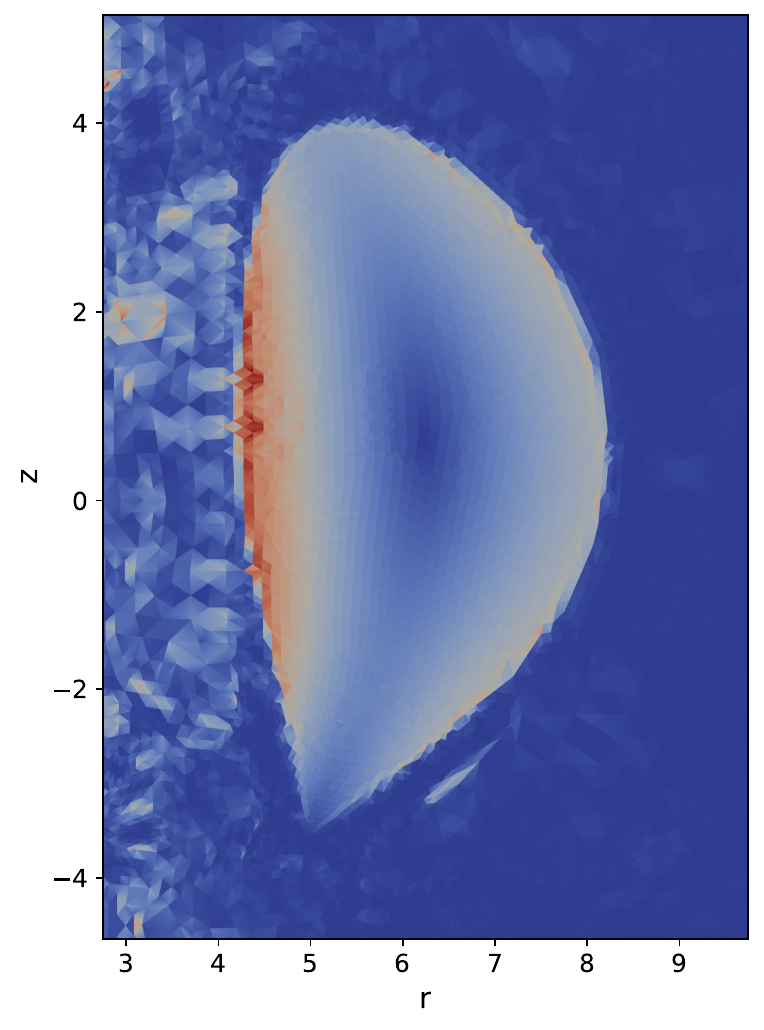}
        \caption{$H(\text{div}, \mathcal{T}_{2D})_m$~\eqref{eq:proj_path_B}}
    \end{subfigure}%
    \begin{subfigure}[t]{0.22\textwidth}
        \centering
        \includegraphics[height=0.20\textheight]{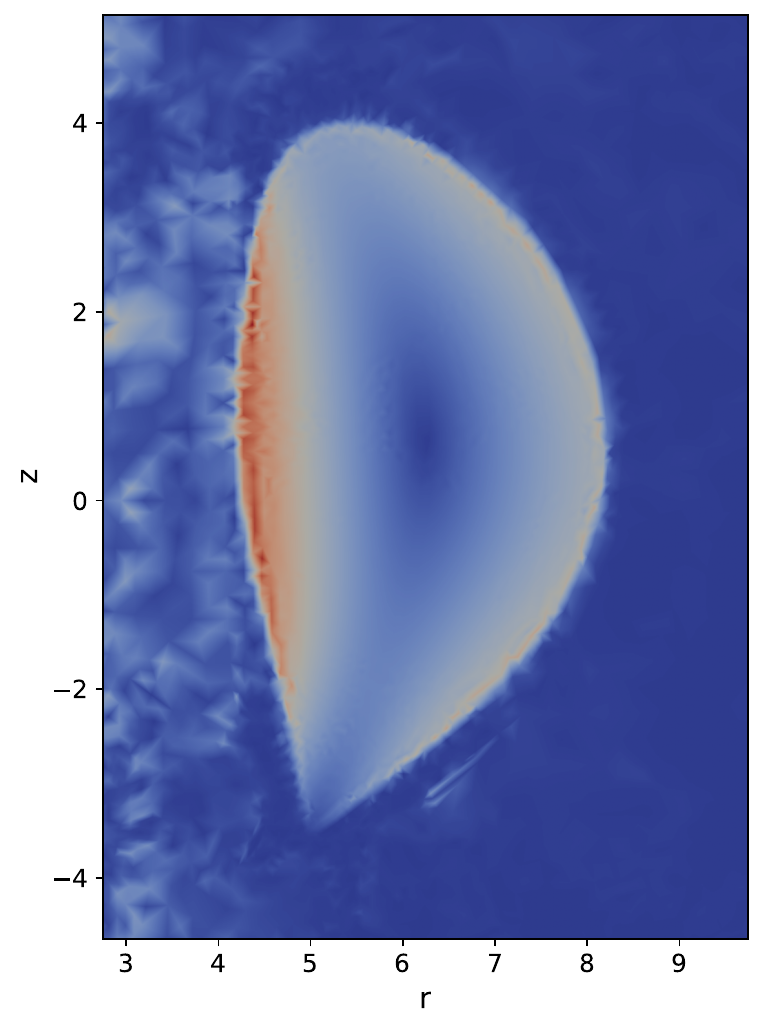}
        \caption{$CG(\mathcal{T}_{2D})_m^2$~\eqref{eq:proj_path_C}}
    \end{subfigure}%
    \begin{subfigure}[t]{0.1\textwidth}
        \includegraphics[height=0.20\textheight]{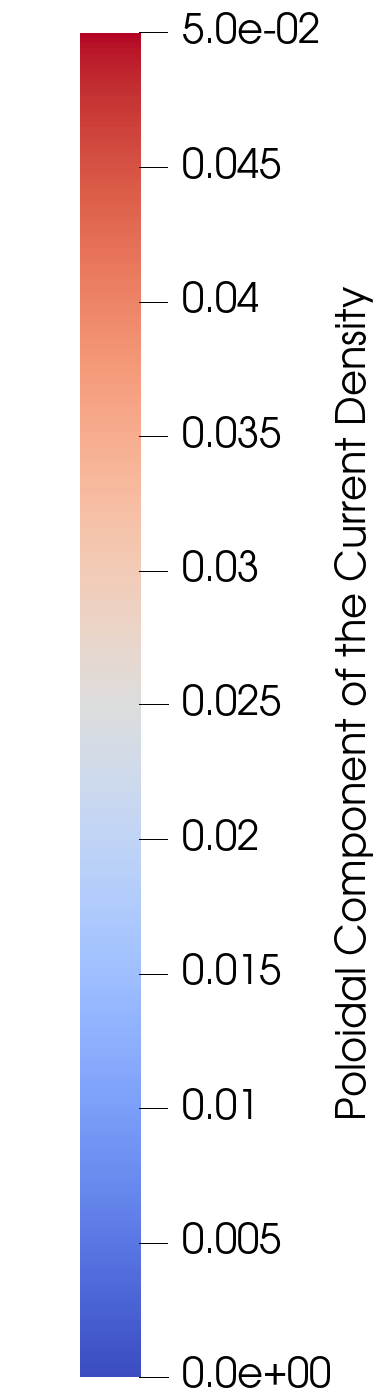}
    \end{subfigure}

    \vspace{1em}

    \begin{subfigure}[t]{0.22\textwidth}
        \centering
        \includegraphics[height=0.20\textheight]{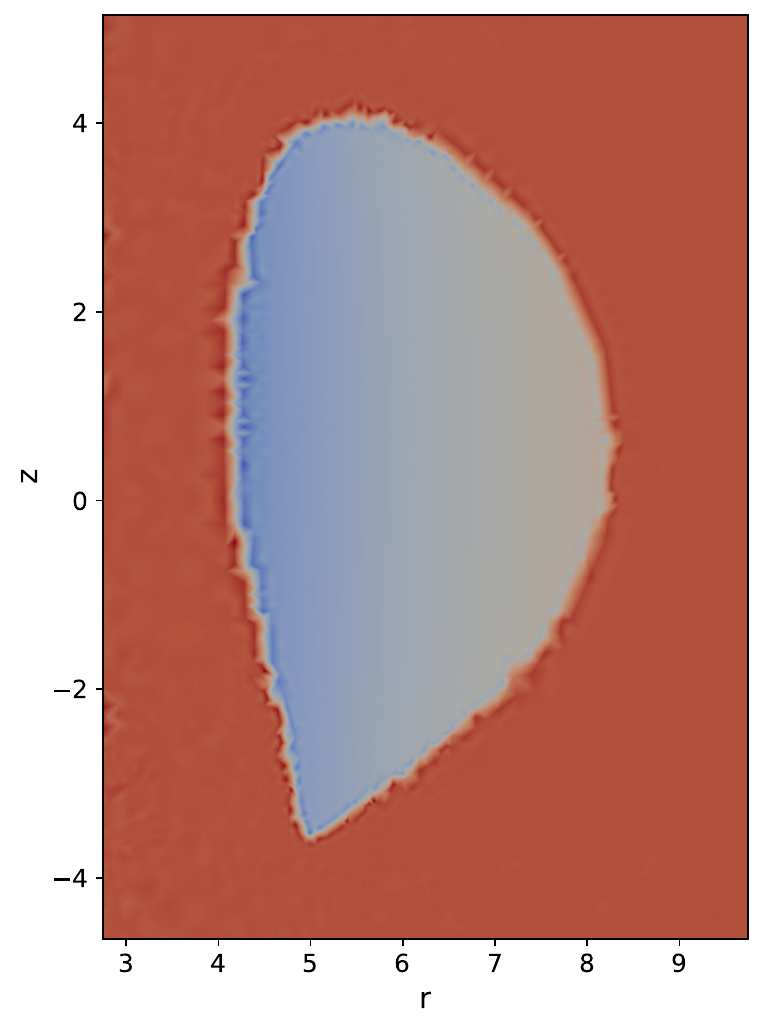}
        \caption{$CG(\mathcal{T}_{2D})_{m}$~\eqref{J_t_direct}}
    \end{subfigure}%
    \begin{subfigure}[t]{0.22\textwidth}
        \centering
        \includegraphics[height=0.20\textheight]{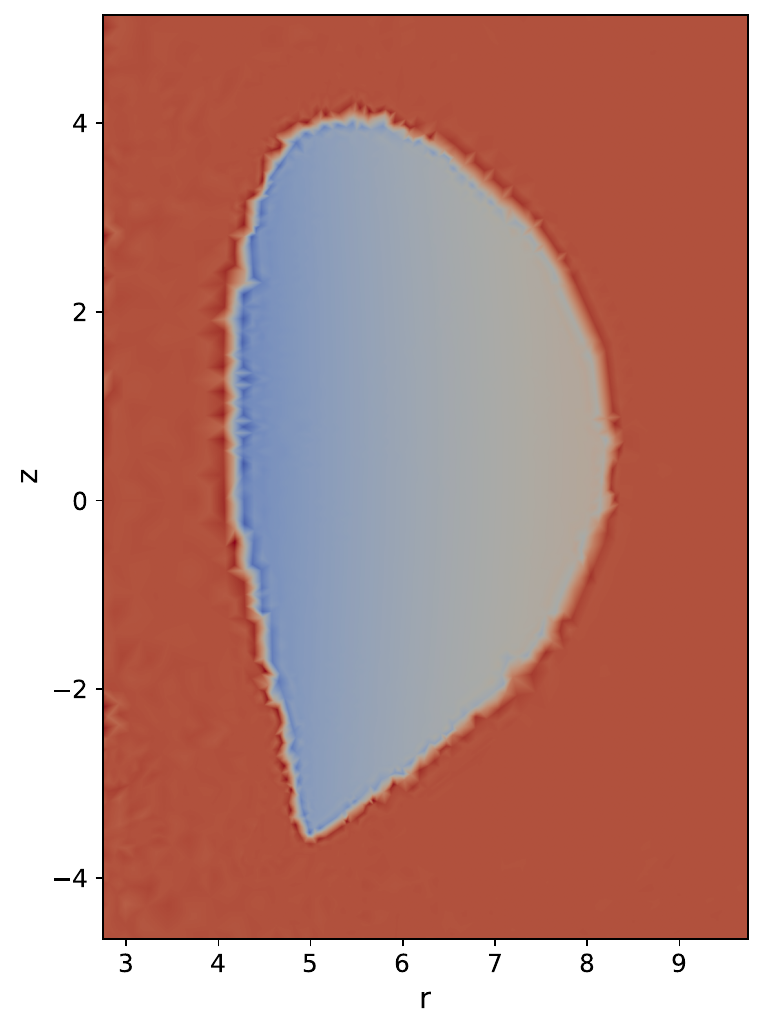}
        \caption{$CG(\mathcal{T}_{2D})_{m}$~\eqref{eq:proj_path_A}}
    \end{subfigure}%
    \begin{subfigure}[t]{0.22\textwidth}
        \centering
        \includegraphics[height=0.20\textheight]{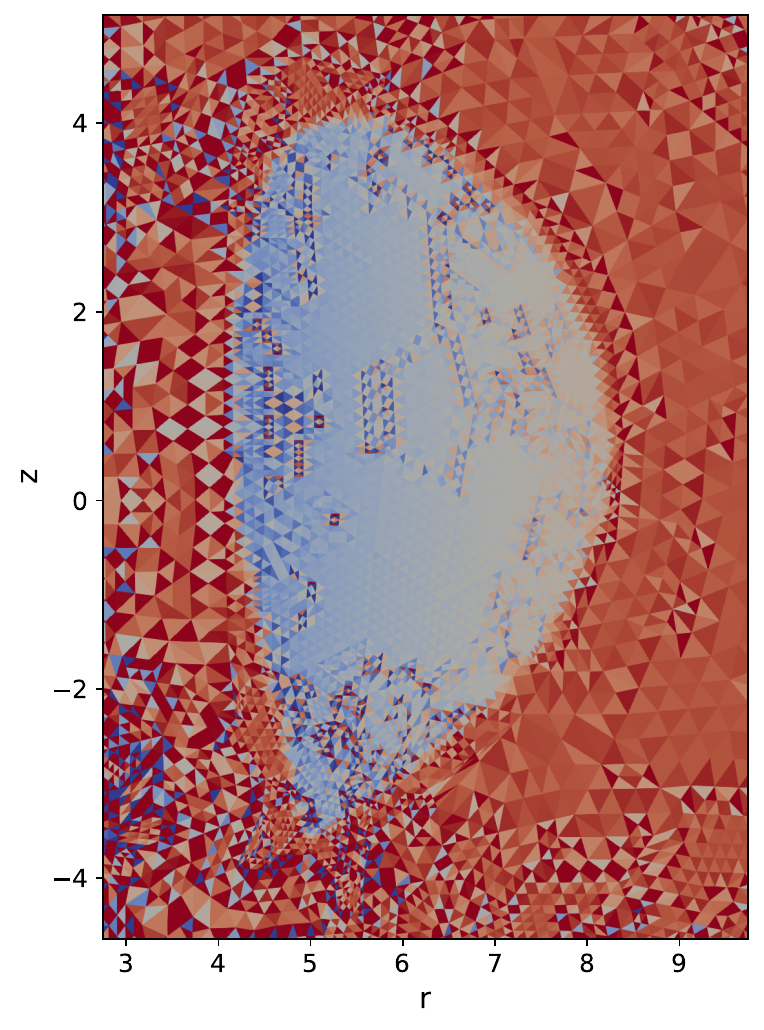}
        \caption{$DG(\mathcal{T}_{2D})_{m-1}$~\eqref{eq:proj_path_B}}
    \end{subfigure}%
    \begin{subfigure}[t]{0.22\textwidth}
        \centering
        \includegraphics[height=0.20\textheight]{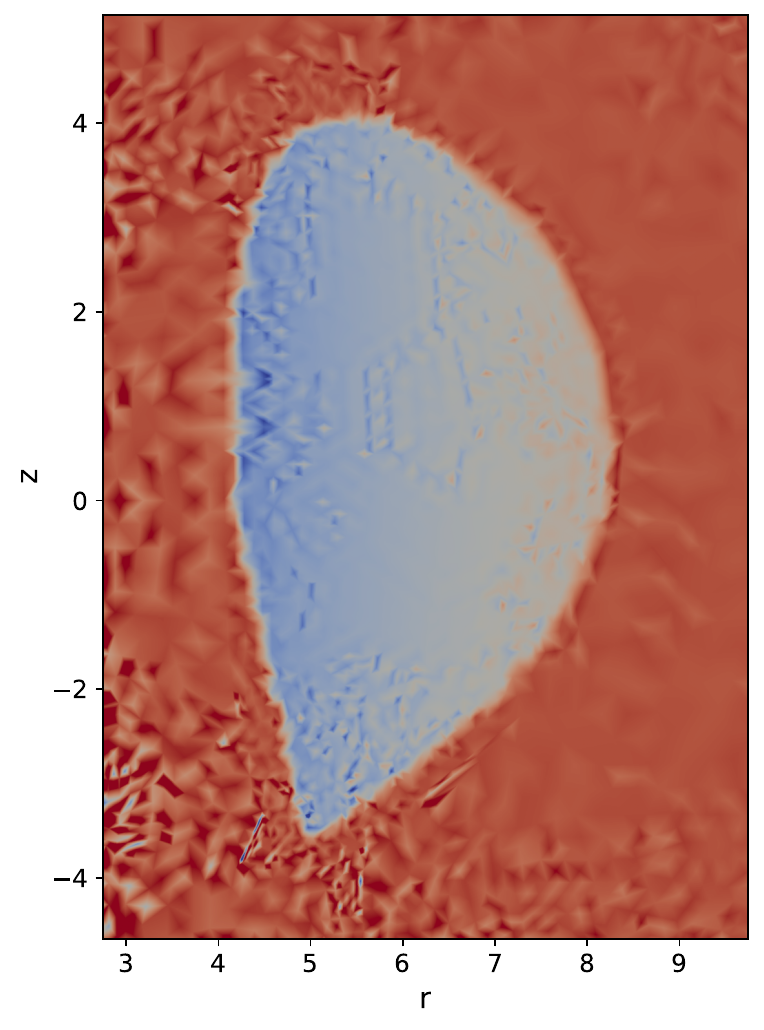}
        \caption{$CG(\mathcal{T}_{2D})_{m}$~\eqref{eq:proj_path_C}}
    \end{subfigure}%
    \begin{subfigure}[t]{0.1\textwidth}
        \includegraphics[height=0.20\textheight]{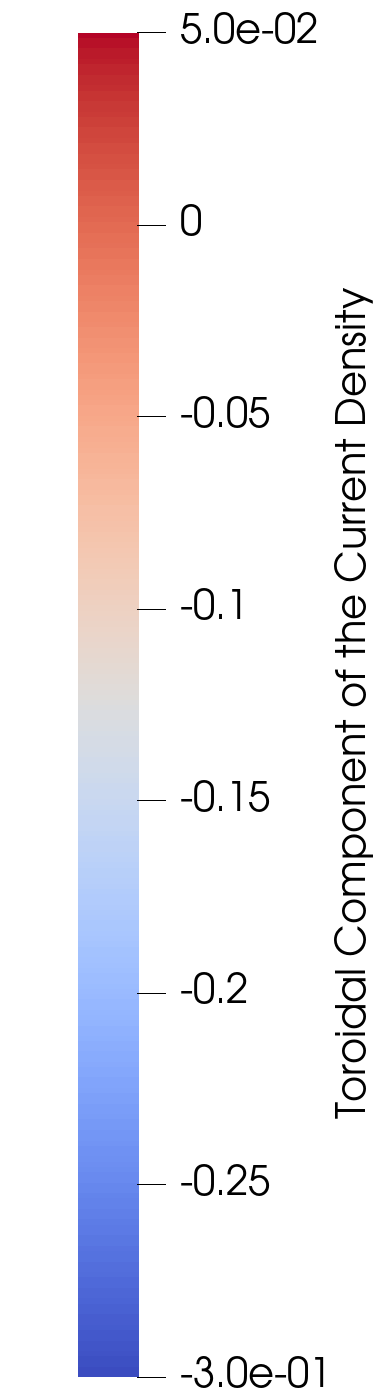}
    \end{subfigure}

    \caption{Comparison of the current density {of Taylor equilibrium} from different projection paths, for $\Psi, f \in CG(\mathcal{T}_{2D})_m$ and $m = 1$.
        Top row: Magnitude of poloidal component $\mathbf{J}_p$.
        (a) reference $\mathbf{J}_p \in H(\text{div}, \mathcal{T}_{2D})_m$ from \eqref{J_p_direct},
        (b) $\mathbf{J}_p \in H(\text{curl}, \mathcal{T}_{2D})_m$ from \eqref{eq:proj_path_A},
        (c) $\mathbf{J}_p \in H(\text{div}, \mathcal{T}_{2D})_m$ from \eqref{eq:proj_path_B},
        (d) $\mathbf{J}_p \in CG(\mathcal{T}_{2D})_m^2$ from \eqref{eq:proj_path_C}.
        Bottom row: Toroidal component $J_t$.
        (e) reference $J_t \in CG(\mathcal{T}_{2D})_{m}$ from \eqref{J_t_direct},
        (f) $J_t \in CG(\mathcal{T}_{2D})_{m}$ from \eqref{eq:proj_path_A},
        (g) $J_t \in DG(\mathcal{T}_{2D})_{m-1}$ from \eqref{eq:proj_path_B},
        (h) $J_t \in CG(\mathcal{T}_{2D})_{m}$ from \eqref{eq:proj_path_C}.}
    \label{fig:J_comparison}
\end{figure}
Besides the projection paths based on compatible finite element spaces, we also consider a setup based on vector-CG spaces. Under this setup, $\mathbf{B}_p$ and $\mathbf{J}_p$ are both in $CG(\mathcal{T}_{2D})_m^2$ and $B_t$ and $J_t$ are both in $CG(\mathcal{T}_{2D})_{m}$. This leads to the following projections:
\begin{align}
    \begin{split}
         & \Psi \in CG(\mathcal{T}_{2D})_m \;\;\; \overset{\eqref{discr_B_p_CG}}{\rightarrow} \;\;\; \mathbf{B}_p \in CG(\mathcal{T}_{2D})_m^2\;\;\; \overset{\eqref{discr_J_t_CG}}{\rightarrow} \;\;\; J_t \in CG(\mathcal{T}_{2D})_{m}, \\
         & f \in CG(\mathcal{T}_{2D})_m \;\;\; \overset{\eqref{discr_B_t_CG}}{\rightarrow} \;\;\; B_t \in CG(\mathcal{T}_{2D})_{m}\;\;\; \overset{\eqref{discr_J_p_CG}}{\rightarrow}\;\;\; \mathbf{J}_p \in CG(\mathcal{T}_{2D})_m^2,
        \label{eq:proj_path_C}
    \end{split}
\end{align}
where all differential operations are computed strongly using, e.g., \eqref{discr_J_p_CG} and \eqref{discr_J_t_CG} for the current density's components.

\begin{figure}[h]
    \centering
    \begin{subfigure}[t]{0.22\textwidth}
        \centering
        \includegraphics[height=0.20\textheight]{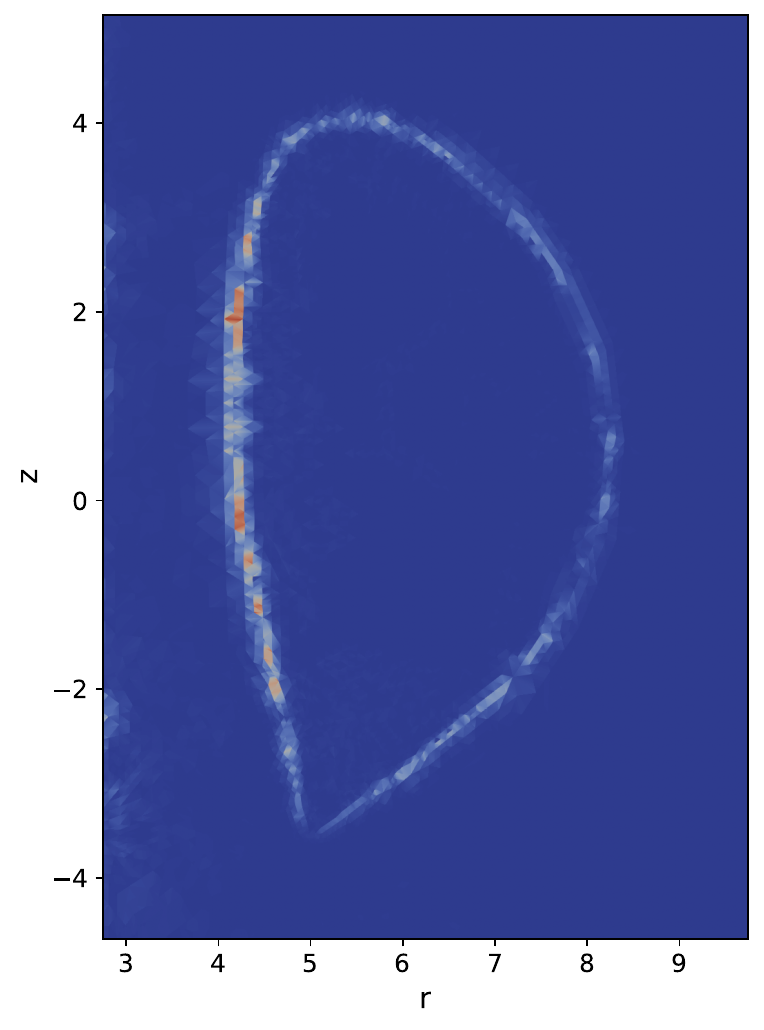}
        \caption{{$H(\text{curl}, \mathcal{T}_{2D})_m$~\eqref{eq:proj_path_A}}\label{fig:JxB_comparison_Taylor1}}
    \end{subfigure}%
    \begin{subfigure}[t]{0.22\textwidth}
        \centering
        \includegraphics[height=0.20\textheight]{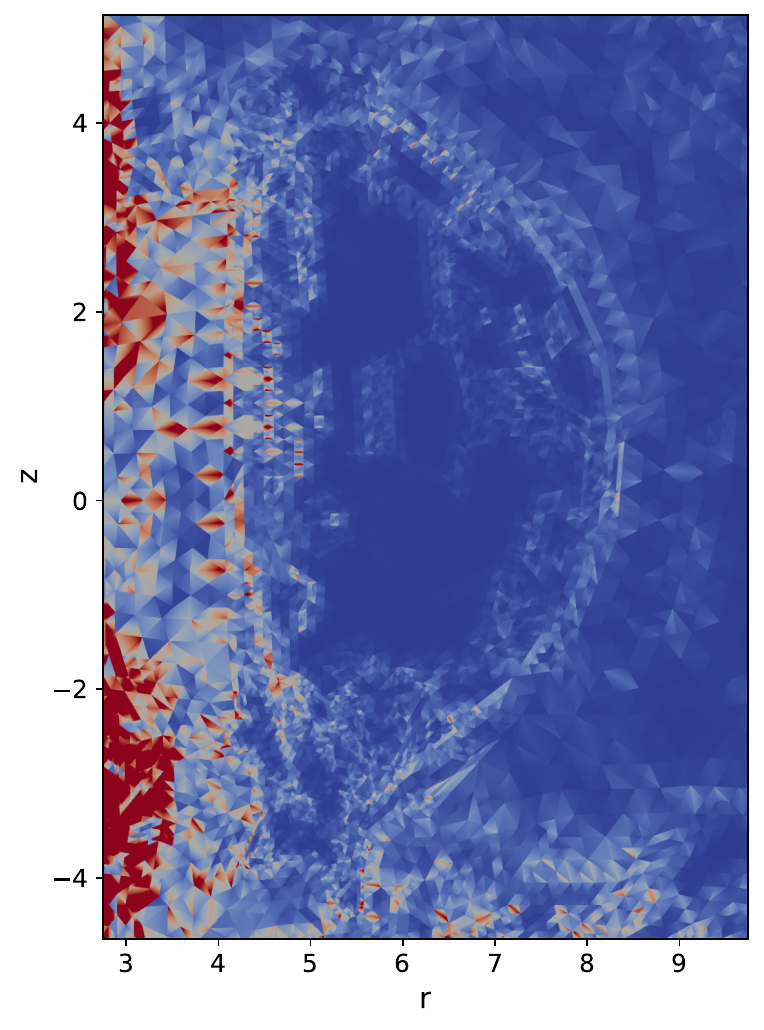}
        \caption{$H(\text{curl}, \mathcal{T}_{2D})_m$~\eqref{eq:proj_path_B}}
    \end{subfigure}%
    \begin{subfigure}[t]{0.22\textwidth}
        \centering
        \includegraphics[height=0.20\textheight]{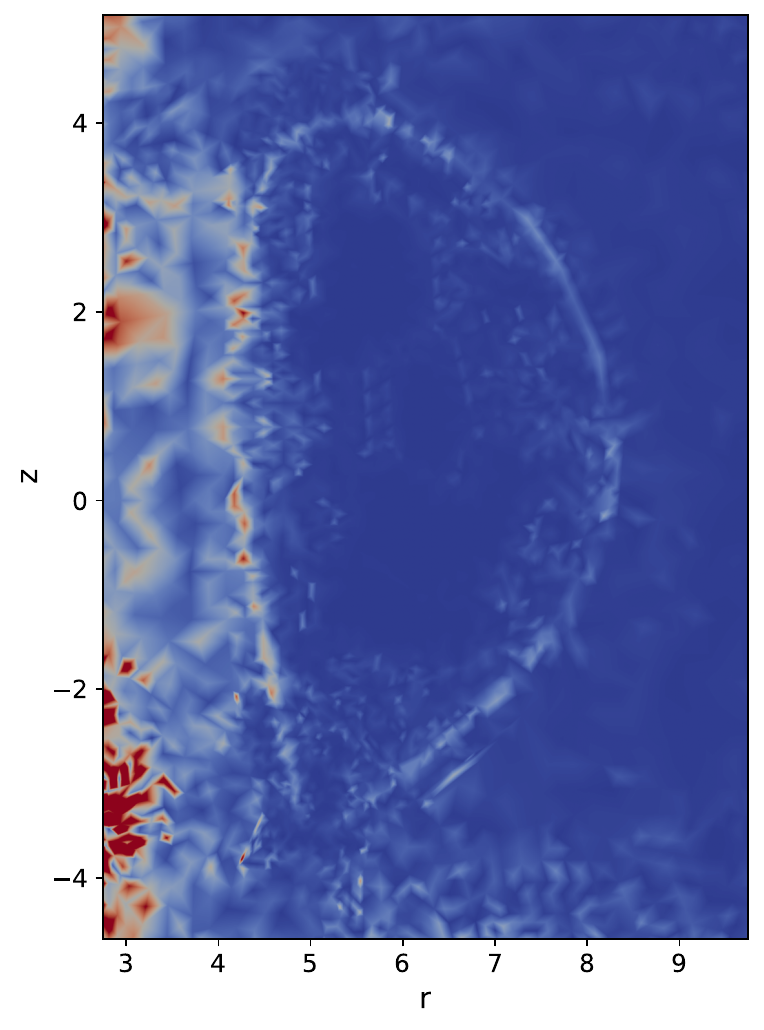}
        \caption{$CG(\mathcal{T}_{2D})_m^2$~\eqref{eq:proj_path_C}}
    \end{subfigure}%
    \begin{subfigure}[t]{0.1\textwidth}
        \includegraphics[height=0.20\textheight]{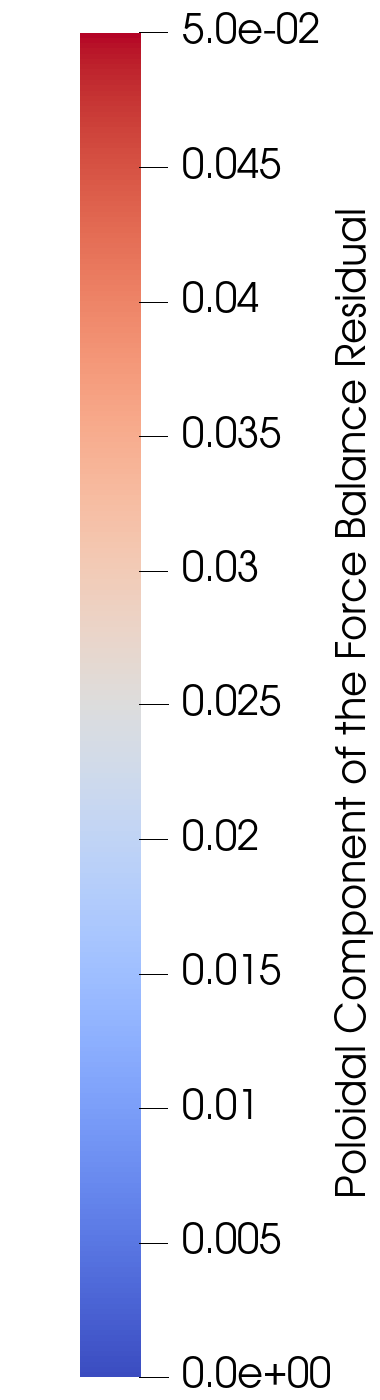}
    \end{subfigure}

    \vspace{1em}

    \begin{subfigure}[t]{0.22\textwidth}
        \centering
        \includegraphics[height=0.20\textheight]{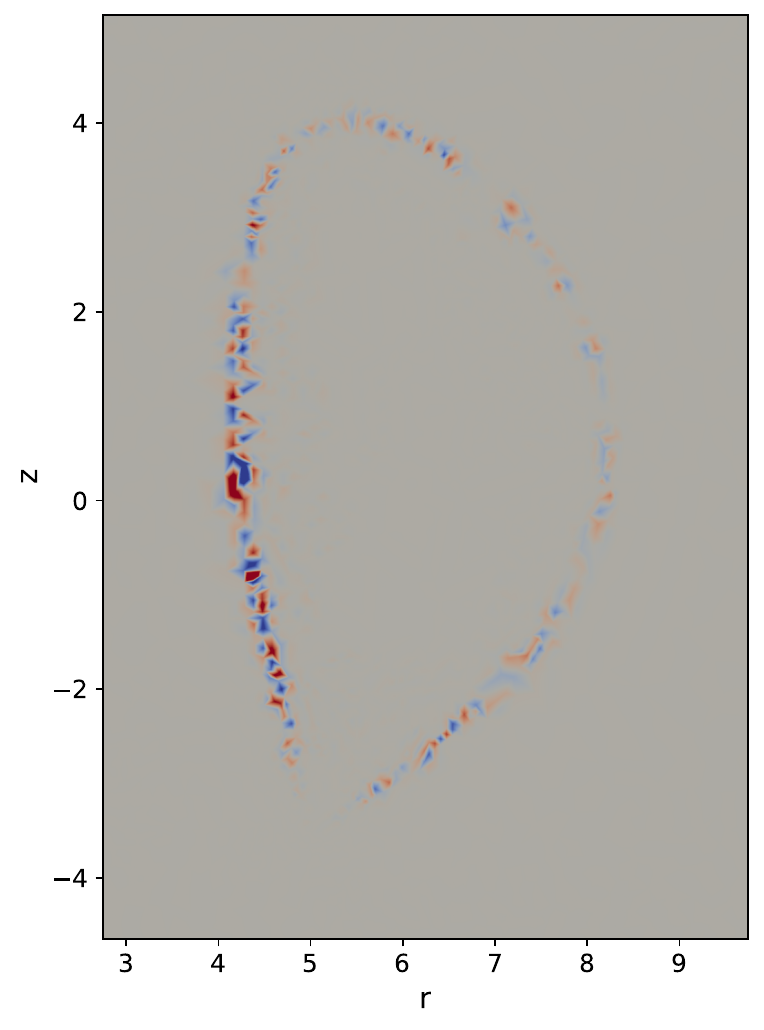}
        \caption{{$CG(\mathcal{T}_{2D})_{m}$~\eqref{eq:proj_path_A}}}
    \end{subfigure}%
    \begin{subfigure}[t]{0.22\textwidth}
        \centering
        \includegraphics[height=0.20\textheight]{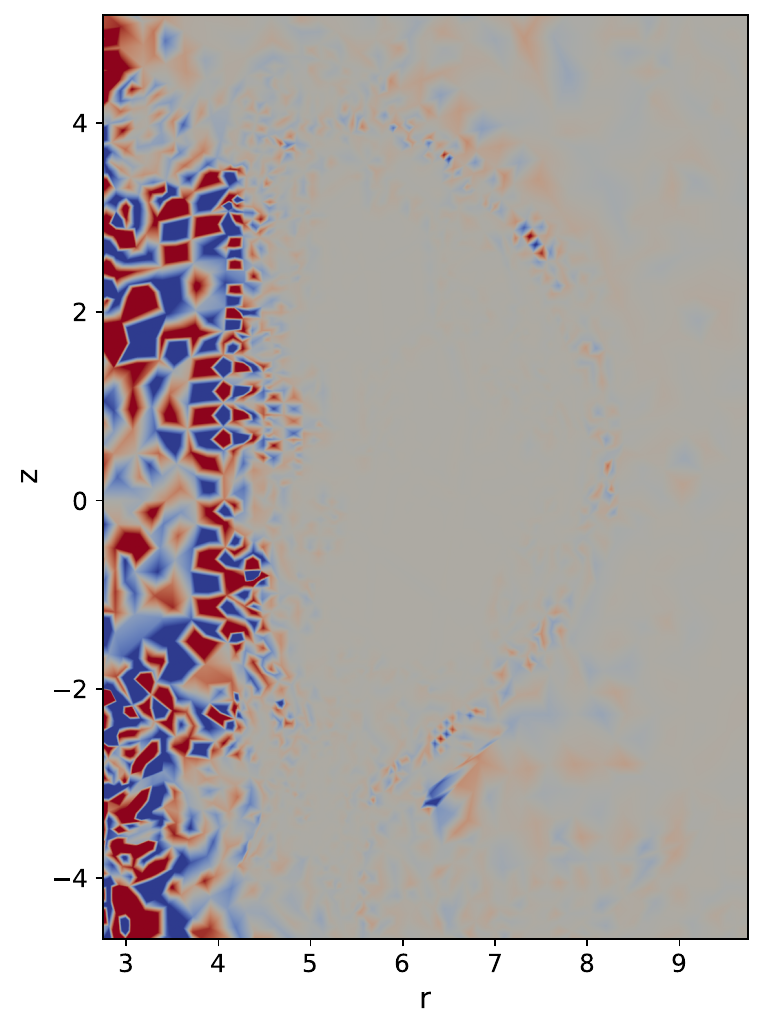}
        \caption{$CG(\mathcal{T}_{2D})_{m}$~\eqref{eq:proj_path_B}}
    \end{subfigure}%
    \begin{subfigure}[t]{0.22\textwidth}
        \centering
        \includegraphics[height=0.20\textheight]{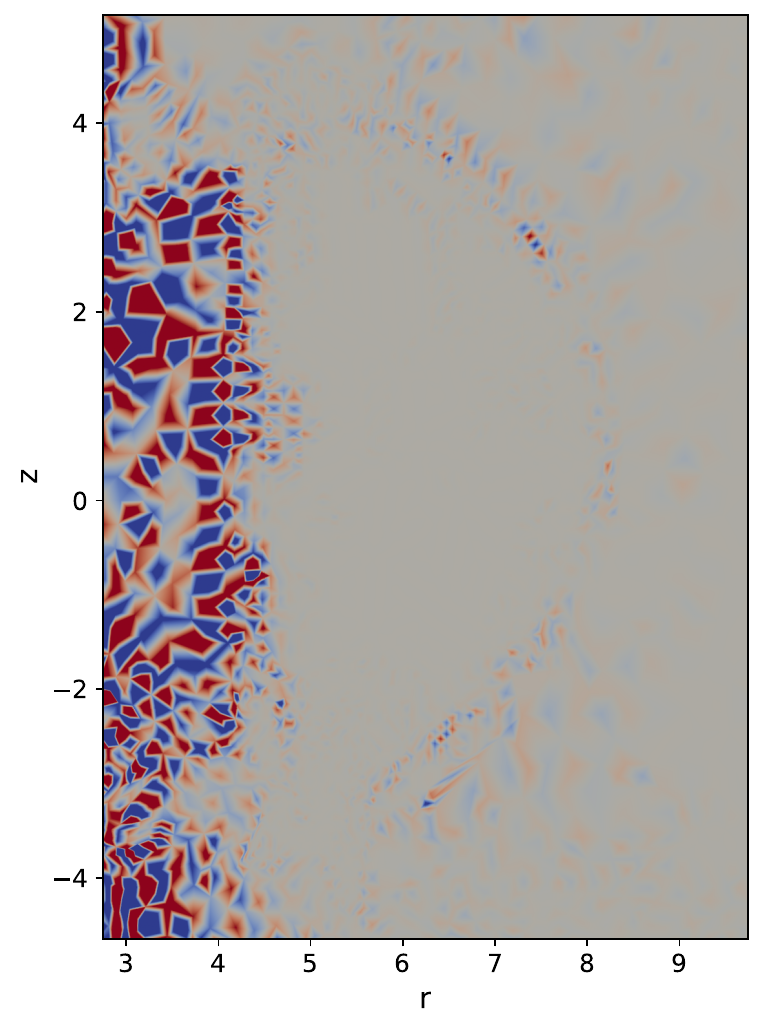}
        \caption{$CG(\mathcal{T}_{2D})_{m}$~\eqref{eq:proj_path_C}}
    \end{subfigure}%
    \begin{subfigure}[t]{0.1\textwidth}
        \includegraphics[height=0.20\textheight]{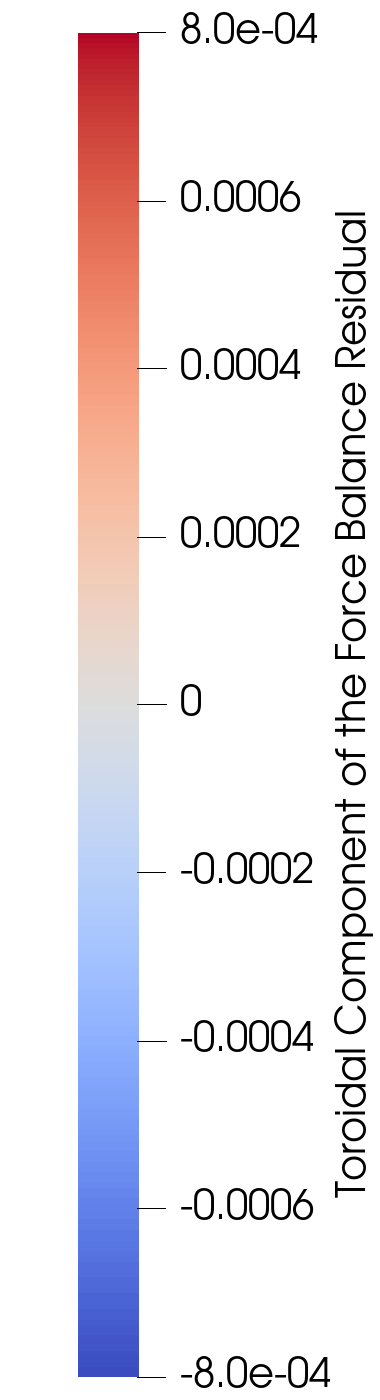}
    \end{subfigure}

    \caption{Comparison of the force balance residual of Taylor equilibrium from different projection paths into the corresponding poloidal and toroidal velocity spaces $\mathbb{V}_p$, $\mathbb{V}_t$, respectively.
        Top row: Poloidal component $[\mathbf{B}\times\mathbf{J}]_p$.
        {(a) $\mathbb{V}_p = H(\text{curl}, \mathcal{T}_{2D})_m$} from \eqref{eq:proj_path_A},
        (b) $\mathbb{V}_p = H(\text{curl}, \mathcal{T}_{2D})_m$ from \eqref{eq:proj_path_B},
        (c) $\mathbb{V}_p = CG(\mathcal{T}_{2D})_m^2$ from \eqref{eq:proj_path_C}.
        Bottom row: Toroidal component $[\mathbf{B}\times\mathbf{J}]_t$.
        (d) $\mathbb{V}_t = DG(\mathcal{T}_{2D})_{m-1}$ from \eqref{eq:proj_path_A},
        (e) $\mathbb{V}_t =  CG(\mathcal{T}_{2D})_{m}$ from \eqref{eq:proj_path_B},
        (f) $\mathbb{V}_t =  CG(\mathcal{T}_{2D})_{m}$ from \eqref{eq:proj_path_C}.}
    \label{fig:JxB_comparison_Taylor}
\end{figure}

\begin{figure}[h]
    \centering
    \begin{subfigure}[t]{0.22\textwidth}
        \centering
        \includegraphics[height=0.20\textheight]{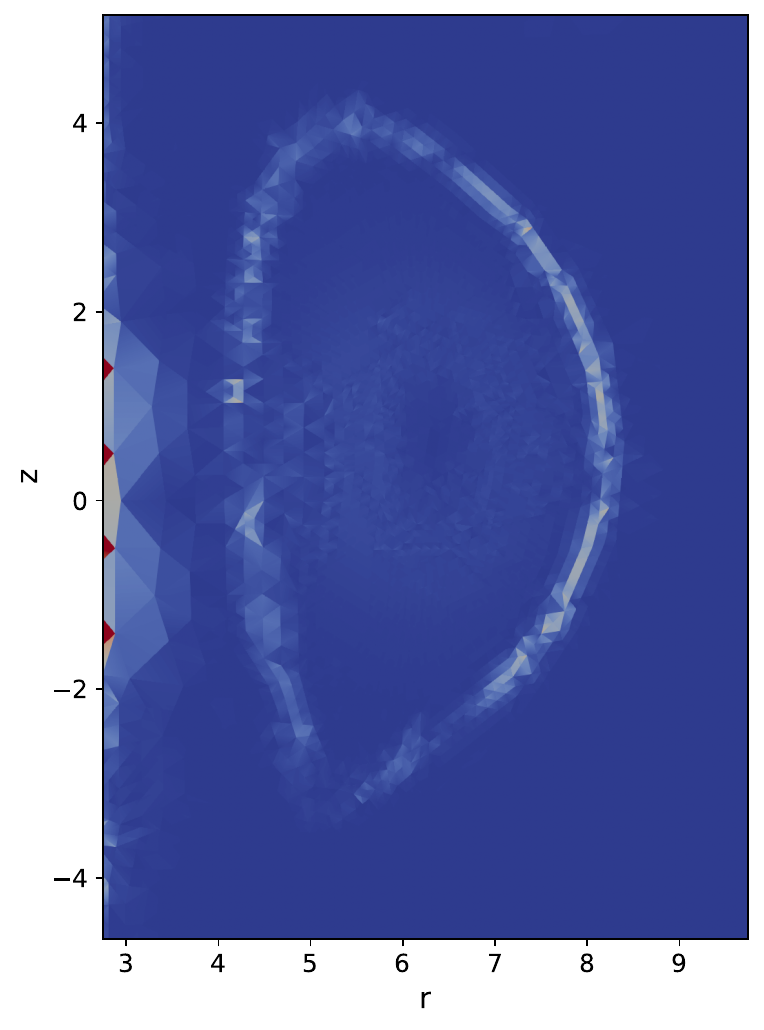}
        \caption{$H(\text{curl}, \mathcal{T}_{2D})_m$~\eqref{eq:proj_path_A}}
    \end{subfigure}%
    \begin{subfigure}[t]{0.22\textwidth}
        \centering
        \includegraphics[height=0.20\textheight]{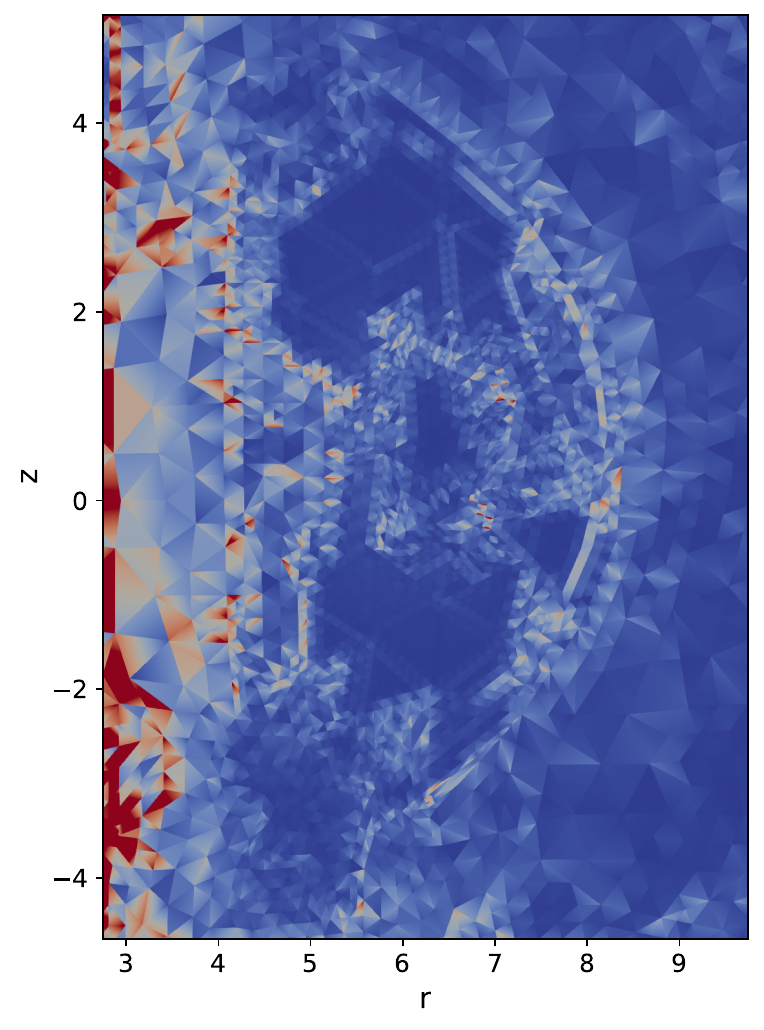}
        \caption{$H(\text{curl}, \mathcal{T}_{2D})_m$~\eqref{eq:proj_path_B}}
    \end{subfigure}%
    \begin{subfigure}[t]{0.22\textwidth}
        \centering
        \includegraphics[height=0.20\textheight]{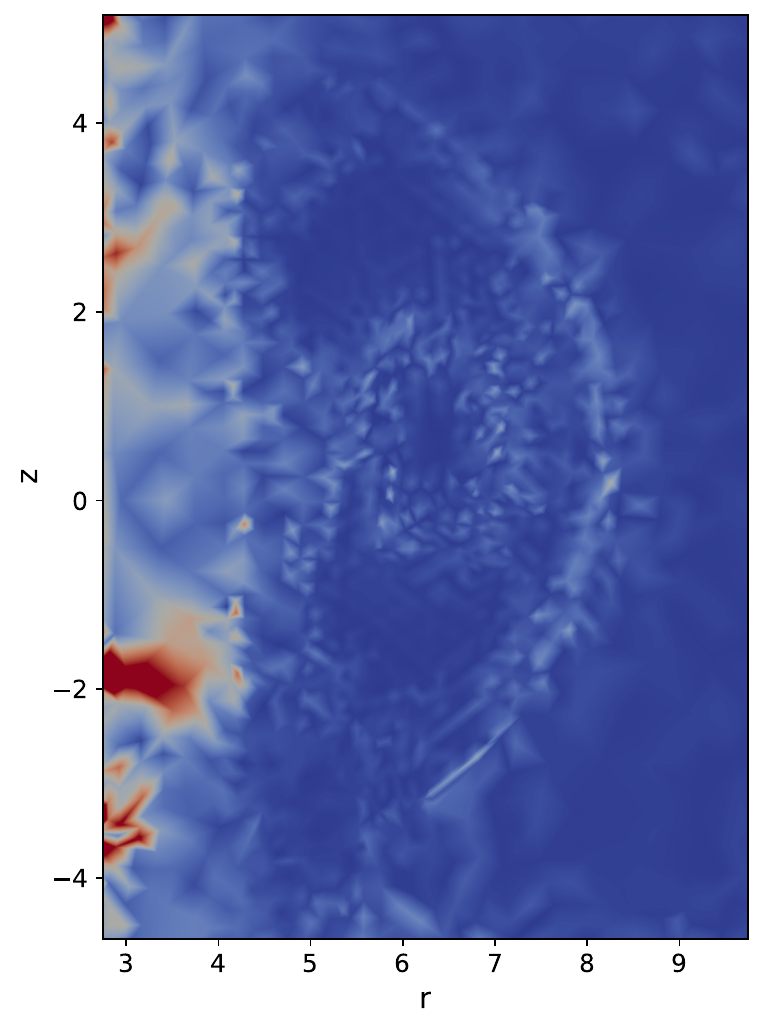}
        \caption{$CG(\mathcal{T}_{2D})_m^2$~\eqref{eq:proj_path_C}}
    \end{subfigure}%
    \begin{subfigure}[t]{0.1\textwidth}
        \includegraphics[height=0.20\textheight]{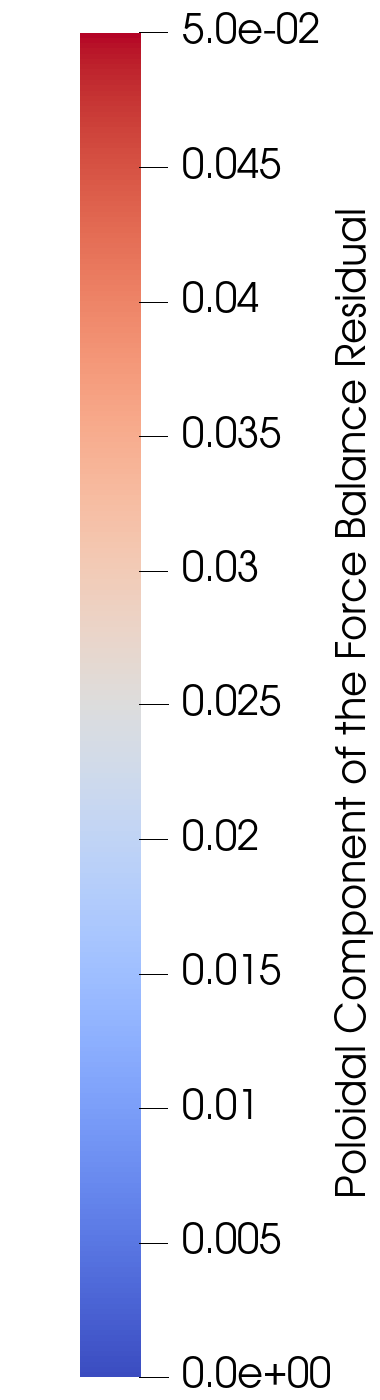}
    \end{subfigure}

    \vspace{1em}

    \begin{subfigure}[t]{0.22\textwidth}
        \centering
        \includegraphics[height=0.20\textheight]{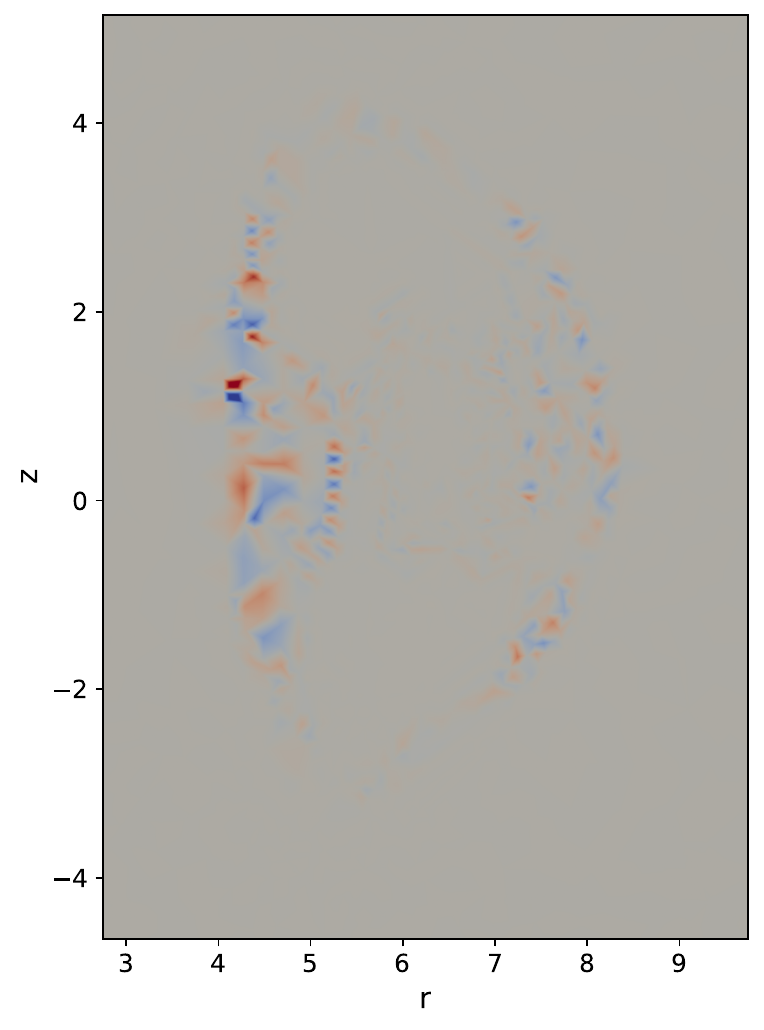}
        \caption{$CG(\mathcal{T}_{2D})_{m}$~\eqref{eq:proj_path_A}}
    \end{subfigure}%
    \begin{subfigure}[t]{0.22\textwidth}
        \centering
        \includegraphics[height=0.20\textheight]{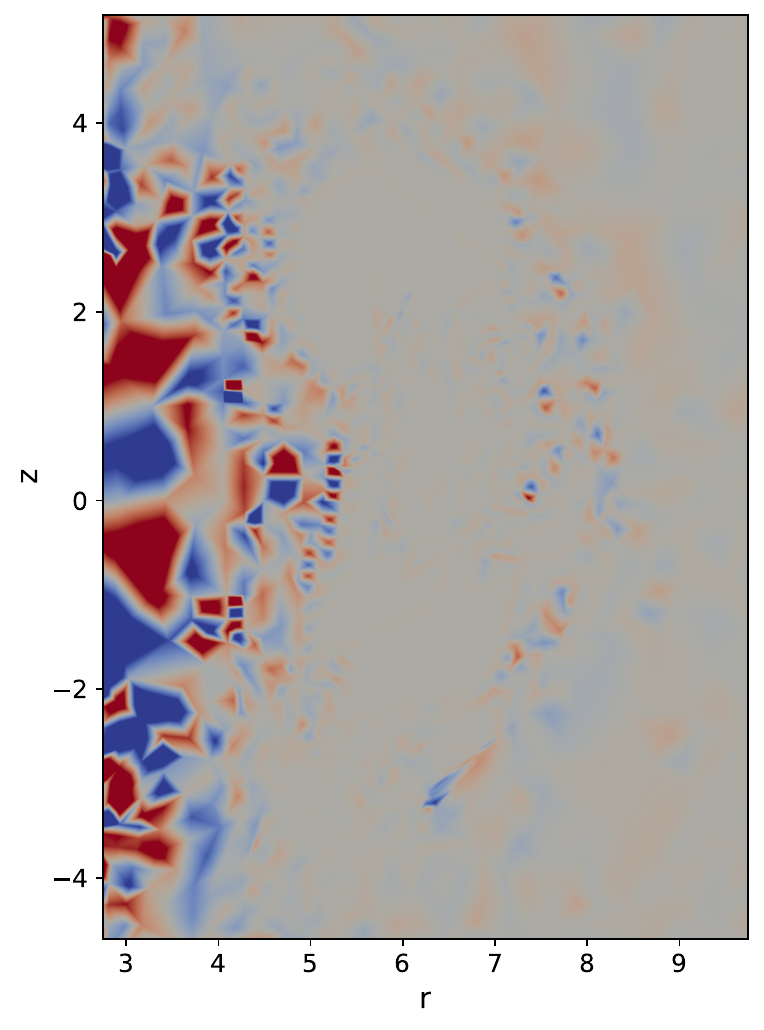}
        \caption{$CG(\mathcal{T}_{2D})_{m}$~\eqref{eq:proj_path_B}}
    \end{subfigure}%
    \begin{subfigure}[t]{0.22\textwidth}
        \centering
        \includegraphics[height=0.20\textheight]{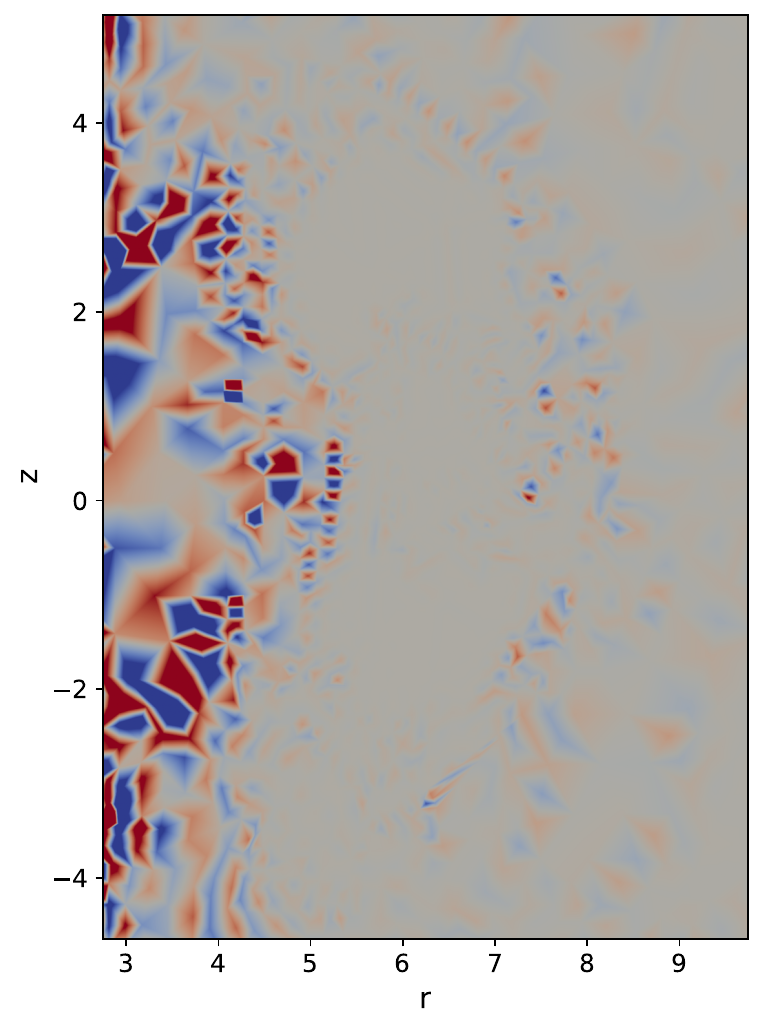}
        \caption{$CG(\mathcal{T}_{2D})_{m}$~\eqref{eq:proj_path_C}}
    \end{subfigure}%
    \begin{subfigure}[t]{0.1\textwidth}
        \includegraphics[height=0.20\textheight]{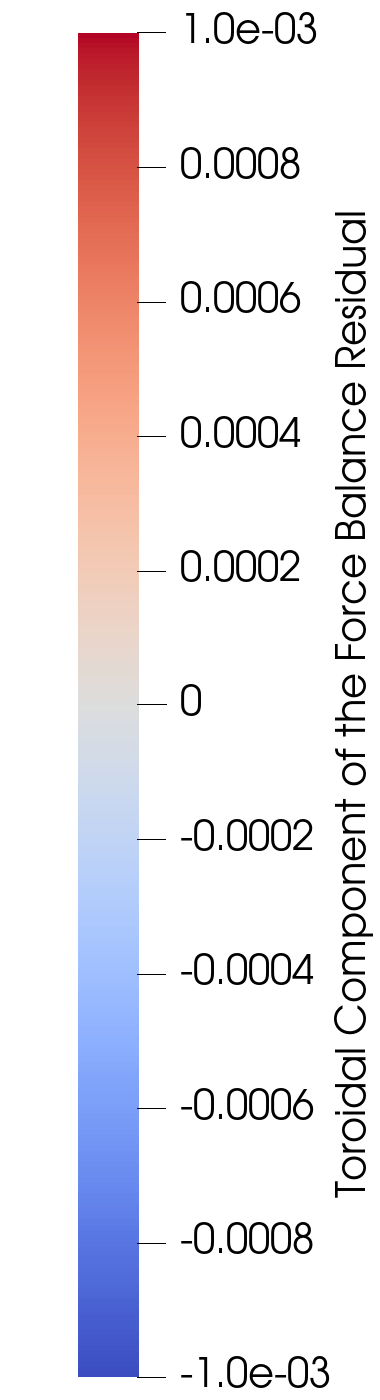}
    \end{subfigure}

    \caption{{Comparison of the force balance residual of ITER equilibrium from different projection paths in the corresponding poloidal and toroidal velocity spaces $\mathbb{V}_p$, $\mathbb{V}_t$, respectively.
        Top row: Poloidal component $[\mathbf{B}\times\mathbf{J}]_p + \beta \nabla p$.
        (a) $\mathbb{V}_p = H(\text{curl}, \mathcal{T}_{2D})_m$ from \eqref{eq:proj_path_A},
        (b) $\mathbb{V}_p = H(\text{curl}, \mathcal{T}_{2D})_m$ from \eqref{eq:proj_path_B},
        (c) $\mathbb{V}_p = CG(\mathcal{T}_{2D})_m^2$ from \eqref{eq:proj_path_C}.
        Bottom row: Toroidal component $[\mathbf{B}\times\mathbf{J}]_t$.
        (d) $\mathbb{V}_t = DG(\mathcal{T}_{2D})_{m-1}$ from \eqref{eq:proj_path_A},
        (e) $\mathbb{V}_t =  CG(\mathcal{T}_{2D})_{m}$ from \eqref{eq:proj_path_B},
        (f) $\mathbb{V}_t =  CG(\mathcal{T}_{2D})_{m}$ from \eqref{eq:proj_path_C}.}}
    \label{fig:JxB_comparison_Iter}
\end{figure}

{We first test these different paths with the Taylor state solution from \cite{serino2024adaptive}.} First, the resulting magnetic fields do not show visually distinguishable differences and are shown in~\ref{apx:B_field_proj_paths}. Second, we compare the current density fields obtained from the three projection paths, together with results from the direct computation \eqref{J_direct} as a reference. Since computing $\mathbf{J}$ directly from $f$ and $\Psi$ is not impacted by the discretization error involved in magnetic field projections, the results from \eqref{J_direct} should naturally contain less error. Figure~\ref{fig:J_comparison} shows the current density fields from the three projection paths and the reference results.

The first row depicts the poloidal component of the current density fields. Since there are two components, $J_r$ and $J_z$ in $\mathbf{J}_p$, we show the magnitude of the field. Within the plasma region (see Figure \ref{fig:2d_meshes} for the plasma domain described by the separatrix), all three projection paths yield results that are similar and are in relatively good agreement with the reference. In contrast, near the wall region and beyond, the projection paths \eqref{eq:proj_path_B} and \eqref{eq:proj_path_C} exhibit considerable noise. This noise is less important since the equilibrium is not evolved within the walls, and a higher quality mesh within the wall region will likely attenuate this noise.

In contrast, the toroidal component of the current density fields reveals more substantial discrepancies among the three projection paths. The second row shows the toroidal component of the current density fields from the reference result and the three projection paths above. This time, projection paths \eqref{eq:proj_path_B} and \eqref{eq:proj_path_C} also exhibit significant noise within the plasma region, which cannot be trivially eliminated. In contrast, the results from projection path \eqref{eq:proj_path_A} closely match the reference solution and remain largely free of noise. In particular, \eqref{eq:proj_path_A} follows the correct compatible finite element chain for computing $\mathbf{B}_p$ and $J_t$, using a strong perpendicular gradient to compute the divergence-conforming field $\mathbf{B}_p$ from the CG field $\Psi$, and a weak perpendicular gradient to compute the CG field $J_t$ from $\mathbf{B}_p$. On the other hand, such a compatible framework does not hold true for \eqref{eq:proj_path_B}, in which $\Psi$ is first implicitly projected into the DG space (within \eqref{discr_B_p_Hcurl}) when computing the curl-conforming field $\mathbf{B}_p$. Similarly, the vector-CG based approach \eqref{eq:proj_path_C} does not follow a framework of compatible function spaces with respect to the involved differential operators. Overall, this underlines the importance of a compatible choice of finite element spaces when computing the current density field as a diagnostic of the initial magnetic field $\mathbf{B}$ in the context of MHD solvers.

We further verify our findings with the force balance equation~\eqref{equilibrium}. {Since $\nabla p = 0$ in the Taylor state equilibrium}, both the toroidal and poloidal components of the Lorentz force should be zero. Figure~\ref{fig:JxB_comparison_Taylor} shows the Lorentz force from the same three projection paths as considered for the current density. The first row shows the poloidal component. Again, we present the magnitude of $[\mathbf{B}\times\mathbf{J}]_p$ for better visualization. Similarly to our findings for the current density, the projection path \eqref{eq:proj_path_A} gives close-to-zero $[\mathbf{B}\times\mathbf{J}]_p$ magnitudes except for the regions close to the separatrix, while the projection paths \eqref{eq:proj_path_B} and \eqref{eq:proj_path_C} lead to significant noise in the entire domain, including parts of the plasma region away from the separatrix. {The remaining imperfect force balance in Figure~\ref{fig:JxB_comparison_Taylor1}  is due to the approximation of the plasma domain for the source term in the GS solver. In particular, the separatrix in our GS solver is not allowed to cut through an element: each element is classified as either inside or outside the plasma domain based on the computed separatrix value.}

This is also evident in the toroidal component of the Lorentz force, shown in the second row. The projection path \eqref{eq:proj_path_B} results in significant noise while the projection path \eqref{eq:proj_path_A} contains much less noise in all regions except for the separatrix. Altogether, given these current density and Lorentz force results, the projection path \eqref{eq:proj_path_A} -- which follows the compatible finite element framework for computations based on the CG field $\Psi$ -- clearly leads to the smallest error, which is confined to the separatrix region.

{Next we consider the three projection paths with an ITER equilibrium. Figure~\ref{fig:JxB_comparison_Iter} shows the residual of the force balance equation for the ITER equilibrium, where the toroidal component is again given by $[\mathbf{B}\times\mathbf{J}]_t$ and the poloidal component is given by $[\mathbf{B}\times\mathbf{J}]_p + \beta \nabla p$ to account for the non-zero pressure gradient. Similar to the Taylor state equilibrium, projection path \eqref{eq:proj_path_A} leads to significantly less noise in the force balance residual compared to projection paths \eqref{eq:proj_path_B} and \eqref{eq:proj_path_C}, which is consistent with our findings from the Taylor state equilibrium.}

Although projection path \eqref{eq:proj_path_A} yields the least error in terms of force balance, projection path \eqref{eq:proj_path_B} may still be preferable in some cases, as projecting $\mathbf{B}_p$ in $H(\text{curl}, \mathcal{T}_{2D})$ provides additional advantages in maintaining the divergence-free property of the magnetic field. This holds true especially if $\Psi$ is provided by the GS solver as a DG field, in which case the $\Psi$-related computations in \eqref{eq:proj_path_B} also fully follow the compatible finite element framework, based on the weak magnetic field computation \eqref{discr_B_p_Hcurl}. In this work, we only have access to a GS solver that returns $\Psi$ as a CG field, although upgrades to also support $\Psi$ in DG are under way and we aim to also discuss corresponding GS-to-MHD field transfer results in future work. In particular, as shown in Figure~\ref{fig:div_B_pol_comparison_taylor}, the divergence from $\mathbf{B}_p$ in $H(\text{curl}, \mathcal{T}_{2D})$ computed as \eqref{div_B_weak} through projection path \eqref{eq:proj_path_B} is negligible in the entire region (see also Figure \ref{fig:div_B_pol_per_refine} in \ref{apx:div_B_pol_per_refine}), while the divergences from $\mathbf{B}_p$ in $H(\text{div}, \mathcal{T}_{2D})$ and $CG(\mathcal{T}_{2D})_m^2$ from projection paths \eqref{eq:proj_path_A} and \eqref{eq:proj_path_C}, respectively, contain significant noise throughout the domain. {Similar results are observed for the ITER equilibrium, as shown in Figure~\ref{fig:div_B_pol_comparison_iter}.}
\begin{figure}[!htbp]
    \centering
    \begin{subfigure}[t]{0.22\textwidth}
        \centering
        \includegraphics[height=0.20\textheight]{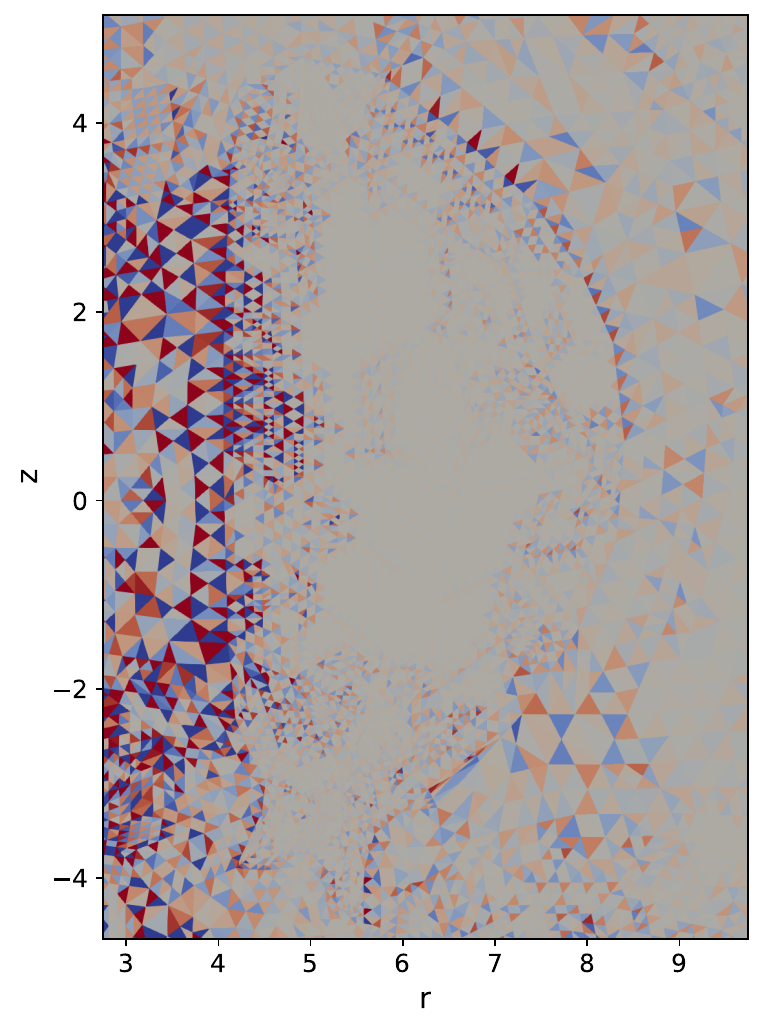}
        \caption{$DG(\mathcal{T}_{2D})_{m-1}$~\eqref{eq:proj_path_A}}
    \end{subfigure}%
    \begin{subfigure}[t]{0.22\textwidth}
        \centering
        \includegraphics[height=0.20\textheight]{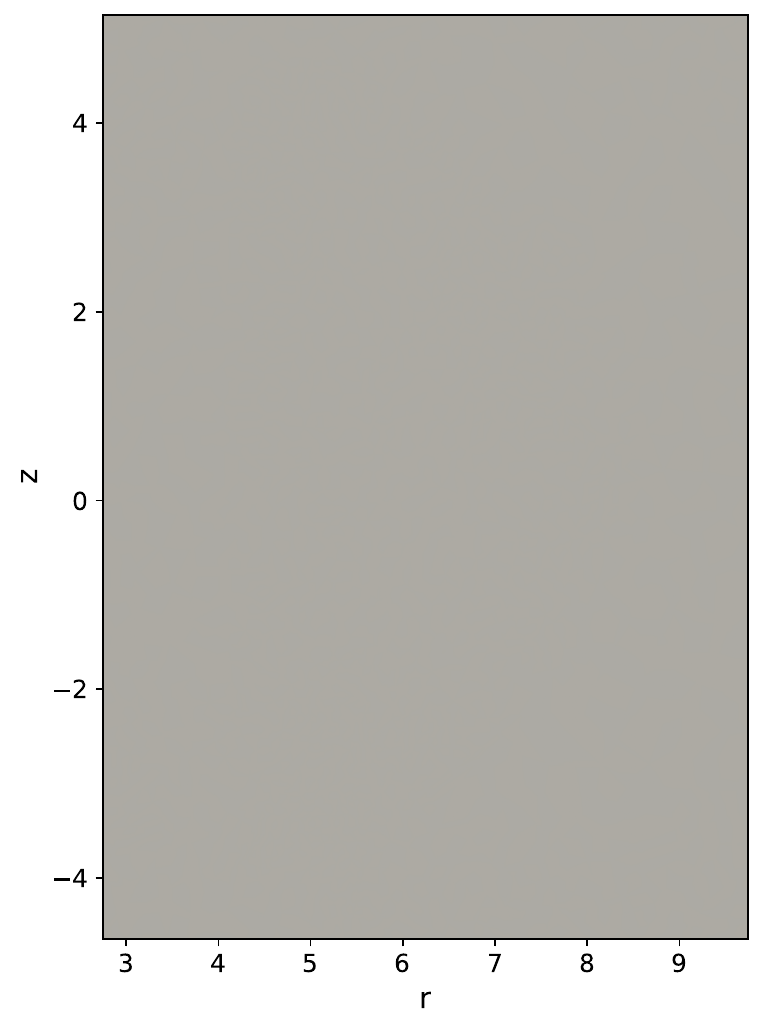}
        \caption{$CG(\mathcal{T}_{2D})_{m}$~\eqref{eq:proj_path_B}}
    \end{subfigure}%
    \begin{subfigure}[t]{0.22\textwidth}
        \centering
        \includegraphics[height=0.20\textheight]{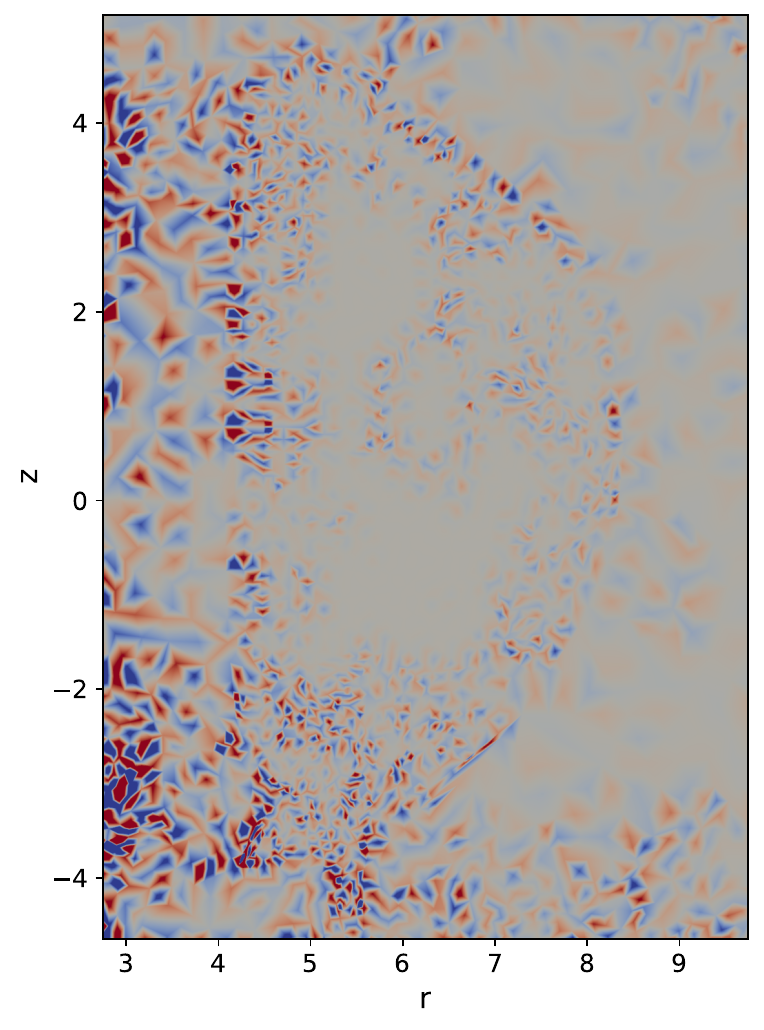}
        \caption{$CG(\mathcal{T}_{2D})_{m}$~\eqref{eq:proj_path_C}}
    \end{subfigure}%
    \begin{subfigure}[t]{0.1\textwidth}
        \centering
        \includegraphics[height=0.20\textheight]{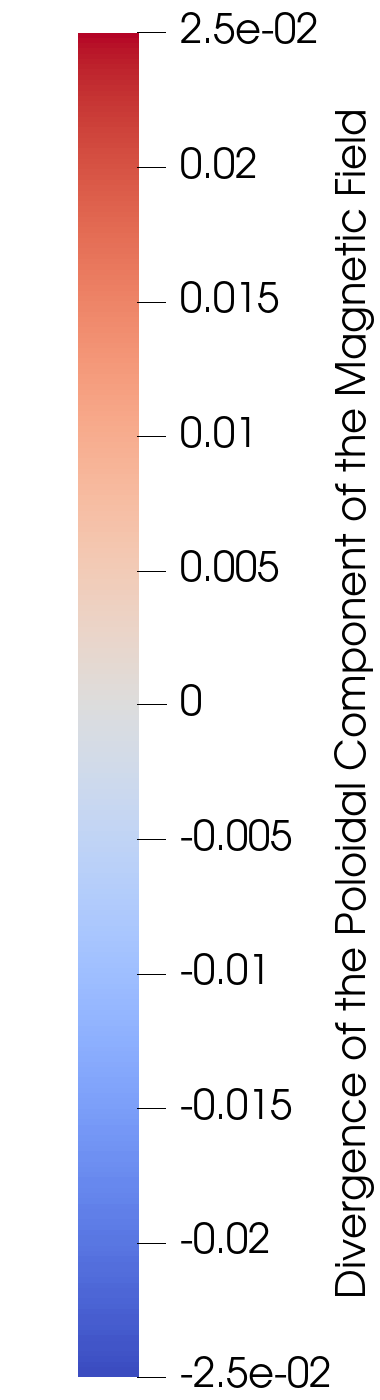}
    \end{subfigure}
    \caption{Comparison of the divergence of the poloidal component of the magnetic field ($D_b$) {of Taylor equilibrium} from different projection paths, with ($D_b$) in different spaces depending on the projection path.
        (a) $D_b \in DG(\mathcal{T}_{2D})_{m-1}$ from \eqref{eq:proj_path_A},
        (b) $D_b \in CG(\mathcal{T}_{2D})_{m}$ from \eqref{eq:proj_path_B},
        (c) $D_b \in CG(\mathcal{T}_{2D})_{m}$ from \eqref{eq:proj_path_C}.}
    \label{fig:div_B_pol_comparison_taylor}
\end{figure}
\begin{figure}[!htbp]
    \centering
    \begin{subfigure}[t]{0.22\textwidth}
        \centering
        \includegraphics[height=0.20\textheight]{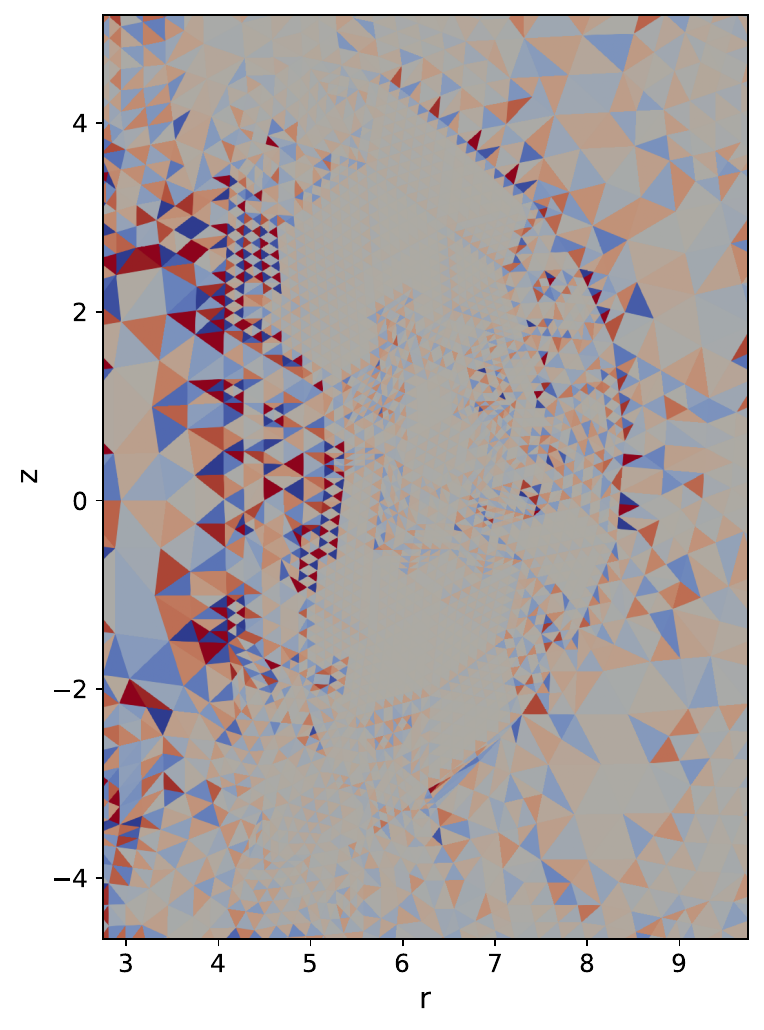}
        \caption{$DG(\mathcal{T}_{2D})_{m-1}$~\eqref{eq:proj_path_A}}
    \end{subfigure}%
    \begin{subfigure}[t]{0.22\textwidth}
        \centering
        \includegraphics[height=0.20\textheight]{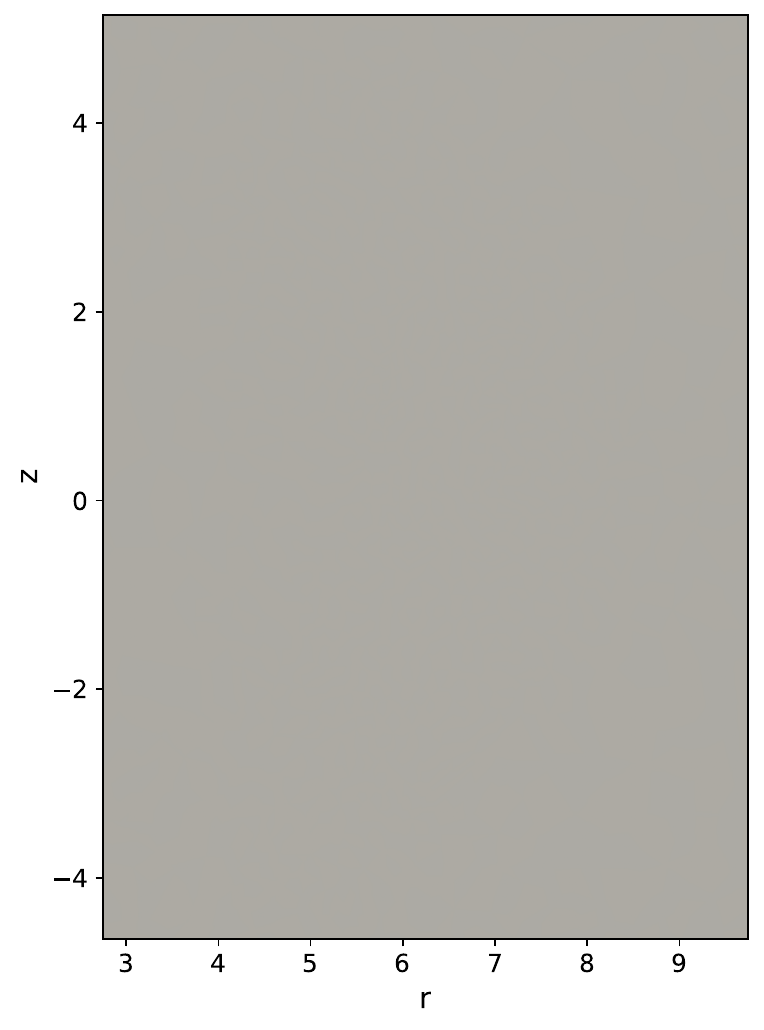}
        \caption{$CG(\mathcal{T}_{2D})_{m}$~\eqref{eq:proj_path_B}}
    \end{subfigure}%
    \begin{subfigure}[t]{0.22\textwidth}
        \centering
        \includegraphics[height=0.20\textheight]{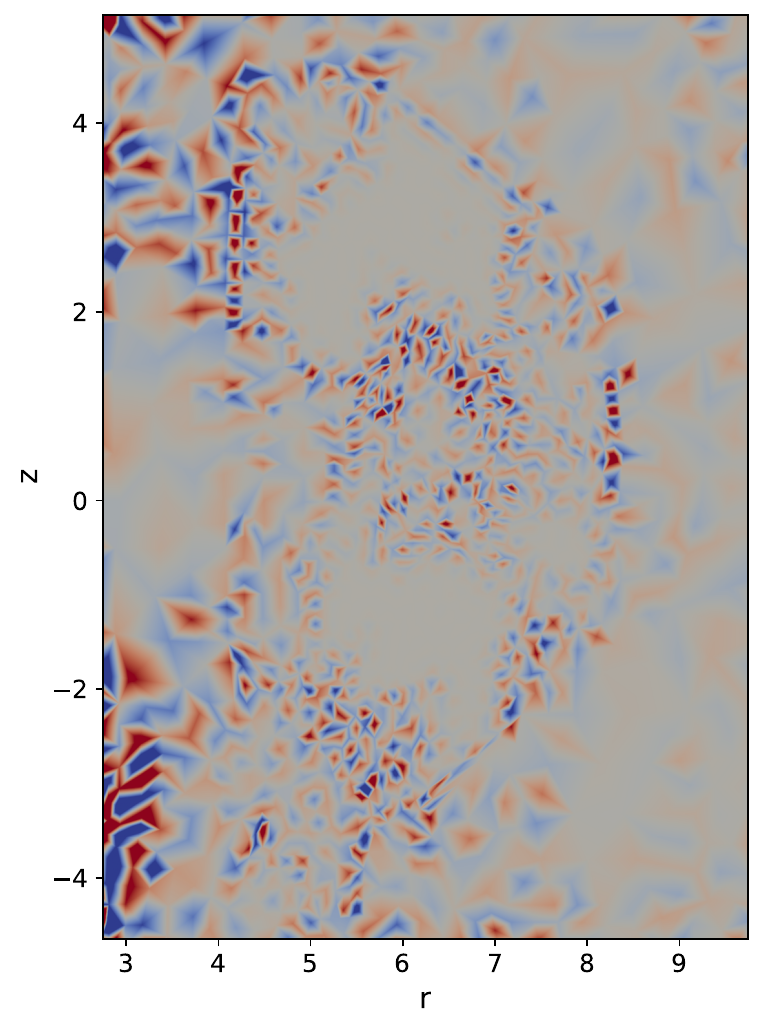}
        \caption{$CG(\mathcal{T}_{2D})_{m}$~\eqref{eq:proj_path_C}}
    \end{subfigure}%
    \begin{subfigure}[t]{0.1\textwidth}
        \centering
        \includegraphics[height=0.20\textheight]{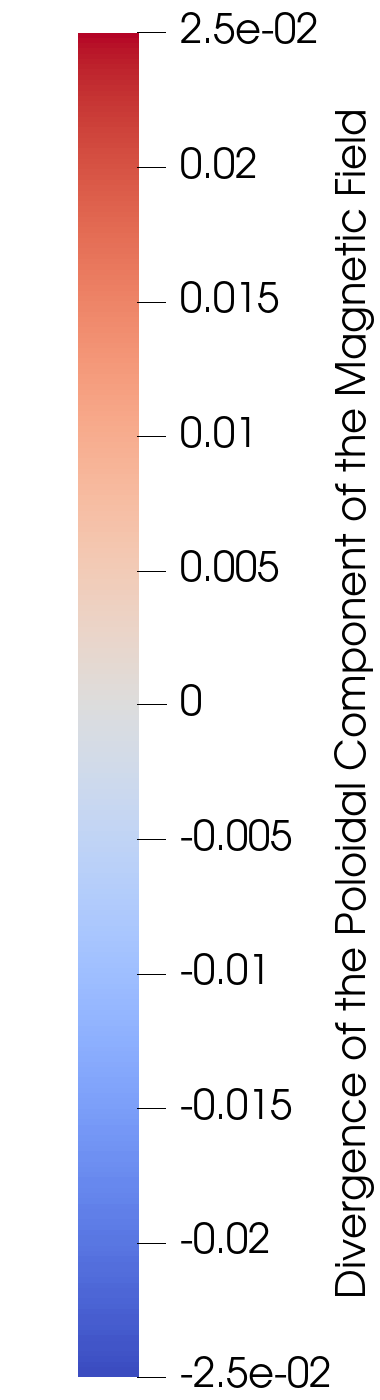}
    \end{subfigure}
    \caption{{Comparison of the divergence of the poloidal component of the magnetic field ($D_b$) of ITER equilibrium from different projection paths, with ($D_b$) in different spaces depending on the projection path.
        (a) $D_b \in DG(\mathcal{T}_{2D})_{m-1}$ from \eqref{eq:proj_path_A},
        (b) $D_b \in CG(\mathcal{T}_{2D})_{m}$ from \eqref{eq:proj_path_B},
        (c) $D_b \in CG(\mathcal{T}_{2D})_{m}$ from \eqref{eq:proj_path_C}.}}
    \label{fig:div_B_pol_comparison_iter}
\end{figure}
Therefore, given $\Psi$ in CG, the natural choice is to compute $\mathbf{B}_p$ in a divergence-conforming function space, which best eliminates the error induced by the implicit projection. {Although both projection paths \eqref{eq:proj_path_A} and \eqref{eq:proj_path_B} involve an implicit projection, the projection associated with \eqref{eq:proj_path_B} induces larger impact to the solution, which is evident from the force balance results above. This is expected given that the path from $\Psi$ to $J_t$ involves two differentiations while the path from $f$ to $\mathbf{J}_p$ involves only one differentiation. Consequently, the error in $J_t$ is more sensitive to the implicit projection error from $CG(\mathcal{T}_{2D})_m$ to $DG(\mathcal{T}_{2D})_{m-1}$.} However, divergence-conforming $\mathbf{B}_p$ does not offer the additional advantage of maintaining the divergence-free property of the magnetic field. In order to minimize the error in both aspects, it is necessary to compute $\Psi$ in DG from the GS solver so that we can obtain curl-conforming $\mathbf{B}_p$ without additional errors introduced by the implicit projection of $\Psi$ into DG.

\subsection{Impact of mesh misalignment}
\begin{figure}[htb]
    \centering
    \begin{subfigure}[t]{0.22\textwidth}
        \centering
        \includegraphics[height=0.20\textheight]{figs/GS_GS_align/JxB_pol_A_ax.pdf}
        \caption{original}
    \end{subfigure}%
    \begin{subfigure}[t]{0.22\textwidth}
        \centering
        \includegraphics[height=0.20\textheight]{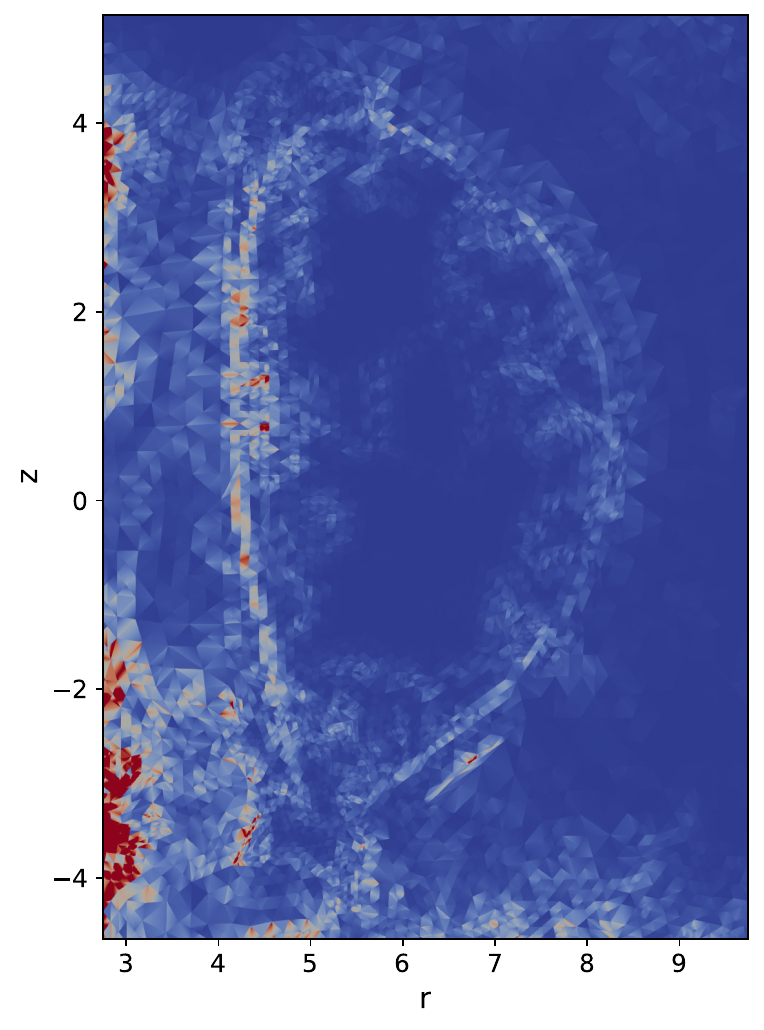}
        \caption{w/ perturbations}
    \end{subfigure}%
    \begin{subfigure}[t]{0.22\textwidth}
        \centering
        \includegraphics[height=0.20\textheight]{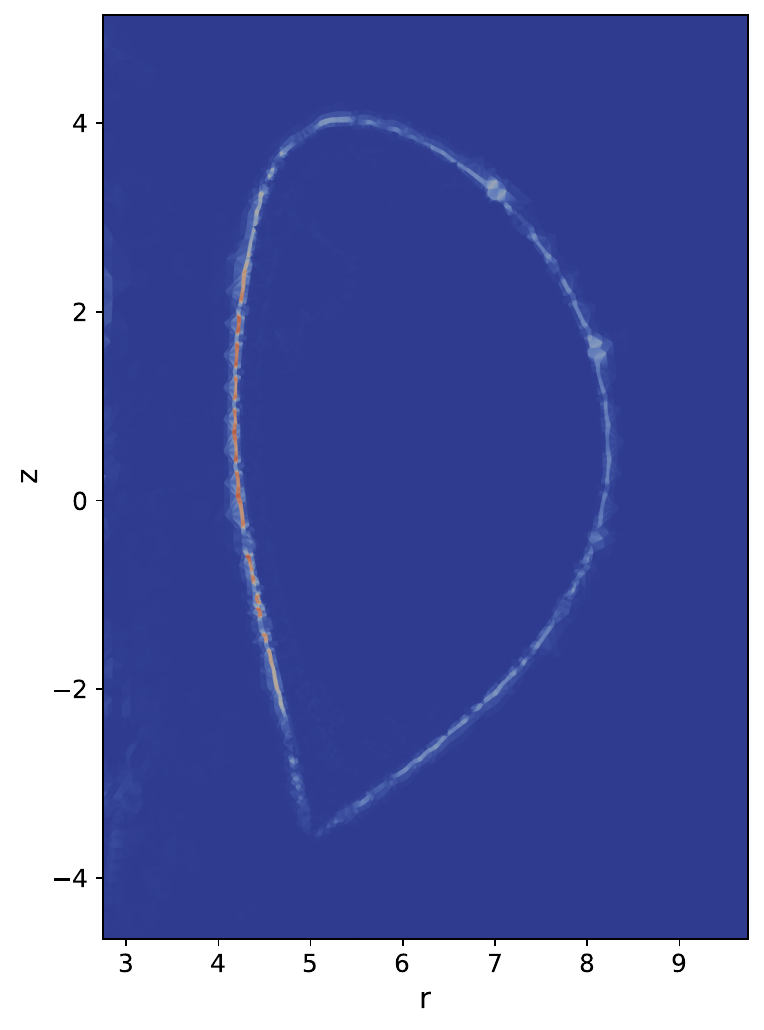}
        \caption{w/ refinement}
    \end{subfigure}%
    \begin{subfigure}[t]{0.1\textwidth}
        \centering
        \includegraphics[height=0.20\textheight]{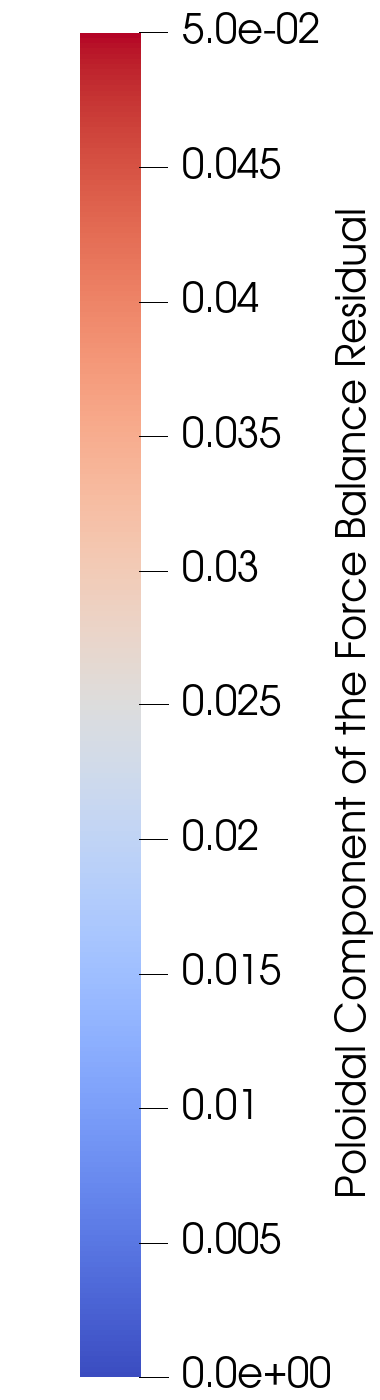}
    \end{subfigure}

    \vspace{1em}

    \begin{subfigure}[t]{0.22\textwidth}
        \centering
        \includegraphics[height=0.20\textheight]{figs/GS_GS_align/JxB_tor_A_ax.pdf}
        \caption{original}
    \end{subfigure}%
    \begin{subfigure}[t]{0.22\textwidth}
        \centering
        \includegraphics[height=0.20\textheight]{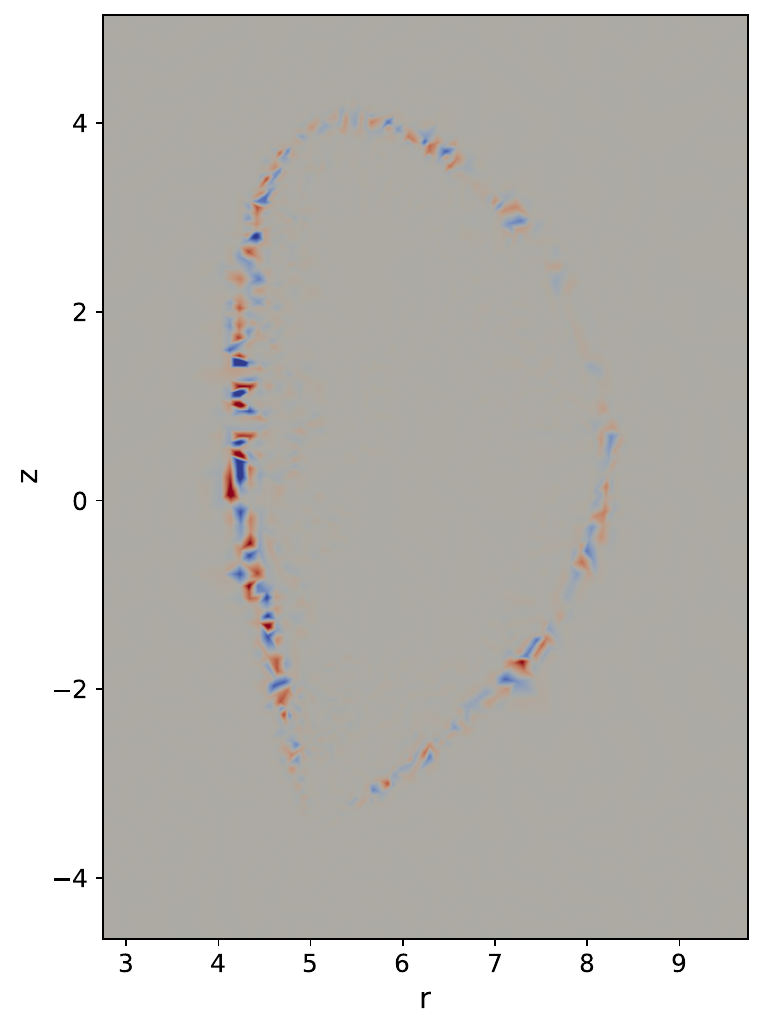}
        \caption{w/ perturbations}
    \end{subfigure}%
    \begin{subfigure}[t]{0.22\textwidth}
        \centering
        \includegraphics[height=0.20\textheight]{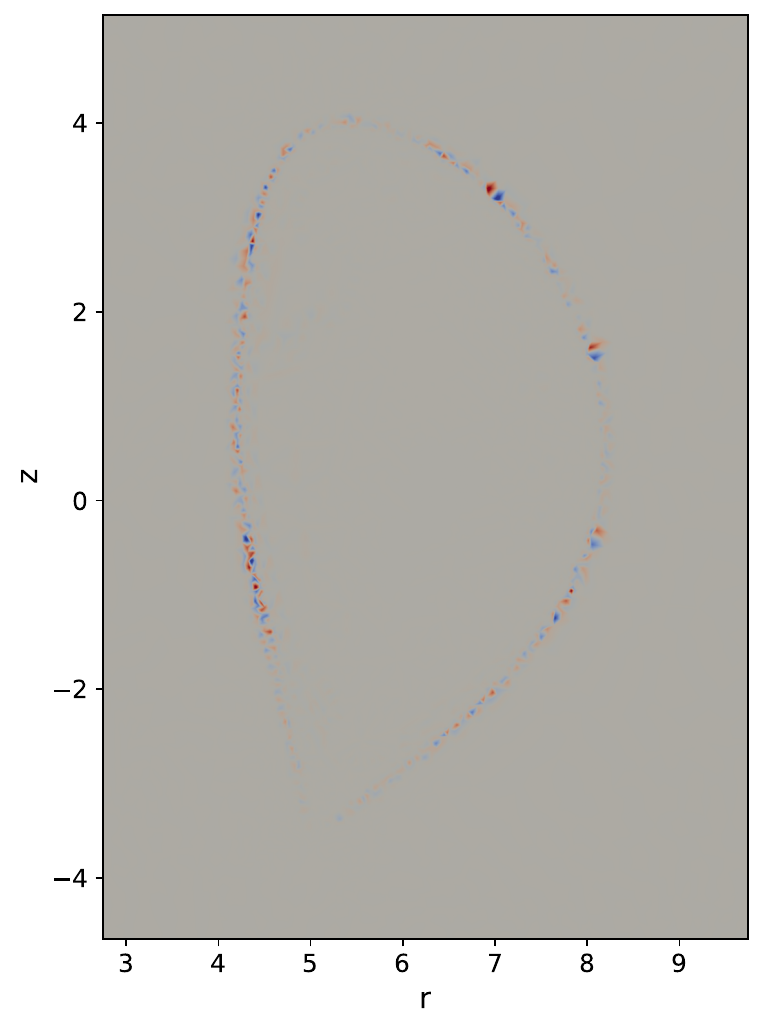}
        \caption{w/ refinement}
    \end{subfigure}%
    \begin{subfigure}[t]{0.1\textwidth}
        \centering
        \includegraphics[height=0.20\textheight]{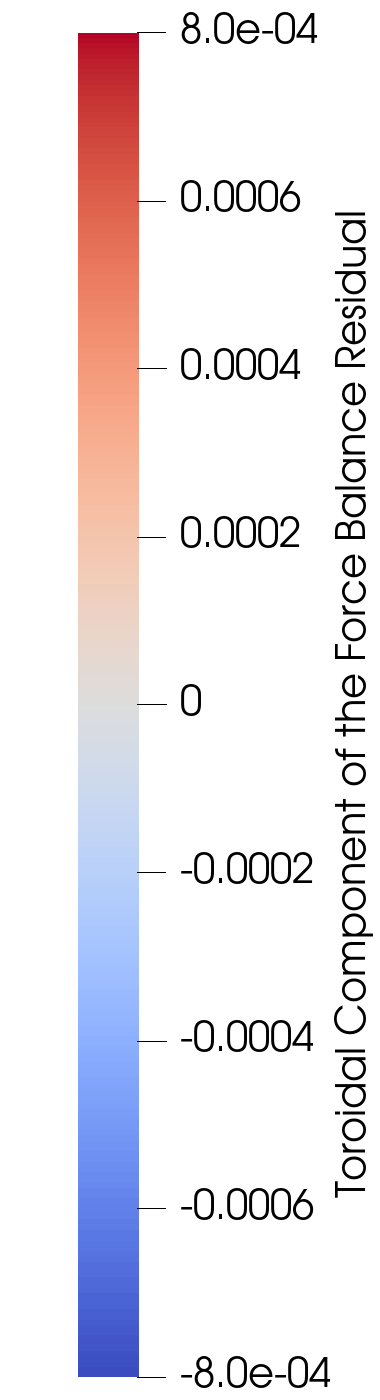}
    \end{subfigure}

    \caption{Comparison of the Lorentz force components {of Taylor equilibrium} from projection path~\eqref{eq:proj_path_A}.
        Top row: poloidal component $[\mathbf{B}\times\mathbf{J}]_p$.
        (a) original,
        (b) w/ perturbations,
        (c) w/ refinement.
        Bottom row: toroidal component $[\mathbf{B}\times\mathbf{J}]_t$.
        (e) original,
        (f) w/ perturbations,
        (g) w/ refinement.}
    \label{fig:JxB_per_refine}
\end{figure}

We further analyze the impact of mesh misalignment on the numerical results, using the GS solver output from~\cite{serino2024adaptive}. The comparison between the Lorentz force projected on the aligned mesh and the perturbed misaligned mesh is shown in Figure~\ref{fig:JxB_per_refine} (b, e). We show the results {of Taylor equilibrium} from projection path \eqref{eq:proj_path_A}, which we showed above to be the path yielding the least noise in the aligned case. As expected and in line with previous observations in the literature, the results show that the misalignment leads to significant additional noise in the entire domain. In our case, we find this to be true especially in the poloidal component of the Lorentz force. In addition to the Lorentz force, we note that results for the divergence field are expected to be analogous in the non-mesh-aligned case to those in the aligned one. This holds true in particular for the weakly divergence-free projection path \eqref{eq:proj_path_B}, which still leads to an approximately divergence-free magnetic field $\mathbf{B}$ (see~\ref{apx:div_B_pol_per_refine}). The reason behind this is that the divergence-free property relies on the identity $\nabla^\perp \cdot \nabla \equiv 0$ as well as a consistency condition of $g_1$ along the boundary (see Proposition \ref{prop_zero_div}), both of which hold true independently of the underlying choice of mesh.

\subsection{Impact of mesh refinement}

Finally, we analyze the impact of mesh refinement on the numerical results. In particular, we refine the mesh near the separatrix, which often contains steep gradients and also contains the main remaining error for the Lorentz force in our above results. This leverages one of the adaptivity strategies presented in our adaptive GS solver~\cite{serino2024adaptive}. Figure~\ref{fig:JxB_per_refine} (c, f) shows the comparison between the Lorentz force computed on meshes with and without refinement. Again, we show the results from projection path \eqref{eq:proj_path_A}. As expected, results show that refinement at the separatrix significantly reduces the error in all components of the Lorentz force.

Finally, we note that we have also conducted experiments on aligning mesh edges with the separatrix. Results indicate that such alignment does not further improve the accuracy as long as mesh refinement has been conducted along the separatrix. It is unclear why alignment with the separatrix does not lead to further improvements, and the final remaining error as seen in Figure~\ref{fig:JxB_per_refine} may be related to the projection of the CG field $f$ into the DG field $B_t$ within projection path \eqref{eq:proj_path_A}. As mentioned before, in this work, we only have access to a GS solver that returns CG fields, including $f$. In future work, we aim to also discuss corresponding GS-to-MHD field transfer results in which $f$ is a DG field, while retaining $\Psi$ in DG (or vice versa), for fully compatible projection paths with divergence- or curl-conforming magnetic fields. We note that our result does not contradict the practical popularity of field-aligned meshes used in transient MHD simulations in MCF.  In recent studies \cite{green2022efficient, vogl2023mesh}, the mesh alignment is found to be useful for anisotropic diffusion in MHD under a strong magnetic field.



\section{Conclusion}
\label{sec:conclusion}
In this work, we have studied the numerical errors involved in transferring GS equilibria to MHD solvers. We have {discussed} several sources of such errors, including incompatibility between the finite element spaces used in the GS and MHD solvers, misalignment between the GS and MHD meshes, and possibly under-resolved strong gradients near the separatrix. We have derived several possible structure-preserving transferring strategies and have designed experiments to compare their numerical errors in terms of preserving the force balance and divergence-free properties of the magnetic field.

The incompatibility between the GS and MHD finite element spaces is examined within a compatible finite element framework. The compatible finite element framework implements a discretized de-Rham complex, where there is a natural choice of finite element space that corresponds to the differential operators involved in the projection path. To examine the incompatibility between the GS and MHD finite element spaces, we derive three candidate projection paths, among which two are compatible finite element space-based projection paths and one is based on vector-CG spaces, serving as a baseline. The numerical results show that the compatible finite element spaces perform significantly better than the vector-CG spaces in terms of MHD equilibrium errors and preserving the divergence-free property in the case of the curl-conforming magnetic field discretization. In particular, we find that given a GS solver output as a CG field, one of the proposed paths performs the best in terms of preserving the force balance, while the other path performs the best in terms of preserving the divergence-free property of the magnetic field, whereas the vector-CG spaces preserve neither property. 

We have also examined the impact of mesh misalignment between the GS and MHD meshes. We experimented with this by applying a small perturbation to the MHD mesh. The numerical results show that a small perturbation to the MHD mesh is enough to result in large numerical errors in the force balance. This indicates that it is important to align the MHD mesh exactly with the GS mesh in order to minimize the force imbalance in the transferred equilibria.
Finally, we have also examined the impact of the possibly under-resolved strong gradients near the separatrix on the numerical error. We find that mesh refinement along the separatrix is effective in {mitigating} the numerical error in the force balance. {To fully eliminate this error, one may consider either (1) supporting CutFEM in the GS solver for a plasma domain that is not known \emph{a priori}, or (2) incorporating mesh-alignment capability for the separatrix in the GS solver. We plan to address these directions in future work.}

Overall, our results highlight that compatible finite element spaces offer significant advantages in minimizing the numerical errors involved in transferring GS equilibria to MHD solvers compared to e.g., vector-CG spaces. The choice of compatible finite element setup in turn requires suitable spaces returned by the GS solver, with $\Psi$ in CG for divergence-conforming magnetic fields, and $\Psi$ in DG for curl-conforming magnetic fields, yet divergence-conforming magnetic fields lose the divergence-free property. In future work, we aim to further test our projection path with curl-conforming magnetic fields that are naturally divergence-free, using an updated GS solver that can also provide $\Psi$ as a DG field.




\appendix
\section{Derivation of component-wise field reconstructions in the 2D poloidal plane}
\label{apx:diff_ops_2D}
Recall that we have defined $\nabla = \left(\pp{}{r}, \pp{}{z}\right)^T$ and $(A_r, A_z)^\perp = (-A_z, A_r)$, and $\nabla_c \times (\cdot)$ denotes the curl in cylindrical coordinates. Thus, for an axisymmetric toroidal vector field $A_t \mathbf{e}_\phi = A_\phi \mathbf{e}_\phi$, we have
\begin{align}
  \begin{split}
    \nabla_c \times (A_t \mathbf{e}_\phi) & = -\pp{A_\phi}{z} \mathbf{e}_r + \frac{1}{r} \pp{(r A_\phi)}{r} \mathbf{e}_z                   = \frac{1}{r} \left(-\pp{(r A_\phi)}{z} \mathbf{e}_r + \pp{(r A_\phi)}{r} \mathbf{e}_z\right) \sim_\phi \; \frac{1}{r} \nabla^\perp (r A_t),
  \end{split}
\end{align}
where $A_t$ is the 2D embedding of $A_\phi$ into the $(r,z)$-plane. Further, for an axisymmetric poloidal vector field $\mathbf{A} = A_r \mathbf{e}_r + A_z \mathbf{e}_z$, we have
\begin{align}
  \begin{split}
    \nabla_c \times \mathbf{A} & =  \frac{1}{r} \frac{\partial A_z}{\partial \phi} \mathbf{e}_r + \left( \frac{\partial A_r}{\partial z} - \frac{\partial A_z}{\partial r} \right) \mathbf{e}_\phi - \frac{\partial A_r}{\partial \phi} \mathbf{e}_z  = -\left(\pp{A_z}{r} - \pp{A_r}{z}\right) \mathbf{e}_\phi \sim_\phi \; -\left( \nabla^\perp \cdot \mathbf{A}_p\right),
  \end{split}
\end{align}
where $\mathbf{A}_p$ is the component-wise 2D embedding of $\mathbf{A}$ into the $(r, z)$-plane, noting that $\mathbf{A}$ has no out-of-plane component. Note that we use the axisymmetry assumption here, such that $\pp{A_z}{\phi} = 0$ and $\pp{A_r}{\phi} = 0$.

For the Lorentz force components $[\mathbf{B} \times \mathbf{J}]_p$ and $[\mathbf{B} \times \mathbf{J}]_t$, we have
\begin{align}
    \mathbf{B} \times \mathbf{J} & = (B_r \mathbf{e}_r + B_\phi \mathbf{e}_\phi + B_z \mathbf{e}_z) \times (J_r \mathbf{e}_r + J_\phi \mathbf{e}_\phi + J_z \mathbf{e}_z)  \nonumber \\
                                 & = (B_\phi J_z - B_z J_\phi) \mathbf{e}_r + (B_z J_r - B_r J_z) \mathbf{e}_\phi + (B_r J_\phi - B_\phi J_r) \mathbf{e}_z               \nonumber   \\
                                 & = -B_\phi (-J_z \mathbf{e}_r + J_r \mathbf{e}_z) + J_\phi (-B_z \mathbf{e}_r + B_r \mathbf{e}_z) + (- B_r J_z + B_z J_r) \mathbf{e}_\phi, 
\end{align}
where the first two terms are the poloidal component and the third term is the toroidal component, that is
\begin{align}
   -B_\phi (-J_z \mathbf{e}_r + J_r \mathbf{e}_z) + J_\phi (-B_z \mathbf{e}_r + B_r \mathbf{e}_z) &\; \sim_\phi \; (-B_t \mathbf{J}_p^\perp  + J_t \mathbf{B}_p^\perp), \\
   (- B_r J_z + B_z J_r) \mathbf{e}_\phi &\; \sim_\phi \;( \mathbf{B}_p \cdot \mathbf{J}_p^\perp).
\end{align}
Thus, we have
\begin{align}
  [\mathbf{B} \times \mathbf{J}]_p & = -B_t \mathbf{J}_p^\perp + J_t \mathbf{B}_p^\perp,\\
  \hspace{1cm} [\mathbf{B} \times \mathbf{J}]_t & = \mathbf{B}_p \cdot \mathbf{J}_p^\perp.
\end{align}

Finally, we can derive $[\mathbf{B} \times \mathbf{J}]_p$ and $[\mathbf{B} \times \mathbf{J}]_t$ in terms of $B_t$ and $\mathbf{B}_p$ only, as
\begin{align}
  \begin{split}
    [\mathbf{B} \times \mathbf{J}]_p & = -B_t \mathbf{J}_p^\perp + J_t \mathbf{B}_p^\perp  = -B_t \left(\frac{1}{r}\nabla^\perp (rB_t)\right)^\perp -(\nabla^\perp \cdot \mathbf{B}_p) \mathbf{B}_p^\perp  = \frac{1}{r} B_t \nabla (rB_t) - (\nabla^\perp\cdot \mathbf{B}_p) \mathbf{B}_p^\perp,
  \end{split}
\end{align}
and
\begin{align}
  \begin{split}
    [\mathbf{B} \times \mathbf{J}]_t & = \mathbf{B}_p \cdot \mathbf{J}_p^\perp = \mathbf{B}_p \cdot \left(\frac{1}{r}\nabla^\perp (rB_t)\right)^\perp  = -\frac{1}{r} \left(\mathbf{B}_p \cdot \nabla(rB_t)\right).
  \end{split}
\end{align}

\section{Integration by parts for weak forms}
\label{apx:IBP}
Integration by parts was used to derive \eqref{discr_B_p_Hcurl} from \eqref{Bp_from_psi}. To see this, take the inner product of \eqref{Bp_from_psi} with a test function (before discretization, assuming suitably regular fields), and obtain
\begin{align}
    \left \langle \mathbf{\Sigma}, r\mathbf{B}_p \right \rangle = \left \langle \mathbf{\Sigma},  \nabla^{\perp} \Psi \right \rangle, \qquad \forall \mathbf{\Sigma} \in H(\text{curl}, \mathcal{T}_{2D})_m,
    \label{int_B_p}
\end{align}
where we include an additional factor of $r$ due to the Jacobian determinant when transforming the integration from cylindrical coordinates to the poloidal plane. We can then use the definition of the perpendicular operator (including its property $\mathbf{a}^\perp \cdot \mathbf{b} = - \mathbf{a} \cdot \mathbf{b}^\perp$ for any vectors $\mathbf{a}$, $\mathbf{b}$) together with integration by parts, which leads to
\begin{equation}
\left \langle \mathbf{\Sigma},  \nabla^{\perp} \Psi \right \rangle = - \left \langle \mathbf{\Sigma}^\perp,  \nabla \Psi \right \rangle = \left \langle \nabla \cdot \mathbf{\Sigma}^\perp,  \Psi \right \rangle - \int_{\partial \mathcal{T}_{2D}} \Psi(\mathbf{\Sigma}^\perp \cdot \mathbf{n}) \, dS = -\left \langle \nabla^\perp \cdot \mathbf{\Sigma}, \Psi \right \rangle + \int_{\partial \mathcal{T}_{2D}} (\mathbf{\Sigma} \cdot \mathbf{n}^\perp) \, dS.
\end{equation}
The weak computation of current density components can be derived in a similar manner.

\section{Comparison of different projection paths for magnetic field}
\label{apx:B_field_proj_paths}
Figure~\ref{fig:B_pol_comparison} shows the poloidal component of the magnetic field from the two projection paths. Since there are two components, $B_r$ and $B_z$ in $\mathbf{B}_p$, we show here the magnitude of the field. Overall there are no significant differences in the $\mathbf{B}_p$ itself from the two projection paths, except for some minor differences that reveal the properties of different finite element spaces.
\begin{figure}[!htbp]
  \centering
  \begin{subfigure}[t]{0.22\textwidth}
    \centering
    \includegraphics[height=0.20\textheight]{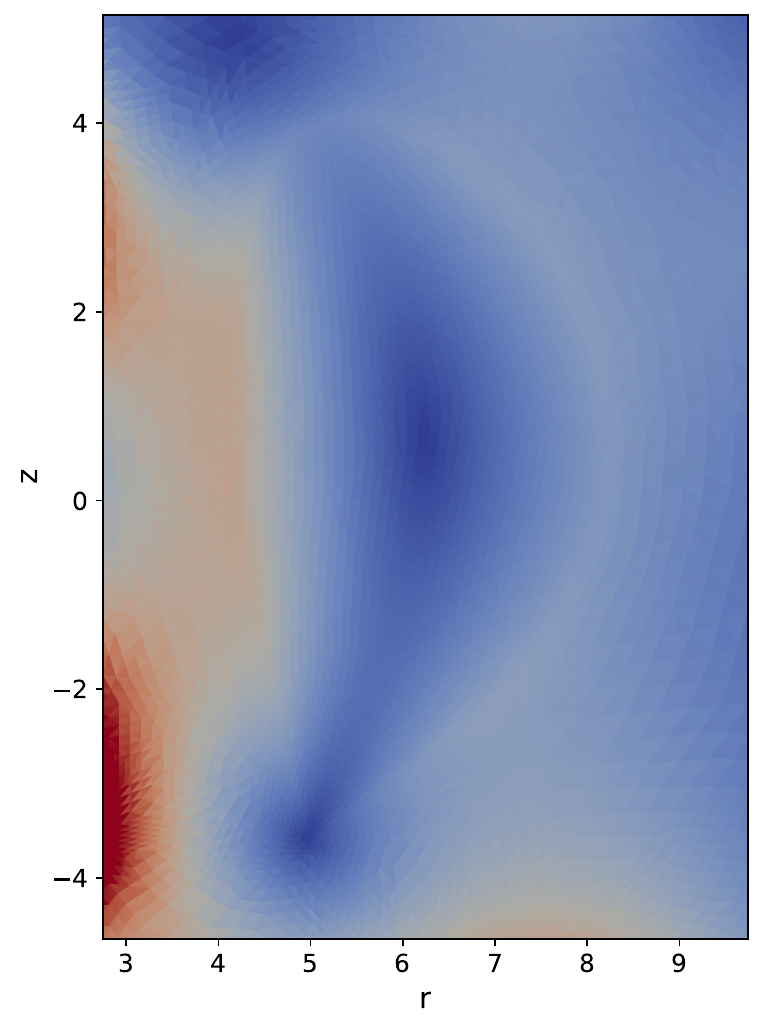}
    \caption{$H(\text{div}, \mathcal{T}_{2D})_m$~\eqref{eq:proj_path_A}}
  \end{subfigure}%
  \begin{subfigure}[t]{0.22\textwidth}
    \centering
    \includegraphics[height=0.20\textheight]{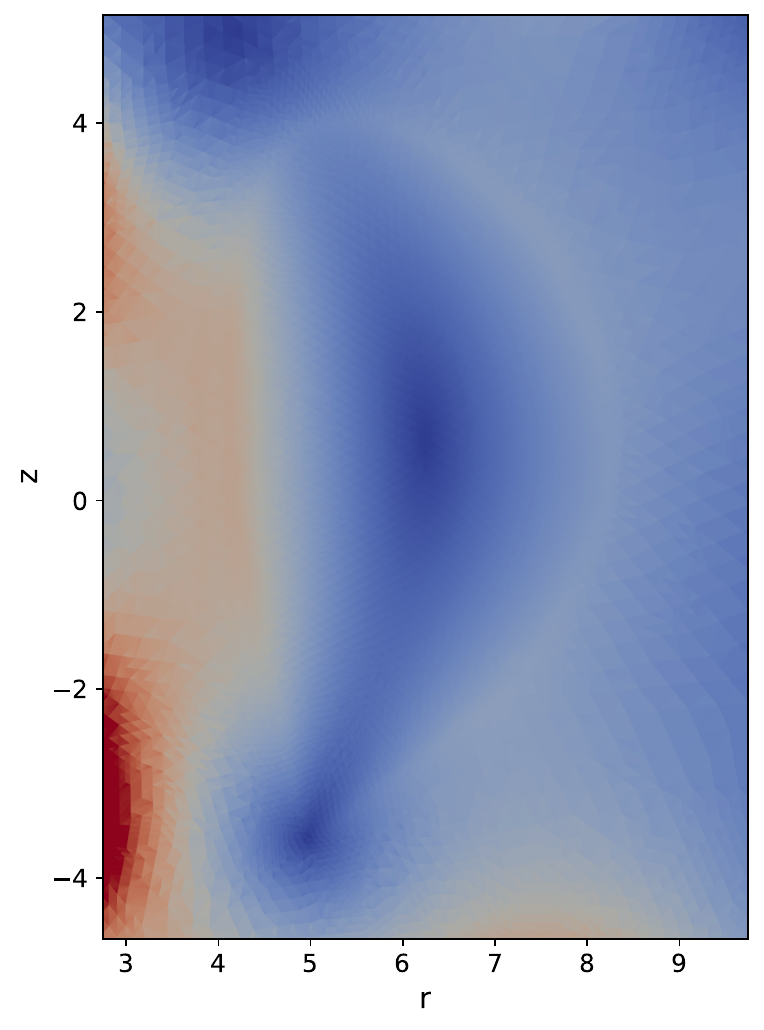}
    \caption{$H(\text{curl}, \mathcal{T}_{2D})_m$~\eqref{eq:proj_path_B}}
  \end{subfigure}%
  \begin{subfigure}[t]{0.22\textwidth}
    \centering
    \includegraphics[height=0.20\textheight]{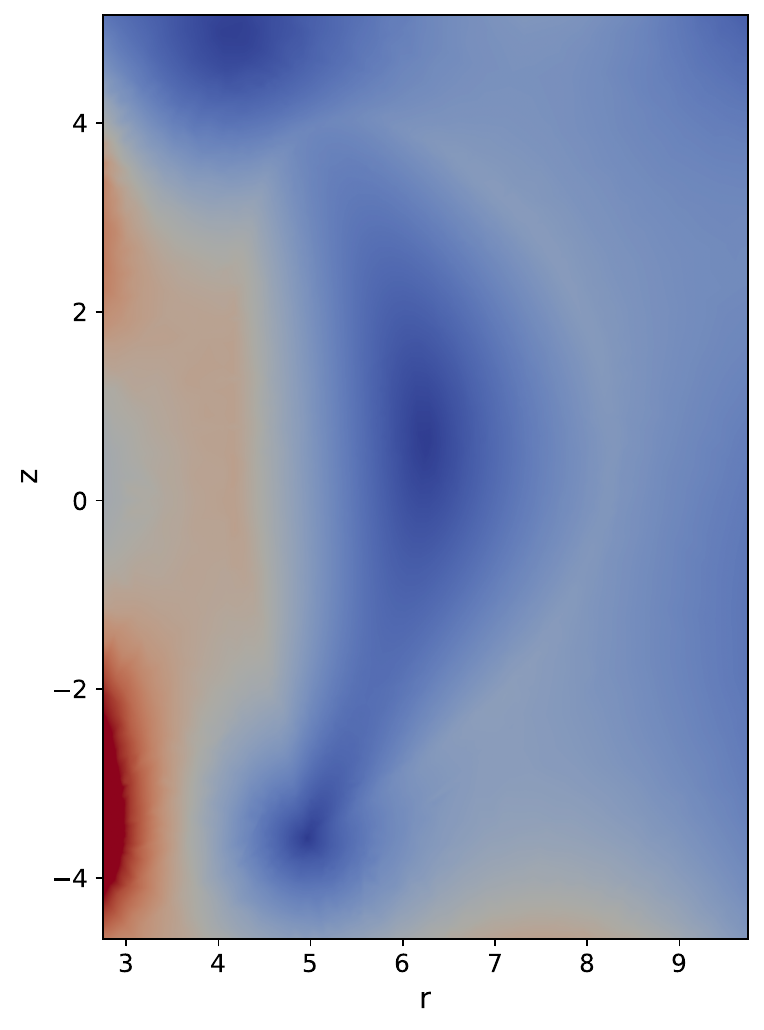}
    \caption{$CG(\mathcal{T}_{2D})_m^2$~\eqref{eq:proj_path_C}}
  \end{subfigure}%
  \begin{subfigure}[t]{0.1\textwidth}
    \centering
    \includegraphics[height=0.20\textheight]{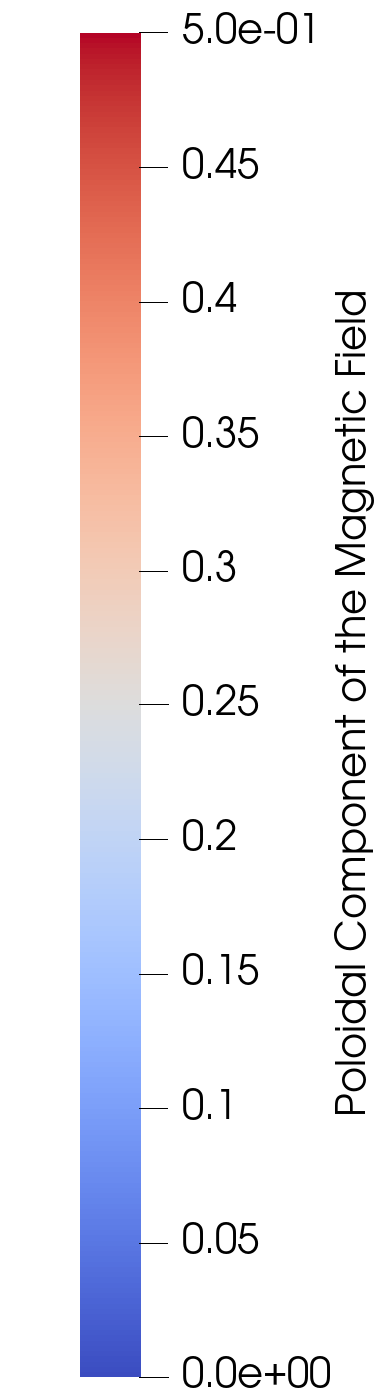}
  \end{subfigure}

  \caption{Comparison of the poloidal component of the magnetic field $\mathbf{B}_p$ {of Taylor equilibrium} from different projection paths.
    (a) $H(\text{div}, \mathcal{T}_{2D})_m$~\eqref{eq:proj_path_A},
    (b) $H(\text{curl}, \mathcal{T}_{2D})_m$~\eqref{eq:proj_path_B},
    (c) $CG(\mathcal{T}_{2D})_m^2$~\eqref{eq:proj_path_C}.  }
  \label{fig:B_pol_comparison}
\end{figure}

Figure~\ref{fig:B_tor_comparison} shows the toroidal component of the magnetic field from the two projection paths. Similarly to the poloidal component, the results do not show significant differences except for the enforced smoothness of the toroidal component in the CG space.
\begin{figure}[!htbp]
  \centering
  \begin{subfigure}[t]{0.22\textwidth}
    \centering
    \includegraphics[height=0.20\textheight]{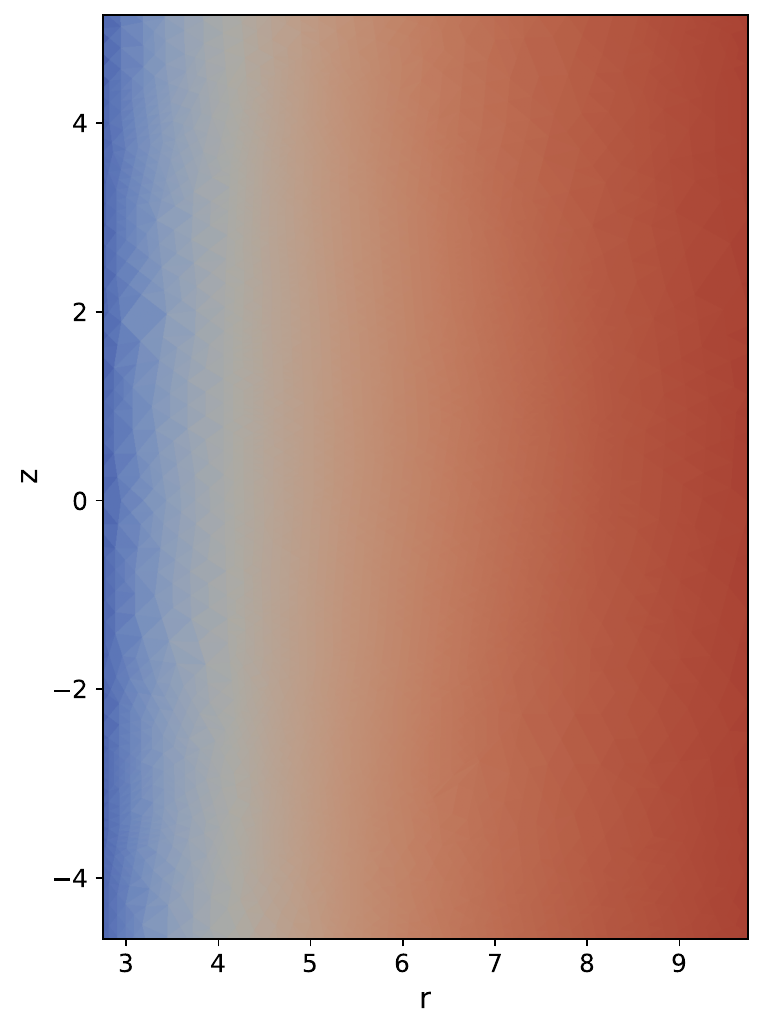}
    \caption{$DG(\mathcal{T}_{2D})_{m-1}$~\eqref{eq:proj_path_A}}
  \end{subfigure}%
  \begin{subfigure}[t]{0.22\textwidth}
    \centering
    \includegraphics[height=0.20\textheight]{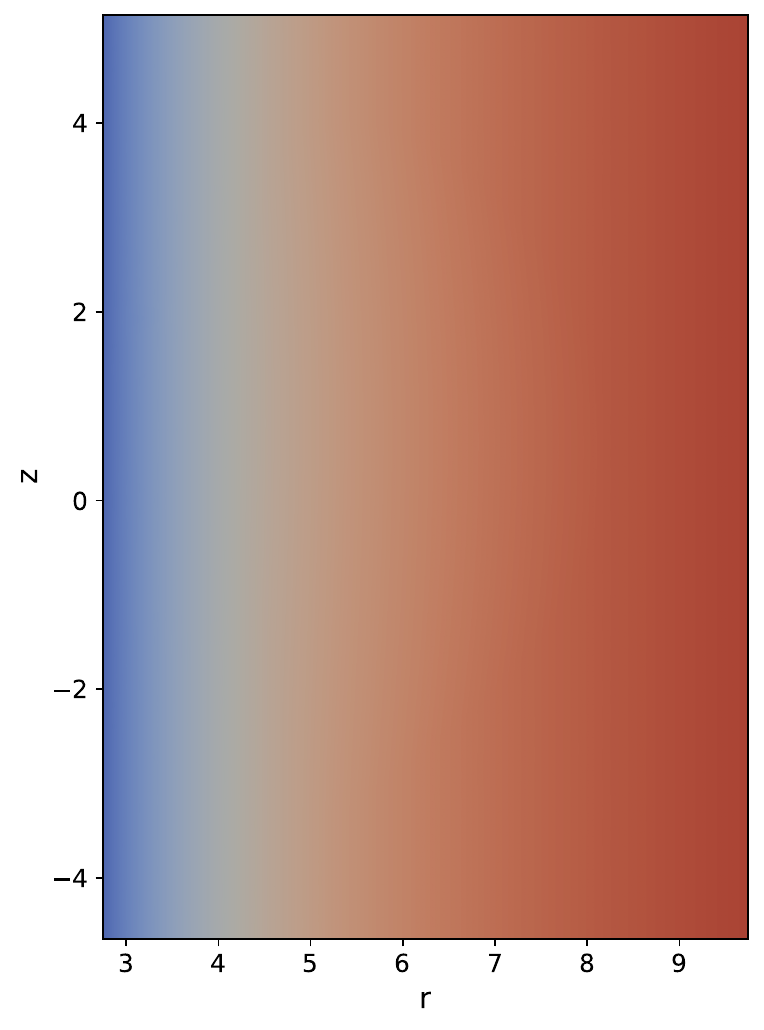}
    \caption{$CG(\mathcal{T}_{2D})_{m}$~(\ref{eq:proj_path_B}, \ref{eq:proj_path_C})}
  \end{subfigure}%
  \begin{subfigure}[t]{0.1\textwidth}
    \centering
    \includegraphics[height=0.20\textheight]{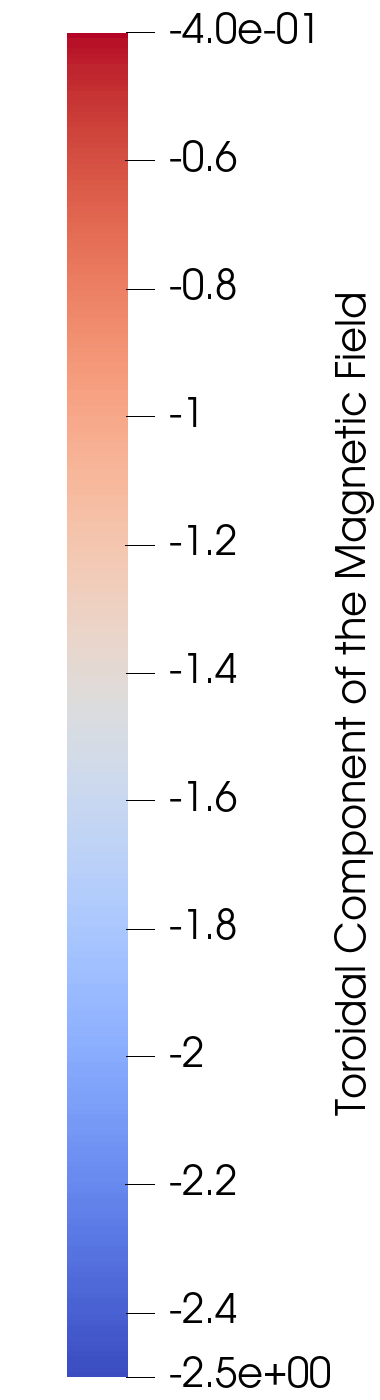}
  \end{subfigure}

  \caption{Comparison of the toroidal component of the magnetic field $B_t$ {of Taylor equilibrium} from different projection paths.
    (a) $DG(\mathcal{T}_{2D})_{m-1}$~\eqref{eq:proj_path_A},
    (b) $CG(\mathcal{T}_{2D})_{m}$~(\ref{eq:proj_path_B}, \ref{eq:proj_path_C}).}
  \label{fig:B_tor_comparison}
\end{figure}

\section{Comparison of divergence of the poloidal component of the magnetic field under perturbations and refinement}
\label{apx:div_B_pol_per_refine}

Figure~\ref{fig:div_B_pol_per_refine} shows the divergence of the poloidal component of the magnetic field ($D_b$) from projection path~\eqref{eq:proj_path_A} under mesh perturbations and refinement. Results show that neither perturbations nor refinement impact the divergence-free property of the poloidal magnetic field.
\begin{figure}[!htbp]
  \centering
  \begin{subfigure}[t]{0.22\textwidth}
    \centering
    \includegraphics[height=0.20\textheight]{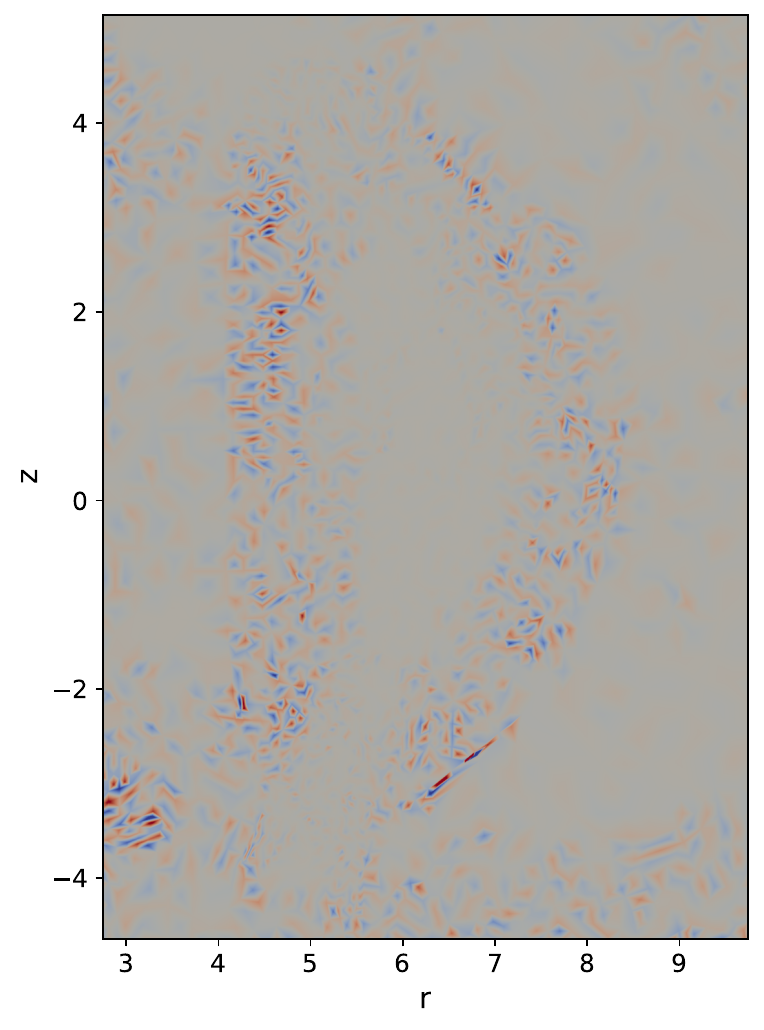}
    \caption{original}
  \end{subfigure}%
  \begin{subfigure}[t]{0.22\textwidth}
    \centering
    \includegraphics[height=0.20\textheight]{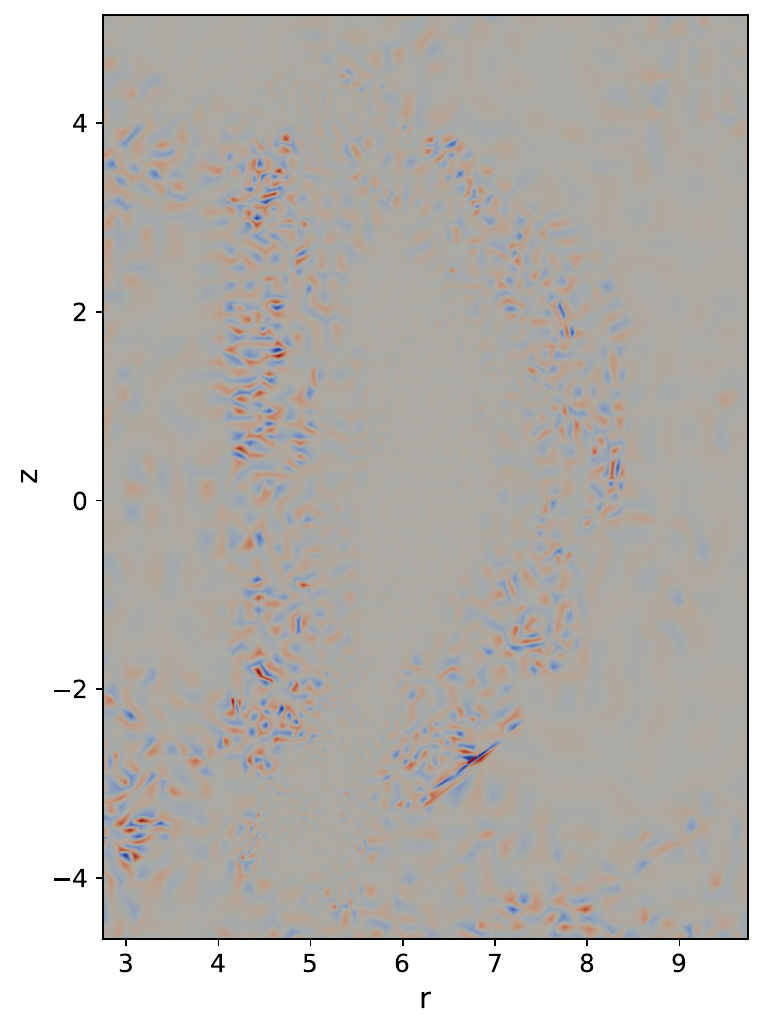}
    \caption{w/ perturbations}
  \end{subfigure}%
  \begin{subfigure}[t]{0.22\textwidth}
    \centering
    \includegraphics[height=0.20\textheight]{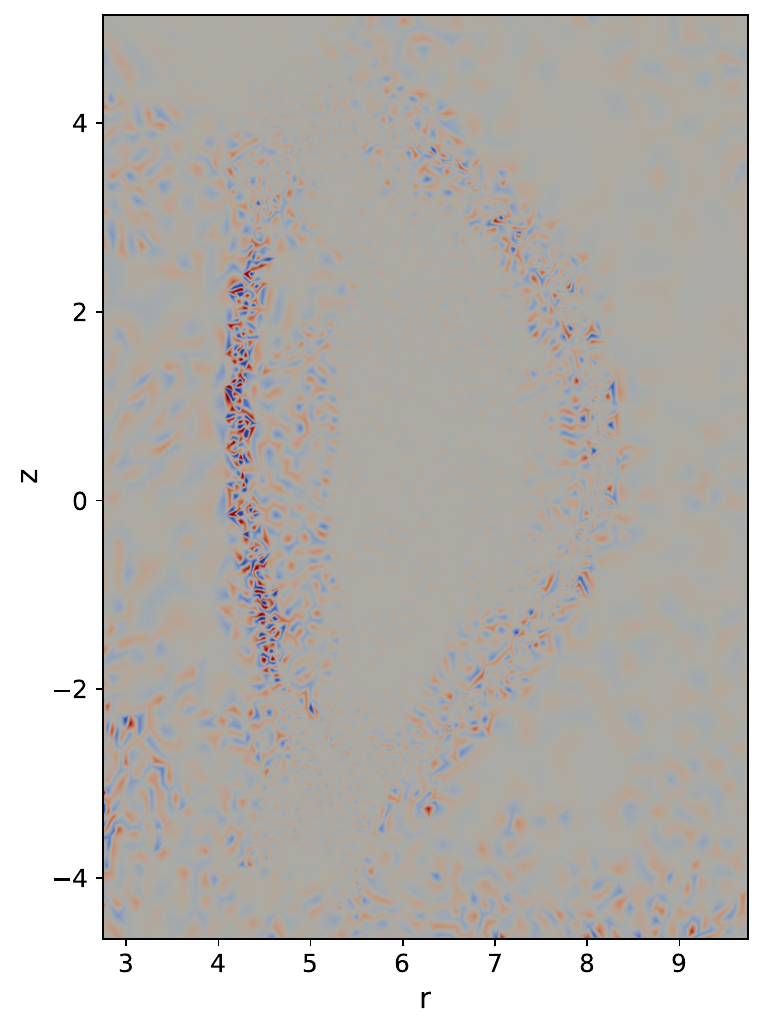}
    \caption{w/ refinement}
  \end{subfigure}%
  \begin{subfigure}[t]{0.1\textwidth}
    \centering
    \includegraphics[height=0.20\textheight]{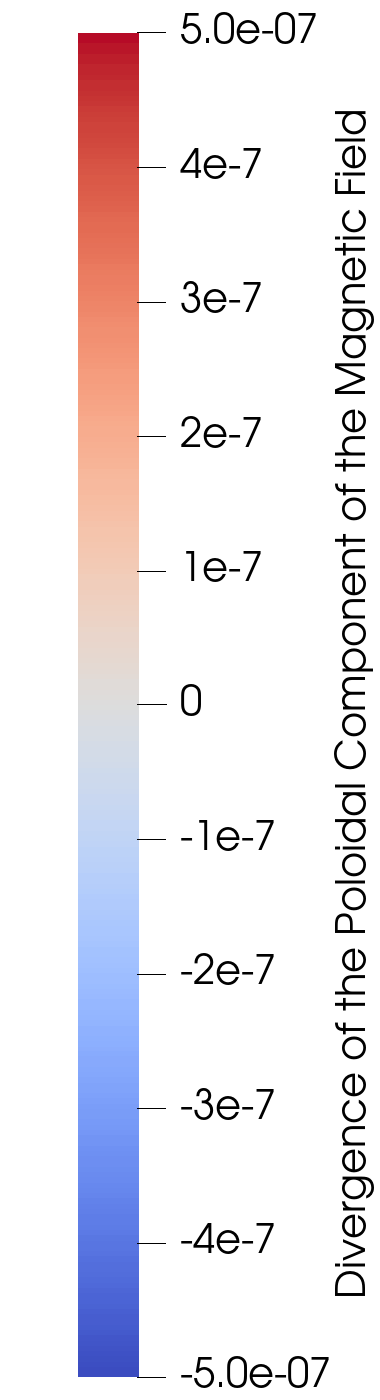}
  \end{subfigure}
  \caption{Comparison of the divergence of the poloidal component of the magnetic field ($D_b$) {of Taylor equilibrium} from projection path~\eqref{eq:proj_path_A}.
    (a) original,
    (b) w/ perturbations,
    (c) w/ refinement.}
  \label{fig:div_B_pol_per_refine}
\end{figure}


\clearpage
\bibliography{references.bib}

\end{document}